\documentclass[reqno]{amsart}
\usepackage{amsmath, amsthm, amssymb, amstext}

\usepackage{hyperref,xcolor}
\hypersetup{pdfborder={0 0 0},colorlinks}
\usepackage{enumitem}
\allowdisplaybreaks
\usepackage{graphicx}
\usepackage{caption}
\usepackage{mathtools}
\usepackage{tikz}

\newtheorem{theorem}{Theorem}
\newtheorem{remark}[theorem]{Remark}
\newtheorem{lemma}[theorem]{Lemma}
\newtheorem{proposition}[theorem]{Proposition}
\newtheorem{corollary}[theorem]{Corollary}
\newtheorem{definition}[theorem]{Definition}
\newtheorem{example}[theorem]{Example}

\DeclareMathOperator*{\loc}{loc}
\DeclareMathOperator*{\dx}{d\textit{x}}
\DeclareMathOperator*{\dy}{d\textit{y}}
\DeclareMathOperator*{\ds}{d\textit{s}}

\newcommand{\N}{\mathbb{N}}
\newcommand{\R}{\mathbb{R}}
\newcommand*\diff{\mathop{}\!\mathrm{d}}
\newcommand{\hlog}{\ensuremath{\mathcal{H}_{\log}}}
\newcommand{\Lp}[1]{L^{#1}(\Omega)}
\newcommand{\Lprand}[1]{L^{#1}(\partial\Omega)}
\newcommand{\Wp}[1]{W^{1,#1}(\Omega)}

\newcommand{\Wpzero}[1]{W^{1,#1}_0(\Omega)}
\newcommand{\eps}{\varepsilon}
\newcommand{\ph}{\varphi}
\newcommand{\into}{\int_{\Omega}}
\newcommand{\weak}{\rightharpoonup}

\newcommand{\close}{\overline{\Omega}}
\renewcommand{\l}{\left}
\renewcommand{\r}{\right}
\newcommand{\WHlog}{W^{1, \hlog}(\Omega)}
\newcommand{\WHlogzero}{W^{1, \hlog}_0(\Omega)}
\newcommand{\LHlog}{L^{\hlog}(\Omega)}

\def\abs#1{\left|{#1}\right|}
\def\norm#1{\left\|#1\right\|}
\def\normHlog#1{\left\|#1\right\|_{\hlog}}
\def\modHlog#1{\ensuremath{\varrho_{\hlog} \left(#1\right)}}
\def\normoneHlog#1{\left\|#1\right\|_{1,\hlog}}
\def\modoneHlog#1{\ensuremath{\varrho_{1,\hlog} \left(#1\right)}}
\def\normoneHlogzero#1{\left\|#1\right\|_{1,\hlog,0}}
\def\modoneHlogzero#1{\ensuremath{\varrho_{1,\hlog,0} \left(#1\right)}}

\numberwithin{theorem}{section}
\numberwithin{equation}{section}

\title[On logarithmic double phase problems]{On logarithmic double phase problems}

\author[R. Arora]{Rakesh Arora}
\address[R. Arora]{ Department of Mathematical Sciences, Indian Institute of Technology Varanasi (IIT-BHU), Uttar Pradesh 221005, India}
\email{rakesh.mat@iitbhu.ac.in, arora.npde@gmail.com}

\author[\'{A}. Crespo-Blanco]{\'{A}ngel Crespo-Blanco}
\address[\'{A}. Crespo-Blanco]{Technische Universit\"{a}t Berlin, Institut f\"{u}r Mathematik, Stra\ss e des 17.\,Juni 136, 10623 Berlin, Germany}
\email{crespo@math.tu-berlin.de}

\author[P. Winkert]{Patrick Winkert}
\address[P. Winkert]{Technische Universit\"{a}t Berlin, Institut f\"{u}r Mathematik, Stra\ss e des 17.\,Juni 136, 10623 Berlin, Germany}
\email{winkert@math.tu-berlin.de}


\begin{document}
\begin{abstract}
	In this paper we introduce a new logarithmic double phase type operator of the form
	\begin{align*}
		\begin{aligned}
			\mathcal{G}u & :=- \operatorname{div}\left( \abs{ \nabla u }^{p(x) - 2} \nabla u \right.  \\
			& \left. \qquad\qquad\quad + \mu(x) \left[ \log ( e + \abs{\nabla u} ) + \frac{ \abs{\nabla u} }{q(x) (e + \abs{\nabla u})} \right] \abs{ \nabla u }^{q(x) - 2} \nabla u  \right),
		\end{aligned}
	\end{align*}
	whose energy functional is given by
	\begin{align*}
		u \mapsto \into \left(  \frac{ \abs{\nabla u}^{p(x)} }{p(x)} + \mu(x) \frac{ \abs{\nabla u}^{q(x)} }{q(x)} \log (e + \abs{\nabla u}) \right)  \dx,
	\end{align*}
	where $\Omega \subseteq \R^N$, $N \geq 2$, is a bounded domain with Lipschitz boundary $\partial \Omega$, $p,q \in C(\close)$ with $1<p(x) \leq q(x)$ for all $x \in \close$ and $0\leq\mu(\cdot) \in \Lp{1}$. First, we prove that the logarithmic Musielak-Orlicz Sobolev spaces $\WHlog$ and $\WHlogzero$ with $\hlog (x,t) = t^{p(x)} + \mu(x) t^{q(x)} \log (e + t)$ for $(x,t)\in \close \times [0,\infty)$ are separable, reflexive Banach spaces and $\WHlogzero$ can be equipped with the equivalent norm
	\begin{equation*}
		\inf \left\lbrace \lambda > 0 \colon \into \left[ \abs{ \frac{\nabla u}{\lambda} }^{p(x)}
		+\mu(x) \abs{\frac{\nabla u}{\lambda}} ^{q(x)} \log \left( e + \frac{ \abs{\nabla u} }{\lambda} \right) \right]\dx \leq 1\right\rbrace.
	\end{equation*}
	We also prove several embedding results for these spaces and the closedness of $\WHlog$ and $\WHlogzero$ under truncations. In addition we show the density of smooth functions in $\WHlog$ even in the case of an unbounded domain by supposing Nekvinda's decay condition on $p(\cdot)$. The second part is devoted to the properties of the operator and it turns out that it is bounded, continuous, strictly monotone, of type \textnormal{(S$_+$)}, coercive and a homeomorphism. Also, the related energy functional is of class $C^1$. As a result of independent interest we also present a new version of Young's inequality for the product of a power-law and a logarithm. In the last part of this work we consider equations of the form
	\begin{equation*}
			\mathcal{G}u = f(x,u) \quad \text{in } \Omega,
			\quad u  = 0         \quad \text{on } \partial\Omega
	\end{equation*}
	with superlinear right-hand sides. We prove multiplicity results for this type of equation, in particular about sign-changing solutions, by making use of a suitable variation of the corresponding Nehari manifold together with the quantitative deformation lemma and the Poincar\'e-Miranda existence theorem.
\end{abstract}

\maketitle
\newpage
\section{Introduction}

In recent years double phase problems have received a lot of attention from the mathematical community. These problems typically involve a functional with the shape
\begin{align}\label{double-phase-functional-without-log}
	u \mapsto \into \left(  \frac{ \abs{\nabla u}^{p(x)} }{p(x)} + \mu(x) \frac{ \abs{\nabla u}^{q(x)} }{q(x)} \right)  \dx
\end{align}
or, alternatively, a differential operator of the form
\begin{align}\label{double-phase-operator-without-log}
	- \operatorname{div} \left( \abs{ \nabla u }^{p(x) - 2} \nabla u + \mu(x) \abs{ \nabla u }^{q(x) - 2} \nabla u  \right) .
\end{align}
Naturally, one can see $p,q$ as functions or limit the study to the constant exponents case. These problems are called double phase problems because of their nonuniform ellipticity, with two regions of different behavior. In the set $\{ x \in \Omega \colon \mu(x)\geq \eps > 0\}$ for any fixed $\eps > 0$, the ellipticity in the gradient of the integrand is of order $q(x)$, while in the set $\{ x \in \Omega \colon \mu(x) = 0\}$ that ellipticity is of order $p(x)$.

Let $p$ and $q$ be constants, the double phase energy functional
\begin{align}\label{integral_minimizer}
	u \mapsto \int_\Omega \left(|\nabla  u|^p+\mu(x)|\nabla  u|^q\right)\dx
\end{align}
appeared for the first time in a work of Zhikov \cite{Zhikov-1986} in order to describe models for strongly anisotropic materials in the context of homogenization and elasticity theory, see also Zhikov \cite{Zhikov-1995,Zhikov-2011}. Indeed, in elasticity theory, the modulating coefficient $\mu(\cdot)$ dictates the geometry of composites made of two different materials with distinct power hardening exponents $q(\cdot)$ and $p(\cdot)$, see the mentioned works of Zhikov. We also point out that there are several other applications in physics and engineering of double phase differential operators and related energy functionals, see the works of Bahrouni--R\u{a}dulescu--Repov\v{s} \cite{Bahrouni-Radulescu-Repovs-2019} on transonic flows, Benci--D'Avenia--Fortunato-Pisani \cite{Benci-DAvenia-Fortunato-Pisani-2000} on quantum physics and  Cherfils--Il'yasov \cite{Cherfils-Ilyasov-2005} on reaction diffusion systems.

In recent years functionals of the shape \eqref{integral_minimizer} have been treated in many papers concerning regularity, in particular of local minimizers (also for nonstandard growth). We refer to the works of Baroni--Colombo--Mingione \cite{Baroni-Colombo-Mingione-2015,Baroni-Colombo-Mingione-2018}, Baroni--Kuusi--Mingione \cite{Baroni-Kuusi-Mingione-2015}, Byun--Oh \cite{Byun-Oh-2020}, Byun--Ok--Song \cite{Byun-Ok-Song-2022},  Colombo--Mingione \cite{Colombo-Mingione-2015a,Colombo-Mingione-2015b}, De Filippis--Palatucci \cite{De-Filippis-Palatucci-2019}, Harjulehto--H\"{a}st\"{o}--Toivanen \cite{Harjulehto-Hasto-Toivanen-2017}, Ok \cite{Ok-2018,Ok-2020}, Ra\-gusa--Tachikawa \cite{Ragusa-Tachikawa-2020, Ragusa-Tachikawa-2024}, Tachikawa \cite{Tachikawa-2024} and the references therein. Furthermore, nonuniformly elliptic variational problems and nonautonomous functionals have been studied in the papers of Beck--Mingione \cite{Beck-Mingione-2020},
De Filippis--Mingione \cite{DeFilippis-Mingione-2021-2,DeFilippis-Mingione-2020-2} and H\"{a}st\"{o}--Ok \cite{Hasto-Ok-2019}. We point out that \eqref{integral_minimizer} also belongs to the class of the integral functionals with nonstandard growth condition as a special case of the outstanding papers of Marcellini \cite{Marcellini-1991,Marcellini-1989}, see also the recent papers by Cupini--Marcellini--Mascolo \cite{Cupini-Marcellini-Mascolo-2023} and Marcellini \cite{Marcellini-2023} with $u$-dependence.

However, such works are limited to consider only a power-law type of growth in each of the addends. If one wants to consider other types of growth, the first idea that comes up naturally is to modify power-laws with a logarithm. For this reason, in this paper we consider logarithmic type functionals of the form
\begin{align}\label{log-energy-functional}
	I(u)=\into \left(  \frac{ \abs{\nabla u}^{p(x)} }{p(x)} + \mu(x) \frac{ \abs{\nabla u}^{q(x)} }{q(x)} \log (e + \abs{\nabla u}) \right)  \dx,
\end{align}
and its associated differential operator
\begin{align}
	\label{log-operator}
	\begin{aligned}
		\mathcal{G}u & :=- \operatorname{div}\left( \abs{ \nabla u }^{p(x) - 2} \nabla u \right. \\
		& \left. \qquad\qquad\quad + \mu(x) \left[ \log ( e + \abs{\nabla u} ) + \frac{ \abs{\nabla u} }{q(x) (e + \abs{\nabla u})} \right] \abs{ \nabla u }^{q(x) - 2} \nabla u  \right),
	\end{aligned}
\end{align}
where $\Omega \subseteq \R^N$, $N \geq 2$, is a bounded domain with Lipschitz boundary $\partial \Omega$, $e$ stands for Euler's number, $p,q \in C(\close)$ with $1 < p(x) \leq q(x)$ for all $x \in \close$ and $0\leq \mu(\cdot) \in \Lp{1}$.

One work closely related to ours is Baroni--Colombo--Mingione \cite{Baroni-Colombo-Mingione-2016}, where they prove the local H\"older continuity of the gradient of local minimizers of the functional
\begin{align} \label{log-functional-Mingione1}
	w \mapsto \into \left[  \abs{D w}^p  + a(x) \abs{D w}^p \log (e + \abs{D w}) \right]  \dx
\end{align}
provided that $1 < p < \infty$ and $0 \leq a(\cdot) \in C^{0,\alpha} (\close)$. Note that when we take $p=q$ and constant, \eqref{log-energy-functional} and \eqref{log-functional-Mingione1} are the same functional up to a multiplicative constant. In another recent work of De Filippis--Mingione \cite{DeFilippis-Mingione-2023} the local H\"{o}lder continuity of the gradients of local minimizers of the functional
\begin{align}\label{log-functional-Mingione2}
	w \mapsto \into \big[|D w|\log(1+|D w|)+a(x)|D w|^q\big]\dx
\end{align}
has been shown provided $0 \leq a(\cdot)\in C^{0,\alpha}(\overline{\Omega})$ and $1<q<1+\frac{\alpha}{n}$ whereby $\Omega \subset \R^n$.
The functional \eqref{log-functional-Mingione2} has its origin in the functional with nearly linear growth given by
\begin{align}\label{log-functional-Mingione3}
	w \mapsto \into |D w|\log(1+|D w|)\dx,
\end{align}
which has been studied in Fuchs--Mingione \cite{Fuchs-Mingione-2000} and Marcellini--Papi \cite{Marcellini-Papi-2006}. Note that functionals of the form \eqref{log-functional-Mingione3} appear in the theory of plasticity with logarithmic hardening, see, for example, Seregin--Frehse \cite{Seregin-Frehse-1999} and the monograph of Fuchs--Seregin \cite{Fuchs-Seregin-2000} about variational methods for problems which have their origins in plasticity theory and generalized Newtonian fluids. Another closely related energy functional to \eqref{log-energy-functional} has been considered by Marcellini in \cite{Marcellini-1991} which has the form
\begin{align*}
	w \mapsto \int_{\Omega} (1+|D w|^2)^\frac{p}{2} \log (1 + |D w|) \dx.
\end{align*}
The above functional can also be generalized to a energy functional related to \eqref{log-energy-functional} and a differential operator satisfying  $p,q +\varepsilon$--growth conditions for every a-priori fixed $\varepsilon>0$. Such type of growth conditions was first introduced by Marcellini in \cite{Marcellini-1991} and later on, great attention has been paid to the study of several aspects of elliptic equations involving $p,q$--growth conditions, see the more recent works by Cupini--Marcellini--Mascolo \cite{Cupini-Marcellini-Mascolo-2023} and Marcellini \cite{Marcellini-2023}. However, a detailed study of differential operators satisfying $p,q$--growth conditions, which includes double phase operators, logarithmic double phase operators, and anisotropic operators as specific cases, is far from complete.

To the best of our knowledge, our work is the first one dealing with such logarithmic operator given in \eqref{log-operator} and associated energy functional  \eqref{log-energy-functional} in a very general setting. Indeed, there are many innovations and novelties in this work which we want to summarize below. The first step in studying the operator and its energy functional is the finding of the right function space. For this purpose, we consider logarithmic Musielak-Orlicz Sobolev spaces $\WHlog$ and $\WHlogzero$ with the generalized weak $\Phi$-function $\hlog \colon \close \times [0,\infty) \to [0,\infty)$ given by
\begin{align*}
	\hlog (x,t) = t^{p(x)} + \mu(x) t^{q(x)} \log (e + t).
\end{align*}
We are able to prove that these spaces are separable, reflexive Banach spaces and $\WHlogzero$ can be equipped with the equivalent norm
\begin{equation*}
	\inf \left\lbrace \lambda > 0 \colon \into \left[ \abs{ \frac{\nabla u}{\lambda} }^{p(x)}
	+\mu(x) \abs{\frac{\nabla u}{\lambda}} ^{q(x)} \log \left( e + \frac{ \abs{\nabla u} }{\lambda} \right) \right]\dx \leq 1\right\rbrace.
\end{equation*}
Such norm will be later useful in our existence results for corresponding logarithmic double phase equations. In addition, we prove several embedding results into variable exponent Lebesgue spaces and the closedness of $\WHlog$ and $\WHlogzero$ under truncations. Moreover, we show the density of smooth functions. To be more precise, under the assumptions that $p,q \colon \close \to [1,\infty)$ and $\mu \colon \close \to [0,\infty)$ are H\"older continuous and
\begin{align}\label{double-phase-condition}
	\left( \frac{q}{p} \right)_+ < 1 + \frac{\gamma}{N},
\end{align}
where $\gamma$ is the H\"older exponent of $\mu$, we obtain that $C^\infty (\Omega) \cap \WHlog$ is dense in $\WHlog$. As an result of independent interest, we extend this assertion to the case of unbounded domains under the additional hypothesis that $p$ satisfies Nekvinda's decay condition and $p$ and $q$ are bounded. This condition was first introduced by Nekvinda in the article \cite{Nekvinda-2004} and says that a measurable function $r \colon \Omega \to [1,\infty]$ satisfies Nekvinda's decay condition if there exists $r_\infty \in [1 , \infty]$ and $c \in (0,1)$ such that
\begin{align*}
	\into c^{\frac{1}{ \abs{ \frac{1}{r(x)} - \frac{1}{r_\infty} } }} \dx < \infty.
\end{align*}

Let us come back to the inequality \eqref{double-phase-condition} and suppose the exponents $p$ and $q$ are constants such that $1<p<q<N$ and $\mu$ to be Lipschitz (i.e.\,$\gamma=1$). Then \eqref{double-phase-condition} reads as
\begin{align}\label{double-phase-condition2}
	\frac{q}{p} < 1 + \frac{1}{N}.
\end{align}
Condition \eqref{double-phase-condition2} has been used for regularity results of local minimizers to related energy functionals given by \eqref{double-phase-functional-without-log} (see the above mentioned works) and to guarantee the density of smooth functions in related Musielak-Orlicz Sobolev spaces, see here the work of Colasuonno--Squassina \cite[Proposition 6.5]{Colasuonno-Squassina-2016}. For existence results for double phase problems, condition \eqref{double-phase-condition2} was crucial to have the Poincar\'e inequality, see again \cite[Proposition 2.18 (iv)]{Colasuonno-Squassina-2016}. Later, Crespo-Blanco--Gasi\'{n}ski--Harjulehto--Winkert \cite{Crespo-Blanco-Gasinski-Harjulehto-Winkert-2022} were able to prove the existence of the Poincar\'e inequality under the weaker assumption $q<p^*$. Note that \eqref{double-phase-condition2} implies $q<p^*$. So in the existence theory, not using density results, the condition \eqref{double-phase-condition2} is not needed anymore for most of the treatments. However, we believe that for a global regularity theory, which does not exist so far, condition \eqref{double-phase-condition2} is the main assumption in the double phase setting.

In the second part of this paper we are interested in the properties of the logarithmic double phase operator $\mathcal{G}$ given in \eqref{log-operator} and its corresponding energy functional $I$ in \eqref{log-energy-functional}. We prove that the operator is bounded, continuous, strictly monotone, of type \textnormal{(S$_+$)}, coercive and a homeomorphism. Moreover, the functional $I$ is of class $C^1$. As a result of independent interest we also present a new version of Young's inequality for the product of a power-law and a logarithm, which says that for $s,t \geq 0$ and $r > 1$ it holds
\begin{align*}
	s t^{r-1} \left[ \log (e + t ) + \frac{ t }{r (e + t)} \right]
	\leq \frac{s^r}{r} \log (e + s )
	+ t^r \left[ \frac{r - 1}{r} \log (e + t ) + \frac{ t }{r (e + t)} \right] .
\end{align*}
This inequality is essential for our proof that the operator fulfills the \textnormal{(S$_+$)}-property.

Finally, in the last part of the paper, we are interested in existence and multiplicity results for the equation
\begin{align}
	\label{Eq:Problem}
	\mathcal{G}u  = f(x,u) \quad \text{in } \Omega, \quad u= 0 \quad \text{on } \partial\Omega,
\end{align}
where $\mathcal{G}$ is given in \eqref{log-operator} and $f\colon\Omega\times \R\to\R$ is a Carath\'{e}odory function with subcritical growth which satisfies appropriate conditions, see \eqref{Hf}, \eqref{Asf:QuotientMono}, \eqref{Asf:GrowthZeroAlt} and \eqref{Asf:Nonnegative} for the details. We prove the existence of three nontrivial weak solutions of problem \eqref{Eq:Problem} including determining their sign. More precisely, one solution is positive, one is negative and the third one has changing sign. The existence of the sign-changing solution is the main difficult part in our existence section, the idea is to the use an appropriate variation $\mathcal{N}_0$ of the corresponding Nehari manifold of problem \eqref{Eq:Problem} given by
\begin{align*}
	\mathcal{N} = \left\{ u \in \WHlogzero \colon \langle \ph'(u) , u \rangle = 0, \; u \neq 0 \right\},
\end{align*}
where $\ph$ is the energy functional corresponding to \eqref{Eq:Problem}. The definition of $\mathcal{N}$ is motivated by the works of Nehari \cite{Nehari-1960,Nehari-1961} and the set $\mathcal{N}_0$  consists of all functions of $\WHlogzero$ such that the positive and negative parts are also in $\mathcal{N}$. The idea in the proof is, among other things, a suitable combination of the quantitative deformation lemma and the Poincar\'e-Miranda existence theorem. Such treatment has been applied to double phase problems without logarithmic term and by using the Brouwer degree instead of the Poincar\'e-Miranda existence theorem by the works of Liu--Dai \cite{Liu-Dai-2018} for $p,q$ constants and Crespo-Blanco--Winkert \cite{Crespo-Blanco-Winkert-2022} for the operator given in \eqref{double-phase-operator-without-log} with associated functional \eqref{double-phase-functional-without-log}. Note that the appearance of the logarithmic term in our operator makes the treatment much more complicated than in the works \cite{Liu-Dai-2018} and \cite{Crespo-Blanco-Winkert-2022}.

As mentioned above, to the best of our knowledge, there exists no other work dealing with the logarithmic double phase operator given in \eqref{log-operator}. However, some papers deal with logarithmic terms on the right-hand side for Schr\"odinger equations or $p$-Laplace problems. In 2009, Montenegro--de Queiroz \cite{Montenegro-deQueiroz} studied nonlinear elliptic problem
\begin{align}\label{reference-1}
	-\Delta u=\chi_{u>0}(\log (u)+\lambda f(x,u))\quad \text{in } \Omega, \quad u= 0 \quad \text{on } \partial\Omega,
\end{align}
where $f\colon \Omega \times [0,\infty) \to [0,\infty)$ is nondecreasing, sublinear and $f_u$ is continuous. The authors show that \eqref{reference-1} has a maximal solution $u_\lambda \geq 0$ of type $C^{1,\gamma}(\close)$. Logarithmic Schr\"odinger equations of the shape
\begin{align}\label{reference-2}
	-\Delta u+V(x)u=Q(x)u \log(u^2)\quad \text{in }\mathbb{R}^N
\end{align}
have been studied by Squassina--Szulkin \cite{Squassina-Szulkin-2015} proving that \eqref{reference-2} has infinitely many solutions, where $V$ and $Q$ are $1$-periodic functions of the variables $x_1,\ldots,x_N$ and $Q\in C^1(\mathbb{R}^N)$. Further results for logarithmic Schr\"odinger equations can be found in the works of Alves--de Morais Filho \cite{Alves-deMoraisFilho-2018}, Alves--Ji \cite{Alves-Ji-2020} and Shuai \cite{Shuai-2019}, see also Gao--Jiang--Liu-Wei \cite{Gao-Jiang-Liu-Wei-2023} for logarithmic Kirchhoff type equations and the recent work of Alves--da Silva \cite{Alves-daSilva-2023} about logarithmic Schr\"{o}dinger equations on exterior domains. We also refer to the papers of Alves--Moussaoui--Tavares \cite{Alves-Moussaoui-Tavares-2019} for singular systems with logarithmic right-hand sides driven by the $\Delta_{p(\cdot)}$-Laplacian and Shuai \cite{Shuai-2023} for a Laplace equation with right-hand side $a(x)u\log(|u|)$ with weight function $a$ which may change sign.

To finish this introduction, we would like to mention some famous works in the direction of double phase problems (without logarithmic term) appearing in the last years based on different methods and techniques. We refer to the papers by Aberqi--Bennouna--Benslimane--Ragusa \cite{Aberqi-Bennouna-Benslimane-Ragusa-2023} (elliptic equations on complete compact Riemannian manifolds), Arora--Shmarev \cite{Arora-Shmarev-2023,Arora-Shmarev-2023-b} (parabolic double-phase problems), Bai--Papageorgiou--Zeng \cite{Bai-Papageorgiou-Zeng-2023} (parametric singular problems), Clop--Gentile--Passarelli \cite{Clop-Gentile-Passarelli-2023} (higher differentiability under sub-quadratic growth conditions), Colasuonno--Squa\-ssina \cite{Colasuonno-Squassina-2016} (double phase eigenvalue problems), Fang--R\u{a}dulescu--Zhang  \cite{Fang-Radulescu-Zhang-2024} (equivalence of weak and viscosity solutions), Farkas--Winkert \cite{Farkas-Winkert-2021} (Finsler double phase problems), Gasi\'{n}ski--Papageorgiou \cite{Gasinski-Papageorgiou-2019} (locally Lipschitz right-hand sides), Gasi\'{n}s\-ki--Winkert \cite{Gasinski-Winkert-2020} (convection problems), Ho--Winkert \cite{Ho-Winkert-2023} (new critical embedding results), Liu--Dai \cite{Liu-Dai-2018} (superlinear problems), Liu--Pucci \cite{Liu-Pucci-2023} (multiplicity results without the AR-condition), Papageorgiou--R\u{a}dulescu--Repov\v{s} \cite{Papageorgiou-Radulescu-Repovs-2019-a, Papageorgiou-Radulescu-Repovs-2020} (property of the spectrum and ground state solutions), Perera--Squassina \cite{Perera-Squassina-2018} (Morse theory for double phase problems), Zhang--R\u{a}dulescu \cite{Zhang-Radulescu-2018} and Shi--R\u{a}dulescu--Re\-pov\v{s}--Zhang \cite{Shi-Radulescu-Repovs-Zhang-2020} (double phase anisotropic variational problems with variable exponents), Zeng--Bai--Gasi\'{n}ski--Winkert \cite{Zeng-Bai-Gasinski-Winkert-2020} (implicit obstacle double phase problems), Zeng--R\u{a}dulescu--Winkert \cite{Zeng-Radulescu-Winkert-2022} (implicit obstacle double phase problems with mixed boundary condition), see also the references therein.

The paper is organized as follows. Section \ref{musielak-orlicz-preliminaries} consists of an outline of the properties of variable exponent spaces, Musielak-Orlicz spaces, their associated Sobolev spaces and other mathematical tools used later in the text. These tools include some inequalities, the quantitative deformation lemma and the Poincar\'e-Miranda existence theorem. In Section \ref{functional-space} we introduce the functional space that will be used throughout the rest of the paper and we give its main characteristics. In Section \ref{energy-functional-operator} we provide some strong properties of the differential operator of problem \eqref{Eq:Problem}. These are essential when applying many techniques used to study nonlinear PDEs, including our treatment. After that, in Section \ref{constant-sign-solutions} we prove the existence of two nontrivial solutions of our problem, one of positive sign and one of negative sign. On top of this, in Section \ref{sign-changing-solution} we show the existence of a third nontrivial solution with changing sign using the Nehari manifold technique and variational arguments. As a closing result, in Section \ref{nodal-domains} we give information on the nodal domains of this sign-changing solution.

\section{Musielak-Orlicz spaces and preliminaries}\label{musielak-orlicz-preliminaries}

Some of the natural ingredients to study this kind of operator are the variable exponent Lebesgue and Sobolev spaces. For this reason, we start this section with a small summary of their properties. Later in this section we will also work with the theory of Musielak-Orlicz spaces in order to build an appropriate space for this operator.

Let $1 \leq r \leq \infty$, we denote by $\Lp{r}$ the standard Lebesgue space equipped with the norm $\|\cdot\|_r$ and by $\Wp{r}$ and $\Wpzero{r}$ the typical Sobolev spaces fitted with the norm $\|\cdot\|_{1,r}$ and, for the case $1 \leq r < \infty$, also the norm $\|\cdot\|_{1,r,0}$.

Let us also denote the positive and negative part as follows. Let $t \in \R$, then $t^\pm = \max\{ \pm t , 0 \}$, i.e.\,$t = t^+ - t^-$ and $\abs{t} = t^+ + t^-$. For any function $u \colon \Omega \to \R$, we denote $u^\pm (x) = \left[ u(x) \right] ^\pm$ for all $x \in \Omega$.

For the variable exponent case, we need to introduce some common notation. Let $r \in C(\close)$, we define $r_- = \min_{x \in \close} r(x)$ and $r_+ = \max_{x \in \close} r(x)$ and also the space
\begin{align*}
	C_+ (\Omega) = \{ r \in C(\close) \colon 1 < r_- \}.
\end{align*}
Let $r \in C_+ (\Omega)$ and let $M(\Omega) = \{ u \colon \Omega \to \R \colon u \text{ is measurable}\}$, we denote by $\Lp{r(\cdot)}$ the Lebesgue space with variable exponent given by
\begin{align*}
	\Lp{r(\cdot)} = \left\lbrace u \in M(\Omega) \colon \varrho_{r(\cdot)} (u) < \infty \right\rbrace,
\end{align*}
where the modular associated with $r$ is
\begin{align*}
	\varrho_{r(\cdot)} (u) = \into \abs{u}^{r(x)} \dx
\end{align*}
and it is equipped with its associated Luxemburg norm
\begin{align*}
	\norm{u}_{r(\cdot)} = \inf \left\lbrace \lambda > 0 \colon \varrho_{r(\cdot)} \left( \frac{u}{\lambda} \right)  \leq 1 \right\rbrace .
\end{align*}

These spaces have been investigated in many works, and as a result we nowadays have a comprehensive theory. One can find most of the relevant results in the book by Diening--Harjulehto--H\"{a}st\"{o}--R$\mathring{\text{u}}$\v{z}i\v{c}ka  \cite{Diening-Harjulehto-Hasto-Ruzicka-2011}. We know that $\Lp{r(\cdot)}$ is a separable and reflexive Banach space and its norm is uniformly convex. We also know that $\left[ \Lp{r(\cdot)} \right] ^*=\Lp{r'(\cdot)}$, where $r' \in C_+(\close)$ is the conjugate variable exponent of $r$ and is given by $r'(x) = r(x) / [r(x) - 1]$ for all $x \in \close$, see for example \cite[Lemma 3.2.20]{Diening-Harjulehto-Hasto-Ruzicka-2011}. In these spaces we also have a weaker version of H\"older's inequality, like the one in \cite[Lemma 3.2.20]{Diening-Harjulehto-Hasto-Ruzicka-2011}, which states that
\begin{align*}
	\into |uv| \dx \leq \left[\frac{1}{r_-}+\frac{1}{r'_-}\right] \|u\|_{r(\cdot)}\|v\|_{r'(\cdot)} \leq 2 \|u\|_{r(\cdot)}\|v\|_{r'(\cdot)} \quad \text{for all } u,v\in \Lp{r(\cdot)}.
\end{align*}
Additionally, if $r_1, r_2\in C_+(\close)$ and $r_1(x) \leq r_2(x)$ for all $x\in \close$, it is possible to embed continuously one space in the other like in \cite[Theorem 3.3.1]{Diening-Harjulehto-Hasto-Ruzicka-2011}, meaning $\Lp{r_2(\cdot)} \hookrightarrow \Lp{r_1(\cdot)}$. Finally, the norm and its modular are strongly related as one can see in the following result, it can be found in the paper of Fan--Zhao \cite[Theorems 1.2 and 1.3]{Fan-Zhao-2001}.

\begin{proposition}
	\label{Prop:ModularNormVarExp}
	Let $r\in C_+(\close)$, $\lambda>0$, and $u\in \Lp{r(\cdot)}$, then
	\begin{enumerate}
		\item[\textnormal{(i)}]
			$\|u\|_{r(\cdot)}=\lambda$  if and only if $ \varrho_{r(\cdot)}\l(\frac{u}{\lambda}\r)=1$ with $u \neq 0$;
		\item[\textnormal{(ii)}]
			$\|u\|_{r(\cdot)}<1$ (resp. $=1$, $>1$)  if and only if $\varrho_{r(\cdot)}(u)<1$ (resp. $=1$, $>1$);
		\item[\textnormal{(iii)}]
			if $\|u\|_{r(\cdot)}<1$, then $\|u\|_{r(\cdot)}^{r_+} \leq \varrho_{r(\cdot)}(u) \leq \|u\|_{r(\cdot)}^{r_-}$;
		\item[\textnormal{(iv)}]
			if $\|u\|_{r(\cdot)}>1$, then $\|u\|_{r(\cdot)}^{r_-} \leq \varrho_{r(\cdot)}(u) \leq \|u\|_{r(\cdot)}^{r_+}$;
		\item[\textnormal{(v)}]
			$\|u\|_{r(\cdot)} \to 0$  if and only if  $\varrho_{r(\cdot)}(u)\to 0$;
		\item[\textnormal{(vi)}]
			$\|u\|_{r(\cdot)}\to +\infty$  if and only if  $\varrho_{r(\cdot)}(u)\to +\infty$.
	\end{enumerate}
\end{proposition}

For our purposes we will also need the associated Sobolev spaces to the variable exponent Lebesgue spaces. These are also treated in the book by Diening--Harjulehto--H\"{a}st\"{o}--R$\mathring{\text{u}}$\v{z}i\v{c}ka  \cite{Diening-Harjulehto-Hasto-Ruzicka-2011}. Let $r \in C_+(\close)$, the Sobolev space $\Wp{r(\cdot)}$ is given by
\begin{align*}
	\Wp{r(\cdot)}=\l\{ u \in \Lp{r(\cdot)} \colon |\nabla u| \in \Lp{r(\cdot)}\r\},
\end{align*}
on it we can define the modular
\begin{align*}
	\varrho_{1, r(\cdot)} (u) = \varrho_{ r(\cdot)} (u) + \varrho_{ r(\cdot)} ( \nabla u ),
\end{align*}
where $\varrho_{ r(\cdot)} ( \nabla u ) = \varrho_{ r(\cdot)} ( \abs{\nabla u} )$, and it is equipped with its associated Luxemburg norm
\begin{align*}
	\norm{u}_{1, r(\cdot)} = \inf \left\lbrace \lambda > 0 \colon \varrho_{1, r(\cdot)} \left( \frac{u}{\lambda} \right)  \leq 1 \right\rbrace .
\end{align*}
Furthermore, similarly to the standard case, we denote
\begin{align*}
	\Wpzero{r(\cdot)}= \overline{C^\infty _0(\Omega)}^{\|\cdot\|_{1,r(\cdot)}}.
\end{align*}
The spaces $\Wp{r(\cdot)}$ and $\Wpzero{r(\cdot)}$ are both separable and reflexive Banach spaces and the norm $\|\cdot\|_{1,r}$ is uniformly convex.

A Poincar\'{e} inequality of the norms holds in the space $\Wpzero{r(\cdot)}$. One way to see this is the paper by Fan--Shen--Zhao \cite[Theorem 1.3]{Fan-Shen-Zhao-2001}, together with the standard way to derive the Poincar\'e inequality from the compact embedding, see for example the paper by Crespo-Blanco--Gasi\'{n}ski--Harjulehto--Winkert \cite[Proposition 2.18 (ii)]{Crespo-Blanco-Gasinski-Harjulehto-Winkert-2022}.
\begin{proposition}
	Let $r \in C_+(\close)$, then there exists $c_0>0$ such that
	\begin{align*}
		\|u\|_{r(\cdot)} \leq c_0 \|\nabla u\|_{r(\cdot)}
		\quad\text{for all } u \in \Wpzero{r(\cdot)}.
	\end{align*}
\end{proposition}
Thus, we can define the equivalent norm on $\Wpzero{r(\cdot)}$ given by $\|u\|_{1,r(\cdot),0}=\|\nabla u\|_{r(\cdot)}$. This norm is also uniformly convex.

Alternatively, assuming an additional monotonicity condition on $r$,  we also have a Poincar\'{e} inequality for the modular, see the paper by Fan--Zhang--Zhao \cite[Theorem 3.3]{Fan-Zhang-Zhao-2005}.
\begin{proposition}
	\label{Prop:PoincareModular}
	Let $r \in C_+(\close)$ be such that there exists a vector $l \in \R^N \setminus \{0\}$ with the property that for all $x \in \Omega$ the function
	\begin{align*}
		h_x (t) = r(x + tl) \quad \text{with } t \in I_x = \{ t \in \R \colon x + tl \in \Omega\}
	\end{align*}
	is monotone. Then there exits a constant $C>0$ such that
	\begin{align*}
		\varrho_{r(\cdot)} (u) \leq C \varrho_{r(\cdot)} (\nabla u)
		\quad \text{for all } u \in \Wp{r(\cdot)},
	\end{align*}
	where $\varrho_{r(\cdot)} (\nabla u) = \varrho_{r(\cdot)} ( | \nabla u | )$.
\end{proposition}

For $r \in C_+ (\close)$ we introduce the critical Sobolev variable exponents $r^*$ and $r_*$ with the following expression
\begin{align*}
	r^*(x) & =
	\begin{cases}
		\frac{N r(x)}{N - r(x) } & \text{if } r(x) < N\\
		+\infty & \text{if } r(x) \geq N
	\end{cases}
	, \quad\text{for all } x \in \close, \\[1ex]
	r_*(x) & =
	\begin{cases}
		\frac{(N -1) r(x)}{N - r(x) } & \text{if } r(x) < N \\
		+\infty & \text{if } r(x) \geq N
	\end{cases}
	, \quad\text{for all } x \in \close.
\end{align*}
Note that for any $r \in C(\close)$ it holds that $(r^*)_-=(r_-)^*$, so we will denote it simply by $r^*_-$.

On the other hand, the space $C^{0, \frac{1}{|\log t|}}(\close)$ is the set of all functions $h\colon \close \to \R$ such that are log-H\"{o}lder continuous, i.e.\,there exists $C>0$ such that
\begin{align*}
	|h(x)-h(y)| \leq \frac{C}{|\log |x-y||}\quad\text{for all } x,y\in \close \text{ with } |x-y|<\frac{1}{2}.
\end{align*}

In the variable exponent setting we also have embeddings analogous to the Sobolev embeddings of the constant exponent case. The next two results can be found in Crespo-Blanco--Gasi\'{n}ski--Harjulehto--Winkert \cite[Propositions 2.1 and 2.2]{Crespo-Blanco-Gasinski-Harjulehto-Winkert-2022} or Ho--Kim--Winkert--Zhang  \cite[Proposition 2.4 and 2.5]{Ho-Kim-Winkert-Zhang-2022}.

\begin{proposition}
	\label{Prop:classicalembedd}
	Let $r\in C^{0, \frac{1}{|\log t|}}(\close) \cap C_+(\close)$ and let $s\in C(\close)$ be such that $1\leq  s(x)\leq r^*(x)$ for all $x\in\close$. Then, we have the continuous embedding
	\begin{align*}
		W^{1,r(\cdot)}(\Omega) \hookrightarrow L^{s(\cdot) }(\Omega ).
	\end{align*}
	If $r\in C_+(\close)$, $s\in C(\close)$ and $1\leq s(x)< r^*(x)$ for all
	$x\in\overline{\Omega}$, then this embedding is compact.
\end{proposition}

\begin{proposition}
	\label{Prop:classicalembedd:boundary}
	Suppose that $r\in C_+(\close)\cap W^{1,\gamma}(\Omega)$ for some $\gamma>N$. Let $s\in C(\close)$ be such that $1\leq  s(x)\leq r_*(x)$ for all $x\in\close$. Then, we have the continuous embedding
	\begin{align*}
		W^{1,r(\cdot)}(\Omega)\hookrightarrow L^{s(\cdot) }(\partial \Omega).
	\end{align*}
	If $r\in C_+(\close)$, $s\in C(\close)$ and $1\leq s(x)< r_*(x)$ for all
	$x\in\overline{\Omega}$, then the embedding is compact.
\end{proposition}

For the purpose of introducing the functional space mentioned in the Introduction, we present now the main features of Musielak-Orlicz spaces. Almost all definitions and results from this part of the work are from the book by Harjulehto--H\"ast\"o \cite{Harjulehto-Hasto-2019}. We start with some special types of growth. For the rest of this section let us denote by $(A,\Sigma,\mu)$ a $\sigma$-finite, complete measure space with $\mu \not \equiv 0$, while $\Omega$ still denotes a bounded domain in $\R^N$ with $N \geq 2$ and Lipschitz boundary $\partial \Omega$.

\begin{definition}
	Let  $\varphi \colon A \times (0,+\infty) \to \R$. We say that
	\begin{enumerate}
		\item[\textnormal{(i)}]
			$\ph$ is almost increasing in the second variable if there exists $a \geq 1$ such that $\ph(x,s) \leq a \ph(x,t)$ for all $0 < s < t$ and for a.a.\,$x \in A$;
		\item[\textnormal{(ii)}]
			$\ph$ is almost decreasing in the second variable if there exists $a \geq 1$ such that $a \ph(x,s) \geq \ph(x,t)$ for all $0 < s < t$ and for a.a.\,$x \in A$.
	\end{enumerate}

	Let $\ph \colon A \times (0,+\infty) \to \R$ and $p,q>0$. We say that $\ph$ satisfies the property
	\begin{enumerate}[leftmargin=2cm]
		\item[\textnormal{(Inc)}$_p$]
			if $t^{-p}\ph(x,t)$ is increasing in the second variable;
		\item[\textnormal{(aInc)}$_p$]
			if $t^{-p}\ph(x,t)$ is almost increasing in the second variable;
		\item[\textnormal{(Dec)}$_q$]
			if $t^{-q}\ph(x,t)$ is decreasing in the second variable;
		\item[\textnormal{(aDec)}$_q$]
			if $t^{-q}\ph(x,t)$ is almost decreasing in the second variable.
	\end{enumerate}
	Without subindex, that is \textnormal{(Inc)}, \textnormal{(aInc)}, \textnormal{(Dec)} and \textnormal{(aDec)}, it indicates that there exists some $p>1$ or $q<\infty$ such that the condition holds.
\end{definition}

Now we are in the position to give the definition of a generalized $\Phi$-function.

\begin{definition}
	A function $\ph \colon A \times [0,+\infty) \to [0,+\infty]$ is said to be a generalized $\Phi$-function if $\ph$ is measurable in the first variable, increasing in the second variable and satisfies $\ph(x,0)=0$, $\lim_{t\to 0^+} \ph(x,t) = 0$ and $\lim_{t \to +\infty} \ph(x,t) = +\infty$ for a.a.\,$x \in A$. Moreover, we say that
	\begin{enumerate}
		\item[\textnormal{(i)}]
			$\ph$ is a generalized weak $\Phi$-function if it satisfies \textnormal{(aInc)}$_1$ on $A \times (0,+\infty)$;
		\item[\textnormal{(ii)}]
			$\ph$ is a generalized convex $\Phi$-function if $\ph(x,\cdot)$ is left-continuous and convex for a.a.\,$x \in A$;
		\item[\textnormal{(iii)}]
			$\ph$ is a generalized strong $\Phi$-function if $\ph(x,\cdot)$ is continuous in the topology of $[0,\infty]$ and convex for a.a.\,$x \in A$.
	\end{enumerate}
\end{definition}

\begin{remark}
	A generalized strong $\Phi$-function is a generalized convex $\Phi$-function, and a generalized convex $\Phi$-function is a generalized weak $\Phi$-function, check equation (2.1.1) in the book by Harjulehto--H\"ast\"o \cite{Harjulehto-Hasto-2019}.
\end{remark}

Associated to each generalized $\Phi$-function, it is possible to define its conjugate function and its left-inverse.

\begin{definition}
	Let $\ph \colon A \times [0,+\infty) \to [0,+\infty]$. We denote by $\ph^*$ the conjugate function of $\ph$ which is defined for $x \in A$ and $s \geq 0$ by
	\begin{align*}
		\ph^*(x,s) = \sup_{t \geq 0} (ts - \ph(x,t)).
	\end{align*}
	We denote by $\ph^{-1}$ the left-continuous inverse of $\ph$, defined for $x \in A$ and $s \geq 0$ by
	\begin{align*}
		\ph^{-1}(x,s) = \inf \{t \geq 0\colon \ph(x,t) \geq s\}.
	\end{align*}
\end{definition}

The function spaces that we will build based on these generalized $\Phi$-functions can have specially nice properties if these functions fulfill some extra assumptions like the following ones.

\begin{definition}
	Let $\ph \colon A \times [0,+\infty) \to [0,+\infty]$, we say that
	\begin{enumerate}
		\item[\textnormal{(i)}]
			$\ph$ is doubling (or satisfies the $\Delta_2$-condition) if there exists a constant $K \geq 2$ such that
			\begin{align*}
				\ph(x,2t) \leq K \ph(x,t)
			\end{align*}
			for all $t \in (0,+\infty]$ and for a.a.\,$x \in A$;
		\item[\textnormal{(ii)}]
			$\ph$ satisfies the $\nabla_2$-condition if $\ph^*$ satisfies the $\Delta_2$-condition.
	\end{enumerate}

	Let $\Omega \subset \R^N$ and $\ph \colon \Omega \times [0,+\infty) \to [0,+\infty]$ be a generalized $\Phi$-function, we say that it satisfies the condition
	\begin{enumerate}[leftmargin=1.5cm]
		\item[\textnormal{(A0)}]
			if there exists $0 < \beta \leq 1$ such that $\beta \leq \ph^{-1} (x,1) \leq \beta^{-1}$ for a.a.\,$x \in \Omega$;
		\item[\textnormal{(A0)'}]
			if there exists $0 < \beta \leq 1$ such that $\ph (x,\beta) \leq 1 \leq \ph(x,\beta^{-1})$ for a.a.\,$x \in \Omega$;
		\item[\textnormal{(A1)}]
			if there exists $0 < \beta < 1$ such that $\beta \ph^{-1}(x,t) \leq \ph^{-1}(y,t)$ for every $t \in [1 , 1/|B|]$ and  for a.a.\,$x,y \in B \cap \Omega$ with every ball $B$ such that $|B| \leq 1$;
		\item[\textnormal{(A1)'}]
			if there exists $0 < \beta < 1$ such that $\ph(x,\beta t) \leq \ph (y,t)$ for every $t \geq 0$ such that $\ph(y,t) \in [1 , 1/|B|]$ and  for a.a.\,$x,y \in B \cap \Omega$ with every ball $B$ such that $|B| \leq 1$;
		\item[\textnormal{(A2)}]
			if for every $s>0$ there exists $0 < \beta \leq 1$ and $h \in \Lp{1} \cap \Lp{\infty}$ such that $\beta \ph^{-1}(x,t) \leq \ph^{-1}(y,t)$ for every $t \in [h(x) + h(y) , s]$ and for a.a.\,$x,y \in \Omega$;
		\item[\textnormal{(A2)'}]
			if there exists $s>0$, $0 < \beta \leq 1$, $\ph_\infty$ weak $\Phi$-function (that is, constant in the first variable) and $h \in \Lp{1} \cap \Lp{\infty}$ such that $\ph(x,\beta t) \leq \ph_\infty(t) + h(x)$ and $\ph_\infty (\beta t) \leq \ph(x,t) + h(x)$ for a.a.\,$x \in \Omega$ and for all $t\geq0$ such that $\ph_\infty(t) \leq s$ and $\ph(x,t) \leq s$.
	\end{enumerate}
\end{definition}

It is important to notice that some of the conditions we already mentioned are equivalent. The reason behind this is that, depending on the context, some conditions can be easier to check than others. For the following result, see Lemma 2.2.6, Corollary 2.4.11, Corollary 3.7.4, Corollary 4.1.6 and Lemma 4.2.7 in the book by Harjulehto--H\"ast\"o \cite{Harjulehto-Hasto-2019}.

\begin{lemma}\label{Le:equivalences}
	Let $\ph \colon A \times [0,+\infty) \to [0,+\infty]$ be a generalized weak $\Phi$-function, then
	\begin{enumerate}
		\item[\textnormal{(i)}]
			it satisfies the $\Delta_2$-condition if and only if it satisfies \textnormal{(aDec)};
		\item[\textnormal{(ii)}]
			if it is a generalized convex $\Phi$-function, it satisfies the $\Delta_2$-condition if and only if it satisfies \textnormal{(Dec)};
		\item[\textnormal{(iii)}]
			it satisfies the $\nabla_2$-condition if and only if it satisfies \textnormal{(aInc)};
	\end{enumerate}
	Let $\ph \colon \Omega \times [0,+\infty) \to [0,+\infty]$ be a generalized weak $\Phi$-function, then
	\begin{enumerate}
		\item[\textnormal{(iv)}]
			it satisfies the \textnormal{(A0)} condition if and only if it satisfies the \textnormal{(A0)'} condition;
		\item[\textnormal{(v)}]
			if it satisfies the \textnormal{(A0)} condition, the \textnormal{(A1)} condition holds if and only if the \textnormal{(A1)'} condition holds;
		\item[\textnormal{(vi)}]
			it satisfies the \textnormal{(A2)} condition if and only if it satisfies the \textnormal{(A2)'} condition.
	\end{enumerate}
\end{lemma}

Now we see how a Musielak-Orlicz space is defined, alongside with which properties it has based on the properties of its associated $\Phi$-function. For this result check Lemma 3.1.3, Lemma 3.2.2, Theorem 3.3.7, Theorem 3.5.2 and Theorem 3.6.6 in the book by Harjulehto--H\"ast\"o \cite{Harjulehto-Hasto-2019}.

\begin{proposition}\label{Prop:AbstractBanach}
	Let $\ph \colon A \times [0,+\infty) \to [0,+\infty]$ be a generalized weak $\Phi$-function and let its associated modular be
	\begin{align*}
		\varrho_\ph (u) = \int_A \ph(x,\abs{u(x)}) \diff \mu (x).
	\end{align*}
	Then, the set
	\begin{align*}
		L^\ph (A) = \{ u \in M(A)\colon \varrho_\ph (\lambda u) < \infty \text{ for some } \lambda > 0 \}
	\end{align*}
	equipped with the associated Luxemburg quasi-norm
	\begin{align*}
		\norm{u}_\ph = \inf \left\lbrace \lambda > 0 \colon \varrho_\ph \left( \frac{u}{\lambda} \right)  \leq 1 \right\rbrace
	\end{align*}
	is a quasi Banach space. Furthermore, if $\ph$ is a generalized convex $\Phi$-function, it is a Banach space; if $\ph$ satisfies \textnormal{(aDec)}, it holds that
	\begin{align*}
		L^\ph (A) = \{ u\in M(A) \colon \varrho_\ph (u) < \infty \};
	\end{align*}
	if $\ph$ satisfies \textnormal{(aDec)} and $\mu$ is separable, then $L^\ph (A)$ is separable; and if $\ph$ satisfies \textnormal{(aInc)} and \textnormal{(aDec)} it possesses an equivalent, uniformly convex norm, hence it is reflexive.
\end{proposition}

In Musielak-Orlicz spaces it is important to understand with detail the relation between the modular and the norm, because it is often used when dealing with growths or convergences. Because of this, the following result is of importance, one can find it as Lemma 3.2.9 in the book of Harjulehto--H\"ast\"o \cite{Harjulehto-Hasto-2019}.

\begin{proposition}
	\label{Prop:AbstractNormModular}
	Let $\ph \colon A \times [0,+\infty) \to [0,+\infty]$ be a generalized weak $\Phi$-function that satisfies \textnormal{(aInc)}$_p$ and \textnormal{(aDec)}$_q$, with $1 \leq p \leq q < \infty$. Then
	\begin{align*}
		\frac{1}{a} \min \left\lbrace \norm{u}_\ph^p ,  \norm{u}_\ph^q \right\rbrace
		\leq \varrho_\ph (u)
		\leq a \max \left\lbrace \norm{u}_\ph^p ,  \norm{u}_\ph^q \right\rbrace
	\end{align*}
	for all measurable functions $u \colon A \to \R$, where $a$ is the maximum of the constants of \textnormal{(aInc)}$_p$ and \textnormal{(aDec)}$_q$.
\end{proposition}

There are embedding relations between the Musielak-Orlicz spaces depending of the chosen function. The following result characterizes these relations and can be found as Theorem 3.2.6 of the book by Harjulehto--H\"ast\"o \cite{Harjulehto-Hasto-2019}.

\begin{proposition}
	\label{Prop:AbstractEmbedding}
	Let $\ph, \psi \colon A \times [0,+\infty) \to [0,+\infty]$ be generalized weak $\Phi$-functions and let $\mu$ be atomless. Then $L^\ph (A) \hookrightarrow L^\psi (A)$ if and only if there exits $K>0$ and $h \in L^1 (A)$ with $\norm{h}_1 \leq 1$ such that for all $t \geq 0$ and for a.a.\,$x \in \Omega$
	\begin{align*}
		\psi\left(  x,\frac{t}{K} \right) \leq \ph (x,t) + h(x).
	\end{align*}
\end{proposition}

In Musielak-Orlicz spaces we even have a H\"older inequality based on the chosen function, see the following result which can be found as Lemma 3.2.11 of the book by Harjulehto--H\"ast\"o \cite{Harjulehto-Hasto-2019}.

\begin{proposition}
	Let $\ph \colon A \times [0,+\infty) \to [0,+\infty]$ be a generalized weak $\Phi$-function, then
	\begin{align*}
		\int_A \abs{u} \abs{v} \diff \mu (x) \leq 2 \norm{u}_{\ph} \norm{v}_{\ph^*} \quad \text{for all } u \in L^{\ph}(A), v \in L^{\ph^*}(A).
	\end{align*}
	Moreover, the constant $2$ is sharp.
\end{proposition}

Lastly, one can define the associated Sobolev spaces to these Musielak-Orlicz spaces analogously to the classical case. The following result can be found as Theorem 6.1.4 and Theorem 6.1.9 of the book by Harjulehto--H\"ast\"o \cite{Harjulehto-Hasto-2019}.

\begin{proposition}
	\label{Prop:AbstractSobolev}
	Let $\ph \colon \Omega \times [0,+\infty) \to [0,+\infty]$ be a generalized weak $\Phi$-function such that $\Lp{\ph} \subseteq L^1_{\loc} (\Omega)$ and $k\geq 1$. Then, the set
	\begin{align*}
		W^{k,\ph} (\Omega) = \{ u \in \Lp{\ph} \colon \partial_\alpha u \in \Lp{\ph} \text{ for all } \abs{\alpha} \leq k \},
	\end{align*}
	where we consider the modular
	\begin{align*}
		\varrho_{k,\ph} (u) = \sum_{0 \leq \abs{\alpha} \leq k } \varrho_\ph(\partial_\alpha u)
	\end{align*}
	and the associated Luxemburg quasi-norm
	\begin{align*}
		\norm{u}_{k,\ph} = \inf \left\lbrace \lambda > 0 : \varrho_{k,\ph} \left( \frac{u}{\lambda} \right)  \leq 1 \right\rbrace
	\end{align*}
	is a quasi Banach space. Analogously, the set
	\begin{align*}
		W^{k,\ph}_0 (\Omega) = \overline{C_0^\infty (\Omega)}^{\norm{\cdot}_{k,\ph}},
	\end{align*}
	where $C_0^\infty (\Omega)$ are the functions in $C^\infty (\Omega)$ with compact support, equipped with the same modular and norm is also a quasi Banach space.

	Furthermore, if $\ph$ is a generalized convex $\Phi$-function, both spaces are Banach spaces; if $\ph$ satisfies \textnormal{(aDec)}, then they are separable; and if $\ph$ satisfies \textnormal{(aInc)} and \textnormal{(aDec)} they possess an equivalent, uniformly convex norm, hence they are reflexive.
\end{proposition}

The following statement cannot be found explicitly written in the book by Harjulehto--H\"ast\"o \cite{Harjulehto-Hasto-2019}. However, due to the form of $\varrho_{k,\ph} (\cdot)$ and $\|\cdot\|_{k,\ph}$, one can repeat step by step the arguments of the proof of Lemma 3.2.9 of this same book to obtain it.

\begin{proposition}
	\label{Prop:AbstractoneNormModular}
	Let $\ph \colon \Omega \times [0,+\infty) \to [0,+\infty]$ be a generalized weak $\Phi$-function that satisfies \textnormal{(aInc)}$_p$ and \textnormal{(aDec)}$_q$, with $1 \leq p \leq q < \infty$. Then
	\begin{align*}
		\frac{1}{a} \min \left\lbrace \norm{u}_{k,\ph}^p ,  \norm{u}_{k,\ph}^q \right\rbrace
		\leq \varrho_{k,\ph} (u)
		\leq a \max \left\lbrace \norm{u}_{k,\ph}^p ,  \norm{u}_{k,\ph}^q \right\rbrace
	\end{align*}
	for all $u \in W^{k,\ph} (\Omega)$, where $a$ is the maximum of the constants of \textnormal{(aInc)}$_p$ and \textnormal{(aDec)}$_q$.
\end{proposition}

One might also wonder when smooth functions are dense in a Sobolev Musielak-Orlicz space. Sufficient conditions are given in the next result taken from Theorem 6.4.7 of the book by Harjulehto--H\"ast\"o \cite{Harjulehto-Hasto-2019}.

\begin{theorem}
	\label{Th:AbstractDensity}
	Let $\ph \colon \Omega \times [0,+\infty) \to [0,+\infty]$ be a generalized weak $\Phi$-function that satisfies \textnormal{(A0)}, \textnormal{(A1)}, \textnormal{(A2)} and \textnormal{(aDec)}. Then $C^\infty (\Omega) \cap W^{k,\ph} (\Omega)$ is dense in $W^{k,\ph} (\Omega)$.
\end{theorem}

As a penultimate, let us recall a few properties of the $\log$ function that are useful when dealing with logarithmic growth next to power-law growth. For $s,t \geq 0$ and $C \geq 1$,
\begin{align}
	\label{Eq:LogGrowthProduct}
	\log (e + st) \leq \log (e + s ) + \log (e + t), \qquad \log (e + C s) \leq C \log (e + s)
\end{align}
and for $s,t \geq 0$ and $q \geq 1$,
\begin{equation}
	\label{Eq:LogGrowthSum}
	\begin{aligned}
		(s + t)^q \log (e + s + t)
		 & \leq (2s)^q \log (e + 2s) + (2t)^q \log (e + 2t)        \\
		 & \leq 2^{q+1} s^q \log (e + s) + 2^{q+1} t^q \log(e + t)
	\end{aligned}
\end{equation}

Lastly, we give here the statement of a version of the mountain pass theorem, the quantitative deformation lemma and the Poincar\'{e}-Miranda existence theorem, all of which will be used later. Let $X$ be a Banach space, we say that a functional $\ph \colon X \to \R$ satisfies the Cerami condition or \textnormal{C}-condition if for every sequence $\{u_n\}_{n \in \N} \subseteq X$ such that $\{ \ph(u_n) \}_{n \in \N} \subseteq \R$ is bounded and it also satisfies
\begin{align*}
	( 1 + \| u_n \| )\ph'(u_n) \to 0 \quad\text{as } n \to \infty,
\end{align*}
then it contains a strongly convergent subsequence. Furthermore, we say that it satisfies the Cerami condition at the level $c \in \R$ or the \textnormal{C$_c$}-condition if it holds for all the sequences such that $\ph (u_n) \to c$ as $n \to \infty$ instead of for all the bounded sequences. The proof of the following mountain pass theorem can be found in the book by Papageorgiou--R\u{a}dulescu--Repov\v{s} \cite[Theorem 5.4.6]{Papageorgiou-Radulescu-Repovs-2019a}. For the quantitative deformation lemma after it we refer to the book by Willem \cite[Lemma 2.3]{Willem-1996} and regarding the Poincar\'e-Miranda existence theorem, the proof can be found in the book by Dinca--Mawhin \cite[Corollary 2.2.15]{Dinca-Mawhin-2021}.
\begin{theorem}[Mountain pass theorem] \label{Th:MPT}
	Let X be a Banach space and suppose $\ph \in C^1(X)$, $u_0, u_1 \in X$ with $\| u_1 - u_0 \| > \delta > 0$,
	\begin{align*}
		\max\{\ph(u_0), \ph(u_1)\} & \leq \inf\{\ph(u) \colon \| u - u_0 \| = \delta \} = m_\delta, \\
		c = \inf_{ \gamma \in \Gamma} \max_{ 0 \leq t \leq 1}
		\ph(\gamma (t)) \text{ with } \Gamma & = \{\gamma \in C([0, 1], X) \colon \gamma(0) = u_0, \gamma(1) = u_1\}
	\end{align*}
	and $\ph$ satisfies the \textnormal{C$_c$}-condition. Then $c \geq m_\delta$ and $c$ is a critical value of $\ph$. Moreover, if $c = m_\delta$, then there exists $u \in \partial B_\delta (u_0)$ such that $\ph'(u) = 0$.
\end{theorem}

\begin{lemma}[Quantitative deformation lemma] \label{Le:DeformationLemma}
	Let $X$ be a Banach space, $\ph \in C^1(X;\R)$, $\emptyset \neq S \subseteq X$, $c \in \R$, $\eps,\delta > 0$ such that for all $u \in \ph^{-1}([c - 2\eps, c + 2\eps]) \cap S_{2 \delta}$ there holds $\| \ph'(u) \|_* \geq 8\eps / \delta$, where $S_{r} = \{ u \in X \colon d(u,S) = \inf_{u_0 \in S} \| u - u_0 \| < r \}$ for any $r > 0$.
	Then there exists $\eta \in C([0, 1] \times X; X)$ such that
	\begin{enumerate} [label=(\roman*),font=\normalfont]
		\item
			$\eta (t, u) = u$, if $t = 0$ or if $u \notin \ph^{-1}([c - 2\eps, c + 2\eps]) \cap S_{2 \delta}$;
		\item
			$\ph( \eta( 1, u ) ) \leq c - \eps$ for all $u \in \ph^{-1} ( ( - \infty, c + \eps] ) \cap S $;
		\item
			$\eta(t, \cdot )$ is an homeomorphism of $X$ for all $t \in [0,1]$;
		\item
			$\| \eta(t, u) - u \| \leq \delta$ for all $u \in X$ and $t \in [0,1]$;
		\item
			$\ph( \eta( \cdot , u))$ is decreasing for all $u \in X$;
		\item
			$\ph(\eta(t, u)) < c$ for all $u \in \ph^{-1} ( ( - \infty, c] ) \cap S_\delta$ and $t \in (0, 1]$.
	\end{enumerate}
\end{lemma}

\begin{theorem}[Poincar\'e-Miranda existence theorem] \label{Th:PoincareMiranda}
	Let $P = [-t_1, t_1] \times \cdots \times [-t_N, t_N]$ with $t_i > 0$ for $i \in {1, \ldots, N}$ and $d \colon P \to \R^N$ be continuous. If for each $i \in \{1, \ldots, N\}$ one has
	\begin{align*}
		\begin{aligned}
			d_i (a) &\leq 0 \quad\text{when } a \in P \text{ and } a_i = -t_i,\\
			d_i (a) &\geq 0 \quad\text{when } a \in P \text{ and } a_i = t_i,
		\end{aligned}
	\end{align*}
	then $d$ has at least one zero in $P$.
\end{theorem}

\section{Logarithmic function space}\label{functional-space}

In light of the previous section, we now choose an appropriate $\Phi$-function that allows us to study our problem \eqref{Eq:Problem}. Let $\hlog \colon \close \times [0,\infty) \to [0,\infty)$ be given by
\begin{align*}
	\hlog (x,t) = t^{p(x)} + \mu(x) t^{q(x)} \log (e + t),
\end{align*}
where we assume the following conditions:
\begin{enumerate}[label=\textnormal{(H$_0$)},ref=\textnormal{H$_0$}]
	\item\label{Assump:Space}
	$\Omega \subseteq \R^N$, with $N \geq 2$, is a bounded domain with Lipschitz boundary $\partial \Omega$, $p,q \in C_+ (\close)$ with $p(x) \leq q(x)$ for all $x \in \close$ and $0\leq\mu(\cdot) \in \Lp{1}$.
\end{enumerate}
We can see that this $\Phi$-function satisfies some important properties from the previous section.

\begin{lemma}
	\label{Le:fepsilon}
	The function $f_\eps \colon [0,+\infty) \to  [0,+\infty)$ given by
	\begin{align*}
		f_\eps(t) = \frac{ t^\eps } { \log (e + t) }
	\end{align*}
	is increasing for $\eps \geq \kappa$ and almost increasing for $0 < \eps < \kappa$ with constant $a_\eps$, where $\kappa = e/(e + t_0)$, with $t_0$ being the only positive solution of $t_0 = e \log(e + t_0)$.
\end{lemma}

\begin{proof}
	It is immediate to check that $f_\eps ' (t) > 0$ ($= 0, < 0$) if and only if
	\begin{align*}
		\eps (e + t) \log (e + t) - t > 0 \quad ( = 0, < 0).
	\end{align*}
	Hence we are interested in when the function
	\begin{align*}
		g (t) = \frac { t }{  (e + t) \log (e + t) }
	\end{align*}
	satisfies $g (t) < \eps$ ($= \eps, > \eps$). Arguing similarly, $g ' (t) > 0$ ($= 0, < 0$) if and only if
	\begin{align*}
		e \log (e+t) - t > 0 \quad ( = 0, < 0).
	\end{align*}
	Therefore we now look at the function
	\begin{align*}
		h(t) = e \log (e+t) - t,
	\end{align*}
	which is strictly decreasing, strictly positive at $0$ and $-\infty$ at $+\infty$. Thus $g$ achieves its global maximum at $t_0$ defined by $h(t_0) = 0$ and it holds that $g(t_0) = \kappa$. This proves the case $\eps \geq \kappa$.

	If $0 < \eps < \kappa$, then there exist $t_{1,\eps} < t_0 < t_{2, \eps}$ such that $t_{1,\eps}$ is the unique local maximum of $f_\eps$ and $t_{2,\eps}$ is the unique local minimum of $f_\eps$ apart from $0$. Since $f_\eps$ is increasing in $(0,t_{1,\eps})$ and in $(t_{2,\eps},\infty)$, $f_\eps$ is almost increasing if and only if there exists $a_\eps > 0$ such that $f_\eps (t_{1,\eps}) / f_\eps (t_{2,\eps}) \leq a_\eps$, which is trivially true.

	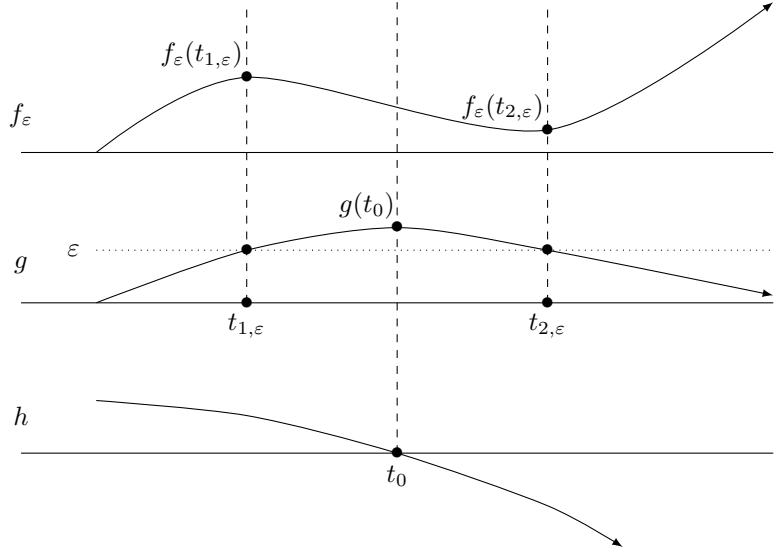
\begin{figure}
		\begin{tikzpicture}
			\draw[-] (-1,0) -- (9,0);
			\draw (-1,0.5) node {$f_\eps$};
			\draw[-latex] plot [smooth] coordinates {(0,0) (2,1) (6,0.3) (9,2)};
			\draw (2,1) node {\textbullet};
			\draw (1.4,1.3) node {$f_\eps(t_{1,\eps})$};
			\draw (6,0.3) node {\textbullet};
			\draw (5.4,0.6) node {$f_\eps(t_{2,\eps})$};
			\draw[-] (-1,-2) -- (9,-2);
			\draw (-1,-1.5) node {$g$};
			\draw[-latex] plot [smooth] coordinates {(0,-2) (2,-1.3) (4,-1) (6,-1.3) (9,-1.9)};
			\draw[dotted] (0,-1.3) -- (9,-1.3);
			\draw (-0.3,-1.3) node {$\eps$};
			\draw (4,-1) node {\textbullet};
			\draw (3.6,-0.7) node {$g(t_0)$};
			\draw (2,-1.3) node {\textbullet};
			\draw (2,-2) node {\textbullet};
			\draw[dashed] (2,-2) -- (2,2);
			\draw (2,-2.3) node {$t_{1,\eps}$};
			\draw (6,-1.3) node {\textbullet};
			\draw (6,-2) node {\textbullet};
			\draw[dashed] (6,-2) -- (6,2);
			\draw (6,-2.3) node {$t_{2,\eps}$};
			\draw[-] (-1,-4) -- (9,-4);
			\draw (-1,-3.5) node {$h$};
			\draw[-latex] plot [smooth] coordinates {(0,-3.3) (2,-3.5) (4,-4) (6,-4.7) (7,-5.25)};
			\draw (4,-4) node {\textbullet};
			\draw (4,-4.3) node {$t_0$};
			\draw[dashed] (4,-4) -- (4,2);
		\end{tikzpicture}
		\caption{Functions in the proof of Lemma \ref{Le:fepsilon}.}
	\end{figure}
\end{proof}

\begin{remark}
	Note that with the choice of $a_\eps$ in the proof of the previous result we cannot ensure that there exists a constant uniform in $\eps$. On the other hand, note also that $t_0 \simeq 5.8340$ and $\kappa \simeq 0.31784$.
\end{remark}

\begin{lemma}
	\label{Le:PropertiesHlog}
	Let \eqref{Assump:Space} be satisfied, then $\hlog$ is a generalized strong $\Phi$-function and it fulfills
	\begin{enumerate}
		\item[\textnormal{(i)}]
			\textnormal{(Inc)}$_{p_-}$;
		\item[\textnormal{(ii)}]
			\textnormal{(Dec)}$_{q_+ + \kappa}$;
		\item[\textnormal{(iii)}]
			\textnormal{(aDec)}$_{q_+ + \eps}$ for $0 < \eps < \kappa$ and with constant $a_\eps$,
	\end{enumerate}
	where $\kappa$ and $a_\eps$ are the same as in Lemma \ref{Le:fepsilon}.
\end{lemma}

\begin{proof}
	First we see that it is a generalized strong $\Phi$-function. Since the other conditions are straightforward, we only need to check the convexity in the second variable, which follows from
	\begin{align*}
		 & {\partial_t}^2\hlog (x,t)\\
		 & = p(x) ( p(x) - 1) t^{p(x) - 2} \\
		 & \quad + \mu(x)
		t^{q(x) -2} \left[  q(x)(q(x) - 1)\log (e + t) + (2q(x) - 1)\frac{t}{e+t} + \frac{et}{(e+t)^2} \right] > 0
	\end{align*}
	for all $t>0$ and for a.a.\,$x \in \Omega$.

	Now we prove \textnormal{(i)}, \textnormal{(ii)} and \textnormal{(iii)}. Notice that
	\begin{align*}
		\frac{\hlog (x,t)}{t^{p_-}} =
		t^{p(x) - p_-} + \mu(x) t^{q(x) - p_-} \log (e + t)
	\end{align*}
	is an increasing function because all the exponents are positive. Similarly, by Lemma \ref{Le:fepsilon}
	\begin{align*}
		\frac{\hlog (x,t)}{t^{q_+ + \eps}} =
		t^{p(x) - q_+ -\eps} + \mu(x) t^{q(x) - q_+} \frac{\log (e + t)}{t^{\eps}}
	\end{align*}
	is a decreasing function when $\eps \geq \kappa$ and almost decreasing when $0 < \eps < \kappa$, with the constant $a_\eps$ from Lemma \ref{Le:fepsilon}.
\end{proof}

As a consequence of the previous result, we obtain  the following.

\begin{proposition}
	\label{Prop:HlogModularNorm}
	Let \eqref{Assump:Space} be satisfied, then $\LHlog$ is a separable, reflexive Banach space and the following hold:
	\begin{enumerate}
		\item[\textnormal{(i)}]
			$\normHlog{u} = \lambda$ if and only if $\modHlog{ \frac{u}{\lambda} } = 1 $ for $u \neq 0$ and $\lambda>0$;
		\item[\textnormal{(ii)}]
			$\normHlog{u} < 1$ (resp. $=1$, $>1$) if and only if $\modHlog{u} < 1$ (resp. $=1$, $>1$);
		\item[\textnormal{(iii)}]
			$\min\left\lbrace \normHlog{u}^{p_-} , \normHlog{u}^{q_+ + \kappa} \right\rbrace
				\leq \modHlog{u} \leq
				\max\left\lbrace \normHlog{u}^{p_-} , \normHlog{u}^{q_+ + \kappa} \right\rbrace $ for $\kappa>0$ as in Lemma \ref{Le:fepsilon};
		\item[\textnormal{(iv)}]
			$\frac{1}{a_\eps}\min\left\lbrace \normHlog{u}^{p_-} , \normHlog{u}^{q_+ + \eps} \right\rbrace
				\leq \modHlog{u} \leq
				a_\eps \max\left\lbrace \normHlog{u}^{p_-} , \normHlog{u}^{q_+ + \eps} \right\rbrace $ for $0 < \eps < \kappa$, where $\kappa$ and $a_\eps$ are the same as in Lemma \ref{Le:fepsilon};
		\item[\textnormal{(v)}]
			$\normHlog{u} \to 0$ if and only if $\modHlog{u} \to 0$;
		\item[\textnormal{(vi)}]
			$\normHlog{u} \to \infty$ if and only if $\modHlog{u} \to \infty$.
	\end{enumerate}
\end{proposition}

\begin{proof}
	First, by Proposition \ref{Prop:AbstractBanach} and Lemma \ref{Le:PropertiesHlog}, we know that $\Lp{\hlog}$ is a separable, reflexive Banach space.

	For \textnormal{(i)} and \textnormal{(ii)}, note that the function $\lambda \mapsto \modHlog{u/\lambda}$ with $\lambda \geq 0$ is continuous, convex and strictly increasing. This directly implies \textnormal{(i)}, and \textnormal{(i)} with the strict increasing property yields \textnormal{(ii)}.

	Finally, \textnormal{(iii)} and \textnormal{(iv)} follow from Proposition \ref{Prop:AbstractNormModular} and Lemma \ref{Le:PropertiesHlog}; and \textnormal{(v)} and \textnormal{(vi)} follow from \textnormal{(iii)}.
\end{proof}

This space satisfies the following embeddings. As in the usual notation, given $r \in C_+(\close)$, let $L^{r(x)} \log L (\Omega) = \Lp{\zeta}$, where $\zeta(x,t) = t^{r(x)}\log (e + t)$.
\begin{proposition}
	\label{Prop:EmbeddingHlog}
	Let \eqref{Assump:Space} be satisfied and let $\eps > 0$, then it holds
	\begin{align*}
		\Lp{q(\cdot) + \eps} \hookrightarrow L^{q(\cdot)} \log L (\Omega)
		\quad\text{and}\quad
		\LHlog \hookrightarrow \Lp{p(\cdot)}.
	\end{align*}
	If we further assume $0\leq\mu(\cdot) \in \Lp{\infty}$, then it also holds $L^{q(\cdot)} \log L (\Omega) \hookrightarrow \LHlog$.
\end{proposition}

\begin{proof}
	We will prove all the embeddings by applying Proposition \ref{Prop:AbstractEmbedding} to the corresponding $\Phi$-functions.

	First, for any $K>0$, by the inequality $\log (e + t) \leq C_\eps + t^\eps$ it holds that
	\begin{align*}
		\left( \frac{t}{K} \right) ^{q(x)} \log \left( e + \frac{t}{K} \right)
		\leq \frac{ t^{q(x) + \eps} + 1 }{ K^{q(x)} } C_\eps + \frac{ t^{q(x) + \eps} }{ K^{q(x) + \eps} }
	\end{align*}
	and if we choose $K \geq \max \{ 2^{1/(q_- + \eps)} , (2 C_\eps)^{1/q_-} , ( C_\eps |\Omega| )^{1/q_-} \}$, it follows
	\begin{align*}
		\left( \frac{t}{K} \right) ^{q(x)} \log \left( e + \frac{t}{K} \right)
		\leq \frac{ C_\eps }{ K^{q(x)} }  + t^{ q(x) + \eps }, \quad \text{ with } \into \frac{ C_\eps }{ K^{q(x)} } \dx \leq 1.
	\end{align*}
	This concludes the proof of the first embedding. The condition for the second embedding is straightforward to verify. Finally, for any $K>0$,
	\begin{align*}
		\hlog \left( x, \frac{t}{K} \right)
		\leq \frac{1}{K^{p(x)}} + \left( \frac{1}{K^{p(x)}} + \frac{ \norm{\mu}_{\infty} }{K^{q(x)}} \right)
		t^{q(x)} \log \left( e + \frac{t}{K} \right)
	\end{align*}
	and if we choose $K \geq \max \{ 2^{1/p_-} , ( 2 \norm{\mu}_{\infty} )^{1/q_-} , |\Omega| ^{1/p_-} \}$, it follows
	\begin{align*}
		\hlog \left( x, \frac{t}{K} \right)
		\leq \frac{1}{K^{p(x)}} + t^{q(x)} \log \left( e + t \right) \quad \text{with } \into \frac{1}{K^{p(x)}} \dx \leq 1.
	\end{align*}
	This shows the proof of the third embedding.
\end{proof}

For our purposes we further need to work on the associated Sobolev space, whose properties are summarized in the following statement. Its proof is completely analogous to the proof of Proposition \ref{Prop:HlogModularNorm} except that now we use Proposition \ref{Prop:AbstractSobolev} (instead of Proposition \ref{Prop:AbstractBanach}) and Proposition \ref{Prop:AbstractoneNormModular} (instead of Proposition \ref{Prop:AbstractNormModular}).

\begin{proposition}
	\label{Prop:oneHlogModularNorm}
	Let \eqref{Assump:Space} be satisfied, then $\WHlog$ and $\WHlogzero$ are separable, reflexive Banach spaces and the following hold:
	\begin{enumerate}
		\item[\textnormal{(i)}]
			$\normoneHlog{u} = \lambda$ if and only if $\modoneHlog{ \frac{u}{\lambda} } = 1 $ for $u \neq 0$ and $\lambda>0$;
		\item[\textnormal{(ii)}]
			$\normoneHlog{u} < 1$ (resp.\,$=1$, $>1$) if and only if $\modoneHlog{u} < 1$ (resp.\,$=1$, $>1$);
		\item[\textnormal{(iii)}]
			$\min\left\lbrace \normoneHlog{u}^{p_-} , \normoneHlog{u}^{q_+ + \kappa} \right\rbrace
				\leq \modoneHlog{u} \leq
				\max\left\lbrace \normoneHlog{u}^{p_-} , \normoneHlog{u}^{q_+ + \kappa} \right\rbrace $;
		\item[\textnormal{(iv)}]
			\begin{align*}
				&\frac{1}{a_\eps}\min\left\lbrace \normoneHlog{u}^{p_-} , \normoneHlog{u}^{q_+ + \eps} \right\rbrace\\
				&\leq \modoneHlog{u} \leq
				a_\eps \max\left\lbrace \normoneHlog{u}^{p_-} , \normoneHlog{u}^{q_+ + \eps} \right\rbrace
			\end{align*}
			for $0 < \eps < \kappa$, where $\kappa$ and $a_\eps$ are the same as in Lemma \ref{Le:fepsilon};
		\item[\textnormal{(v)}]
			$\normoneHlog{u} \to 0$ if and only if $\modoneHlog{u} \to 0$;
		\item[\textnormal{(vi)}]
			$\normoneHlog{u} \to \infty$ if and only if $\modoneHlog{u} \to \infty$.
	\end{enumerate}
\end{proposition}

These Sobolev spaces satisfy the following embeddings.

\begin{proposition}
	\label{Prop:EmbeddingHlogSobolev}
	Let \eqref{Assump:Space} be satisfied, then the following hold:
	\begin{enumerate}
		\item[\textnormal{(i)}]
			$\WHlog \hookrightarrow \Wp{p(\cdot)}$ and $\WHlogzero \hookrightarrow \Wpzero{p(\cdot)}$ are continuous;
		\item[\textnormal{(ii)}]
			if $p \in C_+(\close) \cap C^{0, \frac{1}{|\log t|}}(\close)$, then $\WHlog \hookrightarrow \Lp{p^*(\cdot)}$ and $\WHlogzero$ $\hookrightarrow \Lp{p^*(\cdot)}$ are continuous;
		\item[\textnormal{(iii)}]
			$\WHlog \hookrightarrow \Lp{r(\cdot)}$ and $\WHlogzero \hookrightarrow \Lp{r(\cdot)}$ are compact for $r \in C(\close) $ with $ 1 \leq r(x) < p^*(x)$ for all $x\in \close$;
		\item[\textnormal{(iv)}]
			if $p \in C_+(\close) \cap W^{1,\gamma}(\Omega)$ for some $\gamma>N$, then $\WHlog \hookrightarrow \Lprand{p_*(\cdot)}$ and $\WHlogzero \hookrightarrow \Lprand{p_*(\cdot)}$ are continuous;
		\item[\textnormal{(v)}]
			$\WHlog \hookrightarrow \Lprand{r(\cdot)}$ and $\WHlogzero \hookrightarrow \Lprand{r(\cdot)}$ are compact for $r \in C(\close) $ with $ 1 \leq r(x) < p_*(x)$ for all $x\in \close$.
	\end{enumerate}
\end{proposition}

\begin{proof}
	The proof of \textnormal{(i)} follows directly from Proposition \ref{Prop:EmbeddingHlog}. The proofs of \textnormal{(ii)} - \textnormal{(v)} follow from \textnormal{(i)} and the usual Sobolev embeddings of $\Wp{p(\cdot)}$ and $\Wpzero{p(\cdot)}$ in Propositions \ref{Prop:classicalembedd} and \ref{Prop:classicalembedd:boundary}.
\end{proof}

We also have the property that the truncation of functions on $\WHlog$ and $\WHlogzero$ stays within the space, as it is proven in the next result.

\begin{proposition}
	\label{Prop:Truncations}
	Let \eqref{Assump:Space} be satisfied, then
	\begin{enumerate}
		\item[\textnormal{(i)}]
			If $u \in \WHlog$, then $u^\pm \in \WHlog$ with $\nabla (\pm u) = \nabla u 1_{ \{ \pm u > 0 \} }$;
		\item[\textnormal{(ii)}]
			if $u_n \to u$ in $\WHlog$, then $u_n^\pm \to u^\pm$ in $\WHlog$;
		\item[\textnormal{(iii)}]
			if we further assume $0\leq\mu(\cdot) \in \Lp{\infty}$, then $u \in \WHlogzero$ implies $u^\pm \in \WHlogzero$.
	\end{enumerate}
\end{proposition}

\begin{proof}
	Part \textnormal{(i)} follows from the classical case. Indeed, let $r \in \R$ with $1 \leq r \leq \infty$, then for any $u \in \Wp{r}$ we know that $u^\pm \in \Wp{r}$ and $\nabla (\pm u) = \nabla u 1_{ \{ \pm u > 0 \} }$; see, for example, the book by Heinonen--Kilpel\"{a}inen--Martio \cite[Lemma 1.19]{Heinonen-Kilpelainen-Martio-2006}. The reason is that by Proposition \ref{Prop:EmbeddingHlogSobolev} \textnormal{(i)} $u \in \Wp{p_-}$ and $\abs{\pm u ^\pm} \leq \abs{u}$.

	For part \textnormal{(ii)}, note again that $\abs{\pm u_n ^\pm \mp u^\pm} \leq \abs{u_n - u}$. Then by Proposition \ref{Prop:oneHlogModularNorm} \textnormal{(v)}, we have that $\modHlog{u_n^\pm - u^\pm} \to 0$. It is only left to see the convergence of the terms with the gradients. We do only the logarithmic term, the other one can be done similarly. Note that $\abs{\pm \nabla u_n^\pm \mp \nabla u^\pm} = \abs{1_{\{ \pm u_n > 0 \}} \nabla u_n - 1_{\{ \pm u > 0 \}} \nabla u}$, then by \eqref{Eq:LogGrowthSum} we have
	\begin{align*}
		 & \into \mu(x) \abs{\pm \nabla u_n^\pm \mp \nabla u^\pm}^{q(x)} \log (e + \abs{\pm \nabla u_n^\pm \mp \nabla u^\pm} ) \dx \\
		 & \leq 2^{q_+ + 1} \into \mu(x) \abs{ \nabla u_n - \nabla u } ^{q(x)} \log(e + \abs{ \nabla u_n - \nabla u } ) \dx        \\
		 & \quad + 2^{q_+ + 1} \into \mu(x)\abs{ \nabla u}^{q(x)} \abs{  1_{\{ \pm u_n > 0 \}} - 1_{\{ \pm u > 0 \}} } ^{q(x)}     \\
		 & \qquad\qquad\qquad \times \log (e + \abs{\nabla u} \abs{  1_{\{ \pm u_n > 0 \}} - 1_{\{ \pm u > 0 \}} } ) \dx.
	\end{align*}
	The first term on the right-hand side converges to zero by Proposition \ref{Prop:HlogModularNorm} \textnormal{(v)}. The second one also converges to zero by taking an a.e.\,convergent subsequence, using the dominated convergence theorem and then using the subsequence principle. For the application of the dominated convergence theorem, take into account that $\nabla u = 0$ on the set $\{ u = 0 \}$ by \textnormal{(i)}. All in all, this yields $\modoneHlog{u_n^\pm - u^\pm} \to 0$ and the final result by Proposition \ref{Prop:oneHlogModularNorm} \textnormal{(v)}.

	Part \textnormal{(iii)} is more technical. For $u \in \WHlogzero$ we know that there exists a sequence $\{u_n\}_{n \in \N} \subseteq C^\infty _0 (\Omega)$ such that $u_n \to u$ in $\WHlog$, which by part \textnormal{(ii)} implies $u_n^\pm \to u^\pm$ in $\WHlog$. In particular we have that $u_n^\pm \in C_0 = \{ v \in C(\Omega) : \operatorname{supp} v \text{ is compact}\}$ and $\partial_{x_i} u_n^\pm \in \Lp{\infty}$ for all $n \in \N$ and all $1 \leq i \leq N$. Consider the standard mollifier $\eta_\eps$. For each $n \in \N$ there is a small $\eps_n > 0$ such that $\eta _ \eps \ast u_n^\pm \in C^\infty_0 (\Omega)$ for $0 < \eps < \eps_n$. Moreover, for any $\delta > 0$
	\begin{align*}
		 & \eta_\eps \ast u_n^\pm \to u_n^\pm \quad \text{uniformly in $\Omega$ as } \eps \to 0,   \\
		 & \partial_{x_i} (\eta_\eps \ast u_n^\pm) = \eta_\eps \ast \partial_{x_i} u_n^\pm \to \partial_{x_i} u_n^\pm \quad \text{in } \Lp{q_+ + \delta} \text{ as } \eps \to 0.
	\end{align*}
	By Propositions \ref{Prop:EmbeddingHlog} and \ref{Prop:HlogModularNorm} \textnormal{(v)} this means that $\modoneHlog{\eta _ {\eps } \ast u_n^\pm - u_n^\pm} \to 0$ as $\eps \to 0$, or also in the norm. Altogether, for each $u_n^+, u_n^- $ we can find some $v_n, \widetilde{v}_n \in C_0 ^\infty (\Omega)$ as close as we want to them in the norm of $\WHlog$; and these new sequences satisfy $v_n \to u^+$ and $\widetilde{v}_n \to u^-$ in $\WHlog$.
\end{proof}

As a consequence of the embeddings above we can prove that a Poincar\'{e} inequality is satisfied in $\WHlogzero$ if we further assume  in \eqref{Assump:Space} that $0\leq\mu(\cdot) \in \Lp{\infty}$ and $q(x) < p^*(x)$ for all $x \in \close$. So we suppose the following:
\begin{enumerate}[label=\textnormal{(H)},ref=\textnormal{H}]
	\item\label{Assump:Poincare}
	$\Omega \subseteq \R^N$, with $N \geq 2$, is a bounded domain with Lipschitz boundary $\partial \Omega$, $p,q \in C_+ (\close)$ with $p(x) \leq q(x) < p^*(x)$ for all $x \in \close$ and $0\leq\mu(\cdot) \in \Lp{\infty}$.
\end{enumerate}
As it is usual in the literature, we denote $\normHlog{ \nabla u} = \normHlog{\,\abs{\nabla u}\,} $ and $\modHlog{ \nabla u} = \modHlog{ \abs{\nabla u} } $.

\begin{proposition}
	\label{Prop:Poincare}
	Let \eqref{Assump:Poincare} be satisfied. Then $\WHlog \hookrightarrow \LHlog$ is a compact embedding and there exists a constant $C>0$ such that
	\begin{align*}
		\normHlog{u} \leq C \normHlog{ \nabla u } \quad \text{for all } u \in \WHlogzero.
	\end{align*}
\end{proposition}

\begin{proof}
	From \eqref{Assump:Poincare} we can deduce that there exists $\eps > 0$ such that $q(x) + \eps < p^*(x)$ for all $x \in \close$. Then the compact embedding follows from Propositions \ref{Prop:EmbeddingHlogSobolev} \textnormal{(iii)} and \ref{Prop:EmbeddingHlog}.

	The inequality follows from the compact embedding in the standard way (see, for example, Proposition 2.18 in the paper by Crespo-Blanco--Gasi\'{n}ski--Harjulehto--Winkert \cite{Crespo-Blanco-Gasinski-Harjulehto-Winkert-2022})
\end{proof}

Due to the previous inequality we can take in the space $\WHlogzero$ the equivalent norm $\normoneHlogzero{u} = \normHlog{ \nabla u }$.

In the last part of this section, we investigate the density of smooth functions in the space $\WHlog$. For this purpose, we check which of the \textnormal{(A0)}, \textnormal{(A1)} and \textnormal{(A2)} assumptions are satisfied by $\hlog$.

\begin{lemma}
	\label{Le:A0A2}
	Let \eqref{Assump:Space} be satisfied, then $\hlog$ satisfies \textnormal{(A0)} and \textnormal{(A2)}.
\end{lemma}

\begin{proof}
	By Lemma \ref{Le:equivalences} \textnormal{(iv)} and \textnormal{(vi)}, it is equivalent to check the conditions \textnormal{(A0)'} and \textnormal{(A2)'}. For \textnormal{(A0)'} one can take $\beta = \left[ 2 ( \norm{\mu}_\infty + 1) \log(e + 1/2) \right] ^{-1}$ since
	\begin{align*}
		\hlog (x,\beta) & \leq \frac{1}{2} + \frac{1}{2} \frac{ \norm{\mu}_\infty }{ \norm{\mu}_\infty + 1 } \frac{ \log \left( e + \frac{1}{2} \right) }{ \log \left( e + \frac{1}{2} \right) } \leq 1, \\
		\hlog (x,\beta^{-1}) & \geq 2 + 0 \geq 1.
	\end{align*}
	For \textnormal{(A2)'}, take $s=1$, $\ph_\infty(t) = t^{p_+ + 1}$ and $\beta = 1$. First note that $\ph_\infty(t) \leq 1$ implies $t \leq 1$. Thus, by Young's inequality
	\begin{align*}
		\hlog (x, \beta t)
		 & \leq [ 1 + \norm{\mu}_\infty \log (e + 1) ] t^{p(x)}                                                                                                             \\
		 & \leq \frac{ p(x) }{ p_+ + 1 } t^{ p_+ + 1 } + \frac{p_+ - p (x) + 1 }{ p_+ + 1 } [ 1 + \norm{\mu}_\infty \log( e + 1 ) ] ^{ \frac{ p_+ + 1 }{p_+ - p (x) + 1 } } \\
		 & \leq \ph_\infty (t) + [ 1 + \norm{\mu}_\infty \log( e + 1 ) ] ^{ p_+ + 1  }.
	\end{align*}
	Take $h$ as the additive constant in the previous line, then we also have
	\begin{align*}
		\ph_\infty ( \beta t )
		\leq t^{p(x)} \leq \hlog (x,t) + h(x).
	\end{align*}
\end{proof}

\begin{remark}
	Note that in the proof of \textnormal{(A2)'} the function $h$ would not be in $\Lp{1}$ if $|\Omega| = \infty$, so this argument does not generalize for unbounded domains. For that purpose, one needs an extra assumption, see Theorem \ref{Th:DensityUnbounded}.
\end{remark}

For the remaining assumption \textnormal{(A1)} we use much stricter assumptions. Indeed, instead of supposing just continuity on $p$ and $q$ as well as $L^1$-integrability of $\mu$, we require now, among others, that all three exponents are H\"{o}lder continuous on $\overline{\Omega}$.

\begin{theorem}
	Let $\Omega \subseteq \R^N$, with $N \geq 2$, be a bounded domain and the functions $p,q \colon \close \to [1,\infty)$ and $\mu \colon \close \to [0,\infty)$ be H\"older continuous functions such that $1 < p(x) \leq q(x)$ for all $x \in \close$ and
	\begin{align*}
		\left( \frac{q}{p} \right)_+ < 1 + \frac{\gamma}{N},
	\end{align*}
	where $\gamma$ is the H\"older exponent of $\mu$. Then $\hlog$ satisfies \textnormal{(A1)} and $C^\infty (\Omega) \cap \WHlog$ is dense in $\WHlog$.
\end{theorem}

\begin{proof}
	It suffices to show that $\hlog$ satisfies \textnormal{(A1)} because the density follows from Theorem \ref{Th:AbstractDensity}, Lemma \ref{Le:PropertiesHlog} and Lemma \ref{Le:A0A2}.

	By Lemma \ref{Le:equivalences} \textnormal{(v)}, it is equivalent to check \textnormal{(A1)'}. Let $B \subseteq \R^N$ be a ball such that $|B| \leq 1$. We start by rewriting the condition $\hlog(y,t) \in [ 1 , 1/|B|]$ into a simpler statement. From this condition we can derive that
	\begin{align*}
		\begin{rcases}
			\text{if } t \leq 1, & 1 \leq \hlog(y,t)  \\
			\text{if } t \geq 1, & \phantom{\leq \hlog(y,t)} 1
		\end{rcases}
		\leq [ 1 + \norm{\mu}_\infty \log(e + 1)] t^{p (y)},
	\end{align*}
	and
	\begin{align*}
		\begin{rcases}
			\text{if } t \leq 1, & \phantom{\hlog(y)} \frac{1}{|B|} \geq 1 \\
			\text{if } t \geq 1, & \frac{1}{|B|} \geq \hlog(y,t)
		\end{rcases}
		\geq t^{p (y)},
	\end{align*}
	which altogether means
	\begin{align}
		\label{Eq:A1Interval}
		\frac{1}{[ 1 + \norm{\mu}_\infty \log(e + 1)]^{\frac{1}{p (y)} }} \leq t \leq  \frac{1}{|B|^{\frac{1}{p (y)} }}.
	\end{align}

	{\bf Claim:} There exists a constant $M > 0$ depending only on $N, p, q, \mu$ such that
	\begin{align*}
		t^{p(x)} \leq M t^{p(y)}
		\quad \text{and} \quad
		t^{q(x)} \leq M t^{q(y)}
	\end{align*}
	for all $x,y \in B \cap \Omega$ and $t \geq 0$ such that $\hlog(y,t) \in [ 1 , 1/|B|]$, and any ball $B \subseteq \R^N$ such that $|B| \leq 1$.

	We only do the $p$ case, the $q$ case is identical. If either $t \leq 1$ and $p(x) \geq p(y)$, or $t \geq 1$ and $p(x) \leq p(y)$, then one can simply take $M=1$. If $t \leq 1$ and $p(x) \leq p(y)$, from \eqref{Eq:A1Interval} we obtain
	\begin{align*}
		t^{p(x)} = t^{p(x) - p(y)} t^{p(y)}
		 & \leq \left( [ 1 + \norm{\mu}_\infty \log(e + 1)]^{\frac{1}{p (y)} } \right) ^{p(y) - p(x)} t^{p(y)} \\
		 & \leq [ 1 + \norm{\mu}_\infty \log(e + 1)]^{\frac{p_+}{p_-}}  t^{p(y)},
	\end{align*}
	so we can take $M = [ 1 + \norm{\mu}_\infty \log(e + 1)]^{ p_+ / p_- }$. Consider the remaining case $t \geq 1$ and $p(x) \geq p(y)$. Note that for any ball $B \subseteq \R^N$ of radius $R$ we have $|B| = \omega(N) R^N$, where $\omega(N) > 0$ is the constant associated to $\R^N$, and also that $x,y \in |B|$ implies $\abs{x-y} \leq 2R$. As $p$ is H\"older continuous with exponent $0 < \alpha \leq 1$ and constant $c_p > 0$, $x,y \in |B|$ implies $\abs{p(x) - p(y)} \leq c_p 2^\alpha R^\alpha$. Remember that $|B| \leq 1$, so $|B|^{-1/p(y)} \leq |B|^{-1/p_-}$. From this and \eqref{Eq:A1Interval} we get
	\begin{align*}
		t^{p(x)} = t^{p(x) - p(y)} t^{p(y)}
		 & \leq \left( \omega(N) R^N \right) ^{ \frac{-1}{p_-} c_p 2^\alpha R^\alpha } t^{p(y)} \\
		 & = \left( \omega(N)^{\frac{ - c_p 2^\alpha }{ p_- }} \right)^{R^\alpha} \left(  R^{R^\alpha} \right)^{\frac{- c_p 2^\alpha N}{p_-}} t^{p(y)}.
	\end{align*}
	Since $|B| \leq 1$, we know that $R \leq \omega(N)^{-1/N}$. On the other hand, the function $h(R) = \left( \omega(N)^{ - c_p 2^\alpha / p_- } \right)^{R^\alpha} \left(  R^{R^\alpha} \right)^{ - c_p 2^\alpha N / p_- }$ is strictly positive and continuous in the interval $[0 , \omega(N)^{-1/N}]$ (note $h(0) = 1)$. Thus it attains its maximum at some $R_0$ in that interval and we can take $M = h(R_0)$. This ends the proof of the Claim.

	Let us now prove the inequality of \textnormal{(A1)'} using the previous information. To this end, let $ 0 < \beta < M^{-1/p_-} < 1$, $0 < \gamma \leq 1$ be the H\"{o}lder exponent of $\mu$ and $c_\mu > 0$ be the corresponding constant, so as in the proof of the Claim  $x,y \in |B|$ implies $\abs{\mu(x) - \mu(y)} \leq c_\mu 2^\gamma R^\gamma$. By all of this and the Claim above we have
	\begin{align*}
		\hlog(\beta x, t) & \leq \beta^{p_-} \left( t^{p(x)} + \mu(x) t^{q(x)} \log(e + t) \right)  \\
		& \leq \beta^{p_-} M \left( t^{p(y)} + \mu(x) t^{q(y)} \log(e + t) \right)  \\
		& \leq \beta^{p_-} M \left( t^{p(y)} + \mu(y) t^{q(y)} \log(e + t) + c_\mu 2^\gamma R^\gamma t^{q(y)} \log(e + t) \right)  \\
		& \leq \beta^{p_-} M \left( t^{p(y)} + c_\mu 2^\gamma R^\gamma t^{q(y)} \log(e + t) \right) + \mu(y) t^{q(y)} \log(e + t).
	\end{align*}
	We continue the inequality using \eqref{Eq:A1Interval}, where we take again into account that $|B| = \omega(N) R^N$, with result
	\begin{align*}
		 & \hlog(\beta x, t) \\
		 & \leq \beta^{p_-} M t^{p(y)} \left( 1 + c_\mu 2^\gamma R^\gamma t^{q(y) - p(y)} \log(e + t) \right) + \mu(y) t^{q(y)} \log(e + t)  \\
		 & \leq \beta^{p_-} M t^{p(y)} \left[ 1 + c_\mu 2^\gamma R^\gamma \left( \omega(N) R^N \right) ^{ \frac{-1}{p(y)} (q(y) - p(y)) } \log \left( e + \left( \omega(N) R^N \right) ^{ \frac{-1}{p(y)} } \right) \right] \\
		 & \quad + \mu(y) t^{q(y)} \log(e + t).
	\end{align*}
	Now we need to estimate the part in square brackets independently of $y$ and $R$. Let $\tau_{p,q,N} = q_-/p_+$ if $\omega(N) > 1$ and $\tau_{p,q,N} = q_+/p-$ if $\omega(N) \leq 1$. Once again, as $\omega(N) R^N \leq 1$
	\begin{align*}
		 & R^\gamma \left( \omega(N) R^N \right) ^{ \frac{-1}{p(y)} (q(y) - p(y)) } \log \left( e + \left( \omega(N) R^N \right) ^{ \frac{-1}{p(y)} } \right) \\
		 & \leq \omega(N)^{1 - \tau_{p,q,N}} R^{ \gamma + N - N\frac{q(y)}{p(y)}} \log \left( e + \left( \omega(N) R^N \right) ^{ \frac{-1}{p_-} } \right).
	\end{align*}
	Let us distinguish two cases. If $R \leq 1$,
	\begin{align*}
		 & R^{ \gamma + N - N\frac{q(y)}{p(y)}} \log \left( e + \left( \omega(N) R^N \right) ^{ \frac{-1}{p_-} } \right)                           \\
		 & \leq R^{ \gamma + N - N \left( \frac{q}{p} \right)_+ } \log \left( e + \left( \omega(N) R^N \right) ^{ \frac{-1}{p_-} } \right) = h(R).
	\end{align*}
	This function $h$ is positive and continuous in the interval $[0 , \omega(N)^{-1/N}]$ because
	\begin{align*}
		\lim_{R \to 0^+} h(R) = 0,
	\end{align*}
	where this limit follows from $\gamma + N - N (q/ p)_+ > 0$ and by L'Hospital's rule. Hence $h$ attains its maximum at some $R_0$ in that interval and we can use $h(R_0)$ as upper estimate. In the other case, if $R \geq 1$
	\begin{align*}
		 & R^{ \gamma + N - N\frac{q(y)}{p(y)}} \log \left( e + \left( \omega(N) R^N \right) ^{ \frac{-1}{p_-} } \right) \\
		 & \leq R^{ \gamma + N - N \left( \frac{q}{p} \right)_- } \log \left( e + \left( \omega(N) \right) ^{ \frac{-1}{p_-} } \right)\\
		 & \leq \omega(N)^{ \frac{- \gamma}{N} - 1 + \left( \frac{q}{p} \right)_- } \log \left( e + \left( \omega(N) \right) ^{ \frac{-1}{p_-} } \right) = \widetilde{\Lambda}_{p,q,N},
	\end{align*}
	which follows from $\gamma + N - N ( \frac{q}{p} )_- > 0$ and $R \leq \omega(N)^{-1/N}$ (or equivalently $|B| \leq 1$). Let $\Lambda_{p,q,N}$ be the maximum of $\omega(N)^{1 - \tau_{p,q,N}} h(R_0)$ and $\omega(N)^{1 - \tau_{p,q,N}} \widetilde{\Lambda}_{p,q,N}$. Altogether, we have proved that
	\begin{align*}
		\hlog(\beta x, t) \leq \beta^{p_-} M t^{p(y)} \left[ 1 + c_\mu 2^\gamma \Lambda_{p,q,N} \right] + \mu(y) t^{q(y)} \log(e + t).
	\end{align*}
	If we take
	\begin{align*}
		\beta < M^{\frac{-1}{p_-}} \left[ 1 + c_\mu 2^\gamma \Lambda_{p,q,N} \right] ^{\frac{-1}{p_-}},
	\end{align*}
	we obtain $\hlog(\beta x, t) \leq \hlog(y, t)$ and the proof is complete.
\end{proof}

In this work and in the previous result we deal with bounded domains. However, one can also obtain the density for unbounded domains by adding one more assumption at infinity, hence we include it here for the shake of completion. Let $\Omega \subseteq \R^N$ be an open subset, we say that a measurable function $r \colon \Omega \to [1,\infty]$ satisfies Nekvinda's decay condition if there exists $r_\infty \in [1 , \infty]$ and $c \in (0,1)$ such that
\begin{align*}
	\into c^{\frac{1}{ \abs{ \frac{1}{r(x)} - \frac{1}{r_\infty} } }} \dx < \infty,
\end{align*}
or, equivalently, $1 \in \Lp{s(\cdot)}$, where $s^{-1} (x) = \abs{ p^{-1} (x) - p^{-1}_\infty}$. This condition was first introduced in the paper by Nekvinda \cite{Nekvinda-2004}.

\begin{theorem}
	\label{Th:DensityUnbounded}
	Let $\Omega \subseteq \R^N$, with $N \geq 2$, be an unbounded domain and the functions $p,q \colon \close \to [1,\infty)$ and $\mu \colon \close \to [0,\infty)$ be bounded, H\"older continuous functions such that $p$ satisfies Nekvinda's decay condition, $1 < p(x) \leq q(x)$ for all $x \in \close$ and
	\begin{align*}
		\left( \frac{q}{p} \right)_+ < 1 + \frac{\gamma}{N},
	\end{align*}
	where $\gamma$ is the H\"older exponent of $\mu$. Then $\hlog$ satisfies \textnormal{(A0)}, \textnormal{(A1)}, \textnormal{(A2)}, \textnormal{(aDec)} and $C^\infty (\Omega) \cap \WHlog$ is dense in $\WHlog$.
\end{theorem}

\begin{proof}
	The proof of \textnormal{(A0)}, \textnormal{(A1)} and \textnormal{(aDec)} is exactly like in Lemma \ref{Le:PropertiesHlog} and Lemma \ref{Le:A0A2}, they are not affected if $\Omega$ is unbounded. Therefore, it suffices to show that $\hlog$ satisfies \textnormal{(A2)} because the density follows from Theorem \ref{Th:AbstractDensity}.

	By Lemma \ref{Le:equivalences} \textnormal{(v)}, it is equivalent to check \textnormal{(A2)'}. For this purpose, take $s=1$, $\ph_\infty(t) = t^{p_\infty}$ and $\beta \leq 1$. First note that $\ph_\infty(t) \leq 1$ implies $t \leq 1$. Let us distinguish two cases. In the points where $p(x) < p_\infty$, by Young's inequality
	\begin{align*}
		\hlog(x,\beta t) & \leq [ 1 + \norm{\mu}_\infty \log (e + 1) ] \beta^{p(x)} t^{p(x)}  \\
		& \leq \frac{ p(x) }{ p_\infty } t^{ p_\infty } + \frac{p_\infty - p (x) }{ p_\infty } [ 1 + \norm{\mu}_\infty \log( e + 1 ) ] ^{ \frac{ p_\infty }{p_\infty - p (x) } } \beta ^{\frac{1}{ \abs{ \frac{1}{p(x)} - \frac{1}{p_\infty} } }} \\
		& \leq \ph_\infty (t) + \left( \beta [ 1 + \norm{\mu}_\infty \log( e + 1 ) ] \right)  ^{\frac{1}{ \abs{ \frac{1}{p(x)} - \frac{1}{p_\infty} } }}.
	\end{align*}
	Let us take
	\begin{align*}
		\beta & < c [ 1 + \norm{\mu}_\infty \log( e + 1 ) ]^{-1},\\
		h(x)  & = \left( \beta [ 1 + \norm{\mu}_\infty \log( e + 1 ) ] \right)  ^{\frac{1}{ \abs{ \frac{1}{p(x)} - \frac{1}{p_\infty} } }},
	\end{align*}
	where $c \in (0,1)$ is the constant of Nekvinda's decay condition of $p$. Then we know that $h \in \Lp{1} \cap \Lp{\infty}$. In the points where $p(x) \geq p_\infty$, with the same choice of $\beta$ we have
	\begin{align*}
		\hlog(x,\beta t) \leq [ 1 + \norm{\mu}_\infty \log (e + 1) ] \beta t^{p_\infty} \leq \ph_\infty (t).
	\end{align*}
	We do the other inequality in a similar way. In the points where $p(x) \leq p_\infty$, as $\beta \leq 1$
	\begin{align*}
		\ph_\infty ( \beta t ) \leq t^{p(x)} \leq \hlog(x,t) ,
	\end{align*}
	and in the points where $p(x) > p_\infty$, using again Young's inequality
	\begin{align*}
		\ph_\infty ( \beta t )
		 & \leq \frac{ p_\infty }{ p(x) } t^{ p(x) } + \frac{p(x) - p _\infty }{ p(x) } \beta ^{\frac{1}{ \abs{ \frac{1}{p(x)} - \frac{1}{p_\infty} } }} \\
		 & \leq \hlog(x,t) + h(x).
	\end{align*}
\end{proof}

\begin{remark}
	Note that in the previous result we only imposed Nekvinda's decay condition on $p$, there is no condition at infinity imposed on $q$.
\end{remark}

\section{Energy functional and logarithmic  operator}\label{energy-functional-operator}

In this section we investigate the properties of the associated energy functional and the logarithmic operator given in \eqref{log-operator}. We denote by $\langle\cdot\,,\cdot\rangle$ the duality pairing between $\WHlogzero$ and its dual space $[\WHlogzero]^*$. Let $A \colon \WHlogzero \to [\WHlogzero]^*$ be our operator of interest, which for each $u,v \in \WHlogzero$ is given by the expression
\begin{align*}
	\left\langle A(u),v \right\rangle
	 & = \into \abs{\nabla u}^{p(x)-2}\nabla
	u \cdot \nabla v\dx \\
	 & \quad +\into \mu(x) \left[ \log (e + \abs{\nabla u} ) + \frac{\abs{\nabla u}}{q(x) (e + \abs{\nabla u})} \right]  \abs{\nabla u}^{q(x)-2}\nabla u \cdot \nabla v \dx.
\end{align*}
Furthermore, let $I \colon \WHlogzero \to \R$ be its associated energy functional, which for each $u \in \WHlogzero$ is given by
\begin{align*}
	I(u)=\into \left( \frac{\abs{\nabla u}^{p(x)}}{p(x)}
	+ \mu(x) \frac{\abs{\nabla u}^{q(x)}}{q(x)} \log (e + \abs{\nabla u}) \right) \dx.
\end{align*}

We first deal with the differentiability of the energy functional.

\begin{theorem}
	\label{Th:C1functional}
	Let \eqref{Assump:Space} be satisfied, then the functional $I$ is $C^1$ with $I'(u) = A(u)$.
\end{theorem}

\begin{proof}
	By the additivity of the Fr\'echet derivative, we only do the argument for the logarithmic term, the other one can be done analogously (and can also be found in previous literature, for example, Proposition 3.1 in the paper by Crespo-Blanco--Gasi\'{n}ski--Harjulehto--Winkert \cite{Crespo-Blanco-Gasinski-Harjulehto-Winkert-2022}). The proof is divided in two steps: first we prove the Gateaux differentiability and later its continuous dependence.

	For the Gateaux differentiability, let $u,v \in \WHlogzero$ and $t \in \R$. In the points where $\nabla u \neq 0$, by the mean value theorem there exists $\theta_{x,t} \in (0,1)$ such that
	\begin{align*}
		 & \beta (x,t)  \\
		 & = \frac{\mu(x)}{t q(x)} \left( \abs{\nabla u + t \nabla v}^{q(x)} \log (e + \abs{\nabla u + t \nabla v}) - \abs{\nabla u}^{q(x)} \log (e + \abs{\nabla u}) \right) \\
		 & = \mu(x) \left( \log (e + \abs{\nabla u + t \theta_{x,t} \nabla v}) \abs{\nabla u + t \theta_{x,t} \nabla v}^{q(x) - 2} (\nabla u + t \theta_{x,t} \nabla v) \cdot \nabla v \vphantom{\frac{(\nabla u + t \theta_{x,t} \nabla v) \cdot \nabla v}{ \abs{\nabla u + t \theta_{x,t} \nabla v} }} \right. \\
		 & \quad \left. + \abs{\nabla u + t \theta_{x,t} \nabla v}^{q(x)} \frac{1}{q(x)(e + \abs{\nabla u + t \theta_{x,t} \nabla v})} \frac{(\nabla u + t \theta_{x,t} \nabla v) \cdot \nabla v}{ \abs{\nabla u + t \theta_{x,t} \nabla v} } \right) \\
		 & \xrightarrow{t \to 0} \mu(x) \left( \log (e + \abs{\nabla u }) \abs{\nabla u }^{q(x) - 2} \nabla u  \cdot \nabla v + \abs{\nabla u }^{q(x)} \frac{1}{q(x)(e + \abs{\nabla u })} \frac{\nabla u  \cdot \nabla v}{ \abs{\nabla u } } \right)
	\end{align*}
	and in the points where $\nabla u = 0$ we can derive the same limit directly. So this convergence is true a.e.\,in $\Omega$. On the other hand, using \eqref{Eq:LogGrowthSum}, for $t \leq 1$
	\begin{align*}
		\abs{\beta(x,t)}
		& \leq \mu(x) \left[ \left( \abs{ \nabla u } +  \abs{ \nabla v } \right)^{q(x) - 1} \abs{ \nabla v} \left( \log( e + \abs{ \nabla u } + \abs{ \nabla v }) + 1 \right) \right] \\
		& \leq 2^{q^+ +1}\mu(x) \left[|\nabla u|^{q(x)} \log(e+ |\nabla u|) +  |\nabla v|^{q(x)} \log(e+ |\nabla v|) + 1 \right],
	\end{align*}
	which is an $L^1(\Omega)$-function. Hence, by the dominated convergence theorem, we proved that the Gateaux derivative exists and coincides with $A$.

	For the $C^1$-property, let $u_n \to u$ in $\WHlogzero$ and $v \in \WHlogzero$ with $\normoneHlog{v} \leq 1$. For the following computations, we define
	\begin{align*}
		w_n(x) & = \left[ \log (e + \abs{\nabla u_n}) + \frac{\abs{\nabla u_n}}{q(x) (e + \abs{\nabla u_n})}\right] \abs{\nabla u_n}^{q(x) -2 }\nabla u_n  \\
		& \quad - \left[ \log (e + \abs{\nabla u}) + \frac{\abs{\nabla u}}{q(x) (e + \abs{\nabla u})}\right] \abs{\nabla u}^{q(x) -2 }\nabla u, \\
		g_n(x)& = \log^{1/q(x)} ( e + \max\{ \abs{\nabla u}, \abs{\nabla u_n} \}),\\
		\Omega_u & = \{ x \in \Omega \colon  \max\{ \abs{\nabla v}, \abs{\nabla u}, \abs{\nabla u_n} \} = \abs{\nabla u}\}, \\
		\Omega_{u_n} & = \{ x \in \Omega \colon  \abs{\nabla u} < \max\{ \abs{\nabla v}, \abs{\nabla u}, \abs{\nabla u_n} \} = \abs{\nabla u_n}\},  \\
		\Omega_v & = \{ x \in \Omega \colon  \abs{\nabla u}, \abs{\nabla u_n} < \max\{ \abs{\nabla v}, \abs{\nabla u}, \abs{\nabla u_n} \} = \abs{\nabla v}\}.
	\end{align*}
	By H\"older's inequality in $\Lp{q(\cdot)}$, we have
	\begin{align*}
		\quad \left|  \into \mu(x) w_n  \cdot \nabla v \dx \right|
		\leq 2 \norm{(\mu (x))^{\frac{q(x) - 1}{q(x)}} \frac{\abs{w_n}}{g_n}}_{\frac{q(\cdot)}{q(\cdot) - 1}}
		\norm{(\mu (x))^{1/q(x)} \abs{\nabla v} g_n }_{q(\cdot)}.
	\end{align*}
	The second factor is uniformly bounded in $n$ and $v$ by Proposition \ref{Prop:ModularNormVarExp} \textnormal{(vi)} and
	\begin{align*}
		\varrho_{ q(\cdot)} \left( (\mu (x))^{1/q(x)} \abs{\nabla v} g_n \right)
		 & \leq  \int_{\Omega_u} \mu(x) \abs{ \nabla v }^{q(x)} \log( e + \abs{\nabla u} ) \dx                                     \\
		 & \quad + \int_{\Omega_{u_n}} \mu(x) \abs{ \nabla v }^{q(x)} \log( e + \abs{\nabla u_n} ) \dx                             \\
		 & \quad +  \int_{\Omega_v} \mu(x) \abs{ \nabla v }^{q(x)} \log( e + \abs{\nabla v} ) \dx                                  \\
		 & \leq \modHlog{ \nabla u } + \underbrace{\modHlog{ \nabla v }}_{\leq 1} + \underbrace{\modHlog{ \nabla u_n } }_{\leq M},
	\end{align*}
	where the last two estimates follow from Proposition \ref{Prop:HlogModularNorm} \textnormal{(ii)} and \textnormal{(vi)} and $u_n \to u$ in $\WHlogzero$. Therefore, we only need to prove that the first factor converges to zero. By Proposition \ref{Prop:ModularNormVarExp} \textnormal{(v)}, it is enough to see that this happens in the modular of $\Lp{\frac{q(\cdot)}{q(\cdot) - 1}}$, that is
	\begin{align*}
		\varrho_{ \frac{q(\cdot)}{q(\cdot) - 1} } \left( (\mu (x))^{\frac{q(x) - 1}{q(x)}} \frac{\abs{w_n}}{g_n} \right)
		= \into \mu(x) \left(  \frac{\abs{w_n } }{g_n} \right) ^{\frac{q(x)}{q(x) - 1}}  \dx \; \xrightarrow{n \to \infty} 0.
	\end{align*}
	We prove this convergence by using Vitali's theorem. By Proposition \ref{Prop:EmbeddingHlogSobolev} \textnormal{(i)}, $\nabla u_n \to \nabla u$ in measure, and using the property that convergence in measure is preserved by composition with continuous functions, we obtain the convergence in measure to zero of the integrand. For the uniform integrability, note that
	\begin{align*}
		 & \mu(x) \left(  \frac{\abs{w_n } }{g_n} \right) ^{\frac{q(x)}{q(x) - 1}} \\
		 & \leq \mu(x) \left( \left[ \log^{1 - \frac{1}{q(x)}} (e + \abs{\nabla u_n}) + 1 \right] \abs{\nabla u_n}^{q(x) - 1}\right. \\
		 & \left.\qquad\qquad  + \left[ \log^{1 - \frac{1}{q(x)}} (e + \abs{\nabla u}) + 1 \right] \abs{\nabla u}^{q(x) - 1}  \right) ^{\frac{q(x)}{q(x) - 1}} \\
		 & \leq C \mu(x) \left( \left[ \log (e + \abs{\nabla u_n}) + 1 \right] \abs{\nabla u_n}^{q(x)}
		+ \left[ \log (e + \abs{\nabla u}) + 1 \right] \abs{\nabla u}^{q(x)}  \right).
	\end{align*}
	As $\nabla u_n \to \nabla u$ in measure and $\modHlog{ \nabla u_n } \to \modHlog{ \nabla u }$, we know that\\$\abs{\nabla u_n}^{q(x)} \log (e + \abs{ \nabla u_n })$ is uniformly integrable, hence we also know that our sequence is uniformly integrable and this finishes the proof.
\end{proof}

Next we are concerned with the properties of the operator $A$. For this purpose, we need the following two lemmas. The first one is concerned with the monotonicity of terms that are not power laws, but still something similar.

\begin{lemma}
	\label{Le:MonotoneInequality}
	Let $h \colon [0,\infty) \to [0,\infty)$ be a increasing function and $r > 1$. Then, for any $\xi, \eta \in \R^N$
	\begin{align*}
		\left( h( \abs{\xi} ) \abs{\xi}^{r-2} \xi - h( \abs{\eta} ) \abs{\eta}^{r-2} \eta \right)
		\cdot \left( \xi - \eta \right) \geq
		C_r \abs{ \xi - \eta }^r h(m)
	\end{align*}
	if $r \geq 2$, and
	\begin{align*}
		\left( \abs{\xi} + \abs{\eta} \right)^{2 - r} \left( h( \abs{\xi} ) \abs{\xi}^{r-2} \xi - h( \abs{\eta} ) \abs{\eta}^{r-2} \eta \right)
		\cdot \left( \xi - \eta \right)  \geq
		C_r \abs{\xi - \eta}^2 h(m)
	\end{align*}
	if $1 < r < 2$, where $m = \min \{ \abs{\xi}, \abs{\eta} \}$ and
	\begin{align*}
		C_r =
		\begin{cases}
			\min \{ 2^{2-r}, 2^{-1} \} & \text{if } r \geq 2,  \\
			r-1                        & \text{if } 1 < r < 2.
		\end{cases}
	\end{align*}
\end{lemma}
\begin{proof}
	For the case $r \geq 2$, we obtain the identity
	\begin{align*}
		 & \left( h( \abs{\xi} ) \abs{\xi}^{r-2} \xi - h( \abs{\eta} ) \abs{\eta}^{r-2} \eta \right)
		\cdot \left( \xi - \eta \right) \\
		 & = \left( h( \abs{\xi} ) \abs{\xi}^{r-2} \xi \right)
		\cdot \left( \xi - \eta \right) +
		\left( - h( \abs{\eta} ) \abs{\eta}^{r-2} \eta \right)
		\cdot \left( \xi - \eta \right) \\
		 & = h( \abs{\xi} ) \abs{\xi}^{r-2} \left( \frac{1}{2} \xi - \frac{1}{2} \eta \right) \cdot \left( \xi - \eta \right)
		+ h( \abs{\eta} ) \abs{\eta}^{r-2} \left( \frac{1}{2} \xi - \frac{1}{2} \eta \right) \cdot \left( \xi - \eta \right)        \\
		 & \quad + h( \abs{\xi} ) \abs{\xi}^{r-2} \left( \frac{1}{2} \xi + \frac{1}{2} \eta \right) \cdot \left( \xi - \eta \right)
		+ h( \abs{\eta} ) \abs{\eta}^{r-2} \left( - \frac{1}{2} \xi - \frac{1}{2} \eta \right) \cdot \left( \xi - \eta \right)      \\
		 & = \frac{1}{2} \left( h( \abs{\xi} ) \abs{\xi}^{r-2} + h( \abs{\eta} ) \abs{\eta}^{r-2} \right)
		\abs{ \xi - \eta }^2  \\
		 & \quad + \frac{1}{2} \left( h( \abs{\xi} ) \abs{\xi}^{r-2} - h( \abs{\eta} ) \abs{\eta}^{r-2} \right)
		\left( \abs{\xi}^2 - \abs{\eta}^2 \right) .
	\end{align*}
	Since $h$ is an increasing function, the second term is nonnegative if $r \geq 2$. Hence the inequality follows when $r \geq 2.$ This is the same strategy in which one proves the usual inequality without $h$ (see for example, Chapter 12, (I) in the book by Lindqvist \cite{Lindqvist-2019}). From here the inequality follows.

	For the case $1 < r < 2$, we follow the argument of equation (2.2) from the paper by Simon \cite{Simon-1978}. The main difference now is that the expression is nonhomogeneous and therefore cannot be scaled, but it is still invariant under rotations. For this reason it is enough to consider the case $\abs{\xi} \geq \abs{\eta}$, $\xi = \abs{\xi} e_1$, $\eta= \eta_1 e_1 + \eta_2 e_2$. We split the argument in two cases. First, if $\eta_1 \leq 0$
	\begin{align*}
		\left( h( \abs{\xi} ) \abs{\xi}^{r-1} - h( \abs{\eta} ) \abs{\eta}^{r-2} \eta_1 \right)
		\geq h( \abs{\eta} ) \abs{\xi}^{r-2} \left( \abs{\xi} - \eta_1 \right)
	\end{align*}
	and if $0 \leq \eta_1 \leq \abs{\xi}$, by the mean value theorem
	\begin{align*}
		\left( h( \abs{\xi} ) \abs{\xi}^{r-1} - h( \abs{\eta} ) \abs{\eta}^{r-2} \eta_1 \right)
		 & \geq h( \abs{\xi} ) \left( \abs{\xi}^{r-1} - \eta_1^{r-1} \right) \\
		 & \geq h( \abs{\eta} ) (r-1) \abs{\xi}^{r-2} \left( \abs{\xi} - \eta_1 \right).
	\end{align*}
	Altogether, this yields
	\begin{align*}
		 & \left( h( \abs{\xi} ) \abs{\xi}^{r-2} \xi - h( \abs{\eta} ) \abs{\eta}^{r-2} \eta \right)
		\cdot \left( \xi - \eta \right)   \\
		 & = \left( h( \abs{\xi} ) \abs{\xi}^{r-1} - h( \abs{\eta} ) \abs{\eta}^{r-2} \eta_1 \right) \left( \abs{\xi} - \eta_1 \right)
		+ h( \abs{\eta} ) \abs{\eta}^{r-2} \eta_2^2  \\
		 & \geq  h( \abs{\eta} ) (r-1) \left( \abs{\xi} + \abs{\eta} \right) ^{r-2} \left( [\abs{\xi} - \eta_1]^2 + \eta_2^2 \right)   \\
		 & = h( \abs{\eta} ) (r-1) \left( \abs{\xi} + \abs{\eta} \right) ^{r-2} \abs{\xi - \eta}^2.
	\end{align*}
\end{proof}

The second lemma is concerned with a version of Young's inequality specially tailored for our line of work. It becomes indispensable in the proof of the \textnormal{(S$_+$)}-property.

\begin{lemma}[Young's inequality for the product of a power-law and a logarithm] \label{Le:YoungIneqLog}
	Let $s,t \geq 0$, $r > 1$ then
	\begin{align*}
		s t^{r-1} \left[ \log (e + t ) + \frac{ t }{r (e + t)} \right]
		\leq \frac{s^r}{r} \log (e + s )
		+ t^r \left[ \frac{r - 1}{r} \log (e + t ) + \frac{ t }{r (e + t)} \right] .
	\end{align*}
\end{lemma}

\begin{proof}
	The result is a consequence of the general version of Young's inequality for functions that are positive, continuous, strictly increasing and vanish at zero, see for example Theorem 156 of the classical book by Hardy--Littlewood--P\'olya \cite{Hardy-Littlewood-Polya-1934}. Let $h \colon [0, \infty) \to [0, \infty)$ satisfy all of the above, then for any $s,t \geq 0$ it holds
	\begin{align*}
		s h(t) \leq \int_0^s h(y) \dy + \int_0^{h(t)} h^{-1}(y) \dy.
	\end{align*}
	We also make use of another general result about the primitive of the inverse of a function. For $h$ as above and $t>0$ it holds
	\begin{align*}
		\int_0^{h(t)}  h^{-1}(y) \dy = t h(t) - \int_0^t h(y) \dy.
	\end{align*}
	Choosing
	\begin{align*}
		h(t) = t^{r-1} \left[ \log (e + t ) + \frac{ t }{r(e + t)} \right] ,
	\end{align*}
	hence $\int_0^t h(y) \dy = t^r/r \log (e + t)$, one obtains the desired result.
\end{proof}

Now we can state the main properties of the operator $A$.

\begin{theorem}
	\label{Th:PropertiesOperator}
	Let \eqref{Assump:Space} be satisfied, then the operator $A$ is bounded, continuous and strictly monotone. If we further assume \eqref{Assump:Poincare}, then it also is of type \textnormal{(S$_+$)}, coercive and a homeomorphism.
\end{theorem}

\begin{proof}
	The continuity follows directly from Theorem \ref{Th:C1functional}. The strict monotonicity follows from Lemma \ref{Le:MonotoneInequality} and the widely-known inequality
	\begin{align*}
		\left( \abs{\xi}^{r-2}\xi -\abs{\eta}^{r-2}\eta \right) \cdot (\xi-\eta) > 0 \quad\text{if } r >1 \text{ for all }\xi,\eta \in \R^N \text{ with }\xi\neq\eta,
	\end{align*}
	since for $u,v \in \WHlogzero$ with $u \neq v$ they imply
	\begin{align*}
		\left\langle A(u) - A (v) , u - v \right\rangle
		\geq \into \left( \abs{\nabla u} ^{p(x)-2}\nabla u - \abs{\nabla v}^{p(x)-2}\nabla v \right)  \cdot \left( \nabla u - \nabla v \right) \dx > 0.
	\end{align*}
	For the boundedness, let us take $u,v \in \WHlogzero \setminus \{0\}$ with $\normoneHlog{v} = 1$. Then, use \eqref{Eq:LogGrowthProduct} in the case $\normoneHlog{u} \geq 1$ and that $t \mapsto \log (e + t)$ is increasing and $\normoneHlog{u} ^{1 - p_-} \leq \normoneHlog{u} ^{1 - q_-} $ in the case $\normoneHlog{u} \leq 1$. Together with Young's inequality it holds
	\begin{align*}
		&\min \left\lbrace  \normoneHlog{u}^{- q_+} , \normoneHlog{u}^{1 - p_-} \right\rbrace  \abs{ \langle A(u),v \rangle }\\
		& \leq  \min \left\lbrace  \normoneHlog{u}^{- q_+} , \normoneHlog{u}^{1 - p_-} \right\rbrace \\
		& \qquad \times\into \left( |\nabla u|^{p(x)-1} \abs{ \nabla v } \mathop{+} \mu(x) \left[ \log (e + \abs{ \nabla u }) + 1 \right]  |\nabla u|^{q(x)-1} \abs{ \nabla v } \right) \dx \\
		& \leq \into  \abs{ \frac{\nabla u}{ \normoneHlog{u} } } ^{p(x)-1} \abs{ \nabla v } \dx \\
		& \quad + \into \left(\mu(x)   \left[  \log \left( e + \frac{\abs{ \nabla u }}{ \normoneHlog{u} } \right)  + 1 \right]  \abs{ \frac{\nabla u}{ \normoneHlog{u} } } ^{q(x)-1}  \abs{ \nabla v } \right) \dx \\
		& \leq \frac{p_+-1}{p_-} \into \abs{ \frac{\nabla u}{ \normoneHlog{u} } } ^{p(x)} \dx+ \frac{1}{p_-} \into \abs{ \nabla v }^{p(x)} \dx\\
		& \quad +\frac{q_+-1}{q_-} \into \mu(x) \left[  \log \left( e + \frac{\abs{ \nabla u }}{ \normoneHlog{u} } \right)  + 1 \right] \left|\frac{\nabla u}{ \normoneHlog{u} }\right|^{q(x)} \dx                 \\
		& \quad + \frac{1}{q_-} \into \mu(x) \left[  \log \left( e + \frac{\abs{ \nabla u }}{ \normoneHlog{u} } \right)  + 1 \right] \abs{ \nabla v } ^{q(x)} \dx.
	\end{align*}
	We can estimate the last addend by splitting $\Omega$ in the set where $\abs{ \nabla v }$ is greater than $ \abs{ \nabla u } / \normoneHlog{u} $ and its complement. This together with Proposition \ref{Prop:oneHlogModularNorm} \textnormal{(i)} leads to
	\begin{align*}
		& \min \left\lbrace  \normoneHlog{u}^{- q_+} , \normoneHlog{u}^{1 - p_-} \right\rbrace  \abs{ \langle A(u),v \rangle } \\
		& \leq \max \left\lbrace \frac{q^+}{q_-}, \frac{p_+-1}{p_-} \right\rbrace  \modHlog{ \frac{\nabla  u}{ \normoneHlog{u} }}+\frac{1}{p_-} \modHlog{ \nabla v }   \\
		& \leq   \max \left\lbrace \frac{q^+}{q_-}, \frac{p_+-1}{p_-} \right\rbrace\modoneHlog{ \frac{  u}{ \normoneHlog{u} }}+ \frac{1}{p_-} \modoneHlog{ v }  \\
		& = \max \left\lbrace \frac{q^+}{q_-}, \frac{p_+-1}{p_-} \right\rbrace + \frac{1}{p_-} = \max \left\lbrace  \frac{q_+ p_- + q_-}{q_- p_-} , \frac{p_+}{p_-}  \right\rbrace.
	\end{align*}
	As a consequence
	\begin{align*}
		\|A(u)\|_*
		 & = \sup_{ \normoneHlog{v} = 1 } \left\langle A(u) , v \right\rangle                             \\
		 & \leq \max \left\lbrace  \frac{q_+ p_- + q_-}{q_- p_-} , \frac{p_+}{p_-}  \right\rbrace
		\max \left\lbrace  \normoneHlog{u}^{q_+} , \normoneHlog{u}^{p_- - 1} \right\rbrace.
	\end{align*}

	From now on we assume that \eqref{Assump:Poincare} holds for the proof of the remaining properties. Now we will deal with the \textnormal{(S$_+$)}-property. Let $\{u_n\}_{n \in \N} \subseteq \WHlogzero$ be a sequence such that
	\begin{align}
		\label{Eq:SequenceSplus}
		u_n \weak u \quad\text{in }  \WHlogzero
		\quad\text{and}\quad
		\limsup_{n \to \infty}{\left\langle A(u_n), u_n - u \right\rangle } \leq 0.
	\end{align}
	By the strict monotonicity of $A$ and the weak convergence of $u_n$, we obtain
	\begin{align*}
		0
		\leq \liminf_{n \to \infty} \left\langle A(u_n) - A(u) , u_n - u \right\rangle
		 & \leq \limsup_{n \to \infty}\left\langle A(u_n) - A(u) , u_n - u \right\rangle \\
		 & = \limsup_{n \to \infty}\left\langle A(u_n) , u_n - u \right\rangle
		\leq 0,
	\end{align*}
	which means
	\begin{align*}
		\lim_{n \to \infty} \left\langle A(u_n) - A(u) , u_n - u \right\rangle
		= 0.
	\end{align*}

	{\bf Claim: } $\nabla u_n \to \nabla u$ in measure.

	In particular, as the previous expression can be decomposed in the sum of nonnegative terms, it follows
	\begin{align}
		\label{Eq:CasePgeq2}
		\lim_{n \to \infty} \int_{ \{p \geq 2 \}}
		\left( \abs{\nabla u_n}^{p(x)-2} \nabla u_n - \abs{\nabla u}^{p(x)-2} \nabla u \right) \cdot \left( \nabla u_n - \nabla u \right) \dx = 0,
		\\
		\label{Eq:CaseP<2}
		\lim_{n \to \infty} \int_{ \{p < 2 \}}
		\left( \abs{\nabla u_n}^{p(x)-2} \nabla u_n - \abs{\nabla u}^{p(x)-2} \nabla u \right) \cdot \left( \nabla u_n - \nabla u \right) \dx = 0.
	\end{align}
	From Lemma \ref{Le:MonotoneInequality} with $h = 1$ and \eqref{Eq:CasePgeq2}, we can directly derive that
	\begin{align*}
		\lim_{n \to \infty} \int_{ \{p \geq 2 \}} \abs{\nabla u_n - \nabla u}^{p(x)} \dx = 0,
	\end{align*}
	hence $ \nabla u_n 1_{ \{p \geq 2 \}} \to \nabla u 1_{ \{p \geq 2 \}}$ in measure. On the other hand, let
	\begin{align*}
		E_n = \{ \nabla u_n \neq 0\} \cup \{ \nabla u \neq 0\},
	\end{align*}
	then for any $\eps > 0$ we know that
	\begin{align*}
		 & \left\lbrace 1_{ \{p < 2 \}} (p_- - 1) \abs{\nabla u_n - \nabla u}^2 1_{E_n} \left( \abs{\nabla u_n} + \abs{\nabla u} \right)^{p(x) - 2} \geq \eps \right\rbrace   \\
		 & \subseteq \left\lbrace 1_{ \{p < 2 \}} \left( \abs{\nabla u_n}^{p(x)-2} \nabla u_n - \abs{\nabla u}^{p(x)-2} \nabla u \right) \cdot \left( \nabla u_n - \nabla u \right) \geq \eps \right\rbrace.
	\end{align*}
	From \eqref{Eq:CaseP<2} and the previous expression we obtain that
	\begin{align*}
		1_{ \{p < 2 \}} \abs{\nabla u_n - \nabla u}^2 1_{E_n} \left( \abs{\nabla u_n} + \abs{\nabla u} \right)^{p(x) - 2} \rightarrow 0 \quad \text{in measure.}
	\end{align*}
	Hence the same convergence is true a.e.\,along a subsequence $u_{n_k}$ and then for a.a.\,$x \in \Omega$ there exists $M(x) > 0$ such that for all $k \in \N$
	\begin{align*}
		M(x) & \geq 1_{ \{p < 2 \}} \abs{\nabla u_{n_k} - \nabla u}^2 1_{E_{n_k}} \left( \abs{\nabla u_{n_k}} + \abs{\nabla u} \right)^{p(x) - 2} \\
		& \geq 1_{ \{p < 2 \}} \abs{ \abs{\nabla u_{n_k}} - \abs{\nabla u} }^2 1_{E_{n_k}} \left( \abs{\nabla u_{n_k}} + \abs{\nabla u} \right)^{p(x) - 2}.
	\end{align*}
	Note that, given any $c > 0$ and $0 < P < 1$, the function $h(t) = \abs{t - c}^2 (t + c)^{P-2}$ satisfies $\lim_{t \to +\infty} h(t) = +\infty$. Therefore, there exists $m(x) > 0$ such that $\abs{ \nabla u_{n_k} } \leq m(x)$ for a.a. $x \in \Omega$ and all $k \in \N$. As a consequence
	\begin{align*}
		& 1_{ \{p < 2 \}} \abs{\nabla u_{n_k} - \nabla u}^2 1_{E_{n_k}} \left( \abs{\nabla u_{n_k}} + \abs{\nabla u} \right)^{p(x) - 2} \\
		& \geq 1_{ \{p < 2 \}} \abs{\nabla u_{n_k} - \nabla u}^2 1_{E_{n_k}} \left( m(x) + \abs{\nabla u} \right)^{p(x) - 2}
	\end{align*}
	and the convergence a.e.\,to zero of the left-hand side
	yields $1_{ \{p < 2 \}} \abs{ \nabla u_{n_k} - \nabla u } \to 0$ a.e., since
	\begin{align*}
		& \text{ for all } x \in \Omega \text{ such that } \nabla u (x) \neq 0, 1_{E_{n_k}} (x) = 1 \text{ for all } k \in \N; \\
		& \text{ for all } x \in \Omega \text{ such that } \nabla u (x) = 0, \text{ along any subsequence such that }\\
		& \abs{\nabla u_{n_{k'}} (x)} > \eps > 0 \text{ it holds that }1_{E_{n_{k'}}} (x) = 1 \text{ for all } k' \in \N.
	\end{align*}
	Then $1_{ \{p < 2 \}} \abs{ \nabla u_{n_k} - \nabla u } \to 0$ in measure and by the subsequence principle, this is also true for the whole sequence $u_n$. Together with the case on $\{p \geq 2 \}$, this finishes the proof of the Claim.

	By the usual Young's inequality and Lemma \ref{Le:YoungIneqLog} it follows
	\begin{align*}
			& \into\abs{\nabla u_n}^{p(x)-2}\nabla u_n \cdot \nabla (u_n - u )  \dx\\
			&
			\quad+\into  \mu(x) \left[ \log (e + \abs{\nabla u_n} ) + \frac{\abs{\nabla u_n}}{q(x) (e + \abs{\nabla u_n})} \right]
			\abs{\nabla u_n}^{q(x)-2}\nabla u_n  \cdot\nabla (u_n - u ) \dx\\
			& = \into \abs{\nabla u_n}^{p(x)} \dx- \into \abs{\nabla u_n}^{p(x)-2}\nabla u_n \cdot \nabla u \dx\\
			& \quad + \into \mu(x) \left[ \log (e + \abs{\nabla u_n} ) + \frac{\abs{\nabla u_n}}{q(x) (e + \abs{\nabla u_n})} \right] \abs{\nabla u_n}^{q(x)} \dx \\
			& \quad - \into \mu(x) \left[ \log (e + \abs{\nabla u_n} ) + \frac{\abs{\nabla u_n}}{q(x) (e + \abs{\nabla u_n})} \right] \abs{\nabla u_n}^{q(x)-2}\nabla u_n \cdot \nabla u \dx\\
			& \geq \into \abs{\nabla u_n}^{p(x)} \dx- \into \abs{\nabla u_n}^{p(x)-1} \abs{\nabla u} \dx\\
			& \quad + \into \mu(x) \left[ \log (e + \abs{\nabla u_n} ) + \frac{\abs{\nabla u_n}}{q(x) (e + \abs{\nabla u_n})} \right] \abs{\nabla u_n}^{q(x)} \dx \\
			& \quad - \into \mu(x) \left[ \log (e + \abs{\nabla u_n} ) + \frac{\abs{\nabla u_n}}{q(x) (e + \abs{\nabla u_n})} \right] \abs{\nabla u_n}^{q(x)- 1} \abs{\nabla u} \dx\\
			& \geq \into \abs{\nabla u_n}^{p(x)} \dx - \into \left( \frac{p(x)-1}{p(x)}\abs{\nabla u_n}^{p(x)}+\frac{1}{p(x)} \abs{\nabla u}^{p(x)} \right) \dx\\
			& \quad + \into \mu(x) \left[ \log (e + \abs{\nabla u_n} ) + \frac{\abs{\nabla u_n}}{q(x) (e + \abs{\nabla u_n})} \right] \abs{\nabla u_n}^{q(x)} \dx \\
			& \quad - \into \mu(x)\left( \left[ \frac{q(x)-1}{q(x)} \log (e + \abs{\nabla u_n} ) + \frac{\abs{\nabla u_n}}{q(x) (e + \abs{\nabla u_n})} \right] \abs{\nabla u_n}^{q(x)} \right. \\
			& \quad \left. + \frac{1}{q(x)}\abs{\nabla u}^{q(x)} \log (e + \abs{ \nabla u }) \right) \dx\\
			& = \into \frac{1}{p(x)}\abs{\nabla u_n}^{p(x)} \dx- \into \frac{1}{p(x)} \abs{\nabla u}^{p(x)} \dx\\
			& \quad + \into \frac{\mu(x)}{q(x)}\abs{\nabla u_n}^{q(x)} \log (e + \abs{ \nabla u_n }) \dx- \into \frac{\mu(x)}{q(x)}\abs{\nabla u}^{q(x)} \log (e + \abs{ \nabla u }) \dx .
	\end{align*}
	As a consequence, by \eqref{Eq:SequenceSplus}
	\begin{align*}
		 & \limsup_{n \to \infty} \int_\Omega \left(\frac{\abs{\nabla u_n}^{p(x)}}{p(x)} + \mu(x) \frac{ \abs{\nabla u_n}^{q(x)}}{q(x)} \log (e + \abs{ \nabla u_n }) \right) \dx \\
		 & \leq \int_\Omega \left(\frac{\abs{\nabla u}^{p(x)}}{p(x)} + \mu(x) \frac{ \abs{\nabla u}^{q(x)}}{q(x)} \log (e + \abs{ \nabla u }) \right) \dx.
	\end{align*}
	By Fatou's Lemma one can obtain that the limit inferior satisfies the opposite inequality, so all in all
	\begin{align*}
		 & \lim_{n \to \infty} \int_\Omega \left(\frac{\abs{\nabla u_n}^{p(x)}}{p(x)} + \mu(x) \frac{ \abs{\nabla u_n}^{q(x)}}{q(x)} \log (e + \abs{ \nabla u_n }) \right) \dx \\
		 & = \int_\Omega \left(\frac{\abs{\nabla u}^{p(x)}}{p(x)} + \mu(x) \frac{ \abs{\nabla u}^{q(x)}}{q(x)} \log (e + \abs{ \nabla u }) \right) \dx.
	\end{align*}
	By the previous Claim, passing to a.e.\,convergence along a subsequence and using the subsequence principle, we can prove that the integrand of the left-hand side converges in measure to the integrand of the right-hand side. Then we can derive from the so-called converse of Vitali's theorem its $L^1$ convergence, and in particular the uniform integrability of the sequence
	\begin{align*}
		\left\lbrace \frac{\abs{\nabla u_n}^{p(x)}}{p(x)} + \mu(x) \frac{ \abs{\nabla u_n}^{q(x)}}{q(x)} \log (e + \abs{ \nabla u_n }) \right\rbrace _{n \in \N} .
	\end{align*}
	On the other hand, by \eqref{Eq:LogGrowthSum} we know that
	\begin{align*}
		 & \abs{\nabla u_n - \nabla u}^{p(x)} + \mu(x)  \abs{\nabla u_n - \nabla u}^{q(x)} \log (e + \abs{ \nabla u_n - \nabla u }) \\
		 & \leq 2 ^{q_+ + 1} q_+ \left(  \frac{\abs{ \nabla u_n }^{p(x)}}{p(x)} + \frac{\abs{ \nabla u }^{p(x)}}{p(x)} \right.  \\
		 & \phantom{\leq 2 ^{q_+ + 1} q_+ } \quad \left. + \frac{ \abs{ \nabla u_n }^{q(x)} }{q(x)} \log(e + \abs{ \nabla u_n } ) + \frac{ \abs{ \nabla u }^{q(x)} }{q(x)} \log(e + \abs{ \nabla u } )  \right),
	\end{align*}
	which yields the uniform integrability of the sequence
	\begin{align*}
		\left\lbrace \abs{\nabla u_n - \nabla u}^{p(x)} + \mu(x)  \abs{\nabla u_n - \nabla u}^{q(x)} \log (e + \abs{ \nabla u_n - \nabla u }) \right\rbrace _{n \in \N} .
	\end{align*}
	As above, by the Claim, passing to a.e.\,convergence along a subsequence and using the subsequence principle, we can prove that this sequence converges in measure to zero. By Vitali's theorem, these two facts imply
	\begin{align*}
		 & \lim_{n \to \infty} \modHlog{ \nabla u_n - \nabla u } \\
		 & = \lim_{n \to \infty} \into \left( \abs{\nabla u_n - \nabla u}^{p(x)} + \mu(x)  \abs{\nabla u_n - \nabla u}^{q(x)} \log (e + \abs{ \nabla u_n - \nabla u }) \right) \dx
		= 0.
	\end{align*}
	By Proposition \ref{Prop:HlogModularNorm} \textnormal{(v)}, this is equivalent to $\normoneHlogzero{ u_n - u } = \normHlog{ \nabla u_n - \nabla u } \to 0$, i.e.\,by Proposition \ref{Prop:Poincare} we know that $u_n \to u$ in $\WHlogzero$.

	Let us now prove that the operator is coercive. For any $u \in \WHlog$ with $\normoneHlogzero{u} \geq 1$,  using the fact
	that $t \mapsto \log (e + t)$ is increasing, by Proposition \ref{Prop:HlogModularNorm} \textnormal{(i)} it follows
	\begin{align*}
		\frac{\left\langle A(u) , u \right\rangle}{ \normoneHlogzero{u} }
		 & = \into \normoneHlogzero{u}^{p(x)-1} \left( \frac{\abs{\nabla u}}{ \normoneHlogzero{u} } \right)^{p(x)} \dx \\
		 & \quad +\into \normoneHlogzero{u}^{q(x)-1} \mu(x) \left( \frac{\abs{\nabla u}}{ \normoneHlogzero{u} } \right)^{q(x)} \log (  e + \abs{\nabla u} )  \dx \\
		 & \geq \normoneHlogzero{u}^{p_- - 1}
		\modoneHlogzero{ \frac{\nabla u}{ \normoneHlogzero{u} } } = \normoneHlogzero{u}^{p_- - 1}  \to + \infty
	\end{align*}
	as $\normoneHlogzero{u} \to + \infty$.

	Finally, we show that $A$ is a homeomorphism. By the previously proven properties and the Minty-Browder theorem (see, for example, Zeidler \cite[Theorem 26.A]{Zeidler-1990}), we know that $A$ is invertible and that $A^{-1}$ is strictly monotone, demicontinuous and bounded. It is only left to see that $A^{-1}$ is continuous. With this purpose in mind, let $\{y_n\}_{n \in \N} \subseteq [ \WHlogzero ] ^*$ be a sequence such that $y_n \to y$ in $[ \WHlogzero ] ^*$ and let $u_n=A^{-1}(y_n)$ as well as $u=A^{-1}(y)$. The strong convergence of $\{y_n\}_{n \in \N}$ and the boundedness of $A^{-1}$ imply that $u_n$
	is bounded in $\WHlogzero$. Thus, there exists a subsequence $\{u_{n_k}\}_{k\in\N}$ such that $u_{n_k} \rightharpoonup u_0$ in $\WHlogzero$. All these properties yield
	\begin{align*}
		 & \lim_{k\to \infty} \left\langle A(u_{n_k})-A(u_0),u_{n_k}-u_0\right\rangle \\
		 & = \lim_{k \to \infty} \left\langle y_{n_k}-y,u_{n_k}-u_0\right\rangle+\lim_{k \to
			\infty} \left\langle y-A(u_0),u_{n_k}-u_0\right\rangle=0.
	\end{align*}
	By the \textnormal{(S$_+$)}-property of $A$ we obtain that $u_{n_k}\to u_0$ in $\WHlogzero$. The operator $A$ is also continuous, so
	\begin{align*}
		A(u_0)=\lim_{k \to \infty} A(u_{n_k})=\lim_{k \to \infty} y_{n_k}=y=A(u).
	\end{align*}
	As $A$ is injective, this proves that $u=u_0$. By the subsequence principle we can show that this convergence holds for the whole sequence.
\end{proof}

When we consider the operator $\tilde{A} \colon \WHlog \to \left[ \WHlog \right] ^*$ given by the same expression as $A$, one has the following result.

\begin{theorem}
	Let \eqref{Assump:Space} be satisfied, then the operator $\tilde{A}$ is bounded, continuous and monotone. If we further assume \eqref{Assump:Poincare}, then it also is of type \textnormal{(S$_+$)}.
\end{theorem}

\begin{proof}
	For the \textnormal{(S$_+$)}-property, repeat the same argument as before for the gradient part together with the compact embedding $\WHlog \hookrightarrow \LHlog$ from Proposition \ref{Prop:Poincare}. The rest of the assertions follow in the same way as before.
\end{proof}

\begin{remark}
	As seen in the previous theorem, the operator $\tilde{A}$ has weaker properties than $A$. However, by choosing the space appropriately one can recover the original strong properties. For example, one could define the operator $\tilde{A}$ on the space of zero mean functions, i.e.
	\begin{align*}
		L^{\hlog}_{\langle \cdot \rangle} (\Omega) &= \left\lbrace u \in \LHlog \colon \into u \dx = 0 \right\rbrace , \\
		W^{1,\hlog}_{\langle \cdot \rangle} (\Omega) &= \left\lbrace u \in \WHlog \colon \into u \dx = 0 \right\rbrace.
	\end{align*}
	Following analogous steps as in Proposition \ref{Prop:Poincare} it is possible to prove the compact embedding $W^{1,\hlog}_{\langle \cdot \rangle} (\Omega) \hookrightarrow L^{\hlog}_{\langle \cdot \rangle} (\Omega)$ and the Poincar\'{e} inequality. As a consequence, one can repeat all the steps from Theorem \ref{Th:PropertiesOperator} for $\tilde{A}$ and obtain exactly the same properties.
\end{remark}

For completeness, we consider the operator $B \colon \WHlog \to \left[ \WHlog \right] ^*$ given by
\begin{align*}
	\left\langle B(u),v \right\rangle
	 & = \into \abs{\nabla u}^{p(x)-2}\nabla
	u \cdot \nabla v  \vphantom{\frac{\abs{\nabla u}}{q(x) (e + \abs{\nabla u})}}\dx\\
	 & \quad +\into \mu(x) \left[ \log (e + \abs{\nabla u} ) + \frac{\abs{\nabla u}}{q(x) (e + \abs{\nabla u})} \right]  \abs{\nabla u}^{q(x)-2}\nabla u \cdot \nabla v\dx \\
	 & \quad + \into \abs{ u}^{p(x)-2}
	u v  \dx\\
	 & \quad +\into \mu(x) \left[ \log (e + \abs{ u} ) + \frac{\abs{ u}}{q(x) (e + \abs{ u})} \right]  \abs{ u}^{q(x)-2} u  v\dx,
\end{align*}
and the operator $\tilde{B} \colon \WHlogzero \to \left[ \WHlogzero \right] ^*$ given by the same expression. Then we have the following result without the need for the condition \eqref{Assump:Poincare}.

\begin{theorem}
	Let \eqref{Assump:Space} be satisfied, then the operators $B$ and $\tilde{B}$ are bounded, continuous, strictly monotone, of type \textnormal{(S$_+$)}, coercive and homeomorphisms.
\end{theorem}

\begin{proof}
	Repeat the arguments of Theorem \ref{Th:PropertiesOperator} separately for the part with the gradients and the part without the gradients.
\end{proof}

One direct consequence of the previous theorems is the following existence and uniqueness result. We say that $u \in \WHlogzero$ is a weak solution of \eqref{Eq:Problem} if for all $v \in \WHlogzero$ it holds that
\begin{align*}
	& \into \abs{\nabla u}^{p(x)-2}\nabla u \cdot \nabla v \dx\\
	& \quad +\into \mu(x) \left[ \log (e + \abs{\nabla u} ) + \frac{\abs{\nabla u}}{q(x) (e + \abs{\nabla u})} \right]  \abs{\nabla u}^{q(x)-2}\nabla u  \cdot \nabla v \dx \\
	& = \into f(x,u) v \dx.
\end{align*}

\begin{theorem}
	Let \eqref{Assump:Poincare} hold and $f \in \Lp{\frac{ r(\cdot) }{ r(\cdot) - 1 }}$, where either $r \in C_+ (\close)$ and $r(x) < p^*(x)$ for all $x \in \close$ or $r \in C_+ (\close) \cap C^{0, \frac{1}{ |\log t| }} (\close)$ and $r(x) \leq p^*(x)$ for all $x \in \close$. Then there exists a unique weak solution $u \in \WHlogzero$ of problem \eqref{Eq:Problem}.
\end{theorem}

\begin{proof}
	As by Proposition \ref{Prop:EmbeddingHlogSobolev} \textnormal{(ii)} and \textnormal{(iii)} $\WHlogzero \hookrightarrow \Lp{r(\cdot)}$, then $f \in \Lp{r'(\cdot)} = \left[ \Lp{r(\cdot)} \right] ^* \hookrightarrow \left[ \WHlogzero \right] ^*$. The result follows because $A$ is a bijection by Theorem \ref{Th:PropertiesOperator}.
\end{proof}

\section{Constant sign solutions}\label{constant-sign-solutions}

After the comprehensive list of properties of the function space and the logarithmic double phase operator given in the previous sections, we are now in the position to prove our first multiplicity result for problem \eqref{Eq:Problem}. We impose the following assumptions on the right-hand side $f$ and some slightly stricter requirements on the exponents than on \eqref{Assump:Poincare}:

\begin{enumerate}[label=\textnormal{(H$_2$)},ref=\textnormal{H$_2$}]
	\item\label{Assump:H2}
		$\Omega \subseteq \R^N$, with $N \geq 2$, is a bounded domain with Lipschitz boundary $\partial \Omega$, $p,q \in C_+ (\close)$ with $p(x) \leq q(x) \leq q_+ < p^*_-$ for all $x \in \close$ and $0\leq\mu(\cdot) \in \Lp{\infty}$.
\end{enumerate}

\begin{enumerate}[label=\textnormal{(H$_\textnormal{f}$)},ref=\textnormal{H$_\textnormal{f}$}]
	\item\label{Hf}
		Let $f \colon \Omega \times \R \to \R$ and $F(x,t) = \int_{0}^{t} f(x,s) \ds$.
		\begin{enumerate} [label=\textnormal{(f$_{\arabic*}$)},ref=\textnormal{f$_{\arabic*}$}]
			\item\label{Asf:Caratheodory}
				The function $f$ is Carath\'eodory type, i.e.\,$t \mapsto f(x,t)$ is continuous for a.a.\,$x \in \Omega$ and $x \mapsto f(x,t)$ is measurable for all $t \in \R$.
			\item\label{Asf:WellDef}
				There exists $r \in C_+(\close)$ with $r_+ < p^*_-$ and $C>0$ such that
				\begin{align*}
					\abs{f(x,t)} \leq C \left( 1 + \abs{t}^{r(x)-1} \right)
				\end{align*}
				for a.a.\,$x \in \Omega$ and for all $t \in \R$.
			\item\label{Asf:GrowthInfty}
				\begin{align*}
					\lim\limits_{s \to \pm \infty} \frac{F(x,s)}{\abs{s}^{q_+} \log(e + \abs{s}) } = + \infty \quad \text{uniformly for a.a.\,} x \in \Omega.
				\end{align*}
			\item\label{Asf:GrowthZero}
				There exists $\theta > 0$ such that $F(x,t) \leq 0$ for $\abs{t} \leq \theta$ and for a.a.\,$x \in \Omega$ and $f(x,0)=0$ for a.a.\,$x \in \Omega$.
			\item\label{Asf:CeramiAssumption}
				There exists $l, \widetilde{l} \in C_+(\close)$ such that $\min \{l_-, \widetilde{l}_- \} \in \left( (r_+ - p_-) \frac{N}{p_-} , r_+ \right) $ and $K > 0$ with
				\begin{align*}
					0 < K \leq \liminf_{s \to + \infty} \frac{f(x,s)s - q_+ \left( 1 + \frac{\kappa}{q_-} \right) 	F(x,s)}{\abs{s}^{l(x)}}
				\end{align*}
				uniformly for a.a.\,$ x\in \Omega$, and
				\begin{align*}
					0 < K \leq \liminf_{s \to - \infty} \frac{f(x,s)s - q_+ \left( 1 + \frac{\kappa}{q_-} \right) 	F(x,s)}{\abs{s}^{\widetilde{l}(x)}}
				\end{align*}
				uniformly for a.a.\,$x \in \Omega$, where $\kappa$ is the same one as in Lemma \ref{Le:fepsilon}.
		\end{enumerate}
\end{enumerate}

As a consequence the function $f$ has the following properties.

\begin{lemma}
	\label{Le:Propf}
	Let $f \colon \Omega \times \R \to \R$.
	\begin{enumerate}[label=(\roman*),font=\normalfont]
		\item\label{Propf:q<r}
			If $f$ fulfills \eqref{Asf:WellDef} and  \eqref{Asf:GrowthInfty}, then $q_+ < r_-$.
		\item\label{Propf:Boundedbelow}
			If $f$ fulfills \eqref{Asf:WellDef} and \eqref{Asf:GrowthInfty}, then there exist some $M>0$ such that
			\begin{align*}
				F(x,t) > - M \quad \text{for a.a.\,} x \in \Omega \text{ and for all } t \in \R.
			\end{align*}
		\item\label{Propf:EpsilonUpperBoundF}
			If $f$ fulfills \eqref{Asf:WellDef} and \eqref{Asf:GrowthZero}, then there exists $C > 0$ such that
			\begin{align*}
				F(x,t) \leq C \abs{t}^{r(x)} \quad \text{for a.a.\,} x \in \Omega.
			\end{align*}
		\item\label{Propf:EpsilonLowerBoundF}
			If $f$ fulfills \eqref{Asf:WellDef} and \eqref{Asf:GrowthInfty}, then for each $\eps > 0$ there exists $C_\eps > 0$ such that
			\begin{align*}
				F(x,t) \geq \frac{\eps}{q_+} \abs{t}^{q_+} \log( e + \abs{t} ) - C_\eps  \quad \text{for a.a.\,} x \in \Omega.
			\end{align*}
		\item\label{Propf:ConvergenceFterm}
			If $f$ fulfills \eqref{Asf:Caratheodory} and \eqref{Asf:WellDef}, then the functional $I_f \colon \WHlogzero \to \R$ given by
			\begin{align*}
				I_f(u) = \into F(x,u) \dx
			\end{align*}
			and its derivative $I_f ' \colon \WHlogzero \to \left[ \WHlogzero \right] ^*$, given by
			\begin{align*}
				\left \langle I_f ' (u) , v \right \rangle = \into f(x,u) v \dx,
			\end{align*}
			are strongly continuous, i.e.\,$u_n \weak u$ in $\WHlogzero$ implies $I_f(u_n) \to I_f(u)$ in $\R$ and $I_f '(u_n) \to I_f ' (u)$ in $\left[ \WHlogzero \right] ^*$.
	\end{enumerate}
\end{lemma}

\begin{remark}
	As in \eqref{Assump:H2} we assumed $q_+ < p^*_-$, we can always find at least one $r \in C_+(\close)$ such that $q_+ < r_- \leq r_+ < p^+_-$.
\end{remark}

In order to find weak solutions we will work on the associated energy functional to problem \eqref{Eq:Problem}, since they coincide with the critical points of this functional. It is defined as the functional $\ph \colon \WHlogzero \to \R$ given by
\begin{align*}
	\ph(u)
	= \into \left(  \frac{\abs{\nabla u}^{p(x)}}{p(x)} + \mu(x) \frac{\abs{ \nabla u}^{q(x)}}{q(x)} \log( e + \abs{\nabla u} ) \right)  \dx  \;
	- \into F(x,u) \dx.
\end{align*}
In particular, as we are interested in constant sign solutions, we consider the truncated functionals $\ph_\pm \colon \WHlogzero \to \R$ defined by
\begin{align*}
	\ph_\pm (u)
	= \into \left(  \frac{\abs{\nabla u}^{p(x)}}{p(x)} + \mu(x) \frac{\abs{ \nabla u}^{q(x)}}{q(x)} \log( e + \abs{\nabla u} ) \right)  \dx  \;
	- \into F(x, \pm u^\pm ) \dx.
\end{align*}

\begin{remark}
	\label{Re:StrongCont}
	Note that by the second half of \eqref{Asf:GrowthZero}, $F(x, \pm t^\pm) = \int_0^t f(x, \pm s ^\pm) \ds$ for a.a.\,$x \in \Omega$ and all $t \in \R$. This will allow us to use Lemma \ref{Le:Propf} \ref{Propf:ConvergenceFterm} on the truncated functionals both for the differentiability and the strong continuity.
\end{remark}

In the remaining part of this section we check the assumptions of the mountain pass theorem, see Theorem \ref{Th:MPT}, for the truncated functionals $\ph_\pm$. We start by checking the compactness-type property. For this purpose, we first need the following lemma. Its proof is straightforward.

\begin{lemma}
	\label{Le:QuotientFracLog}
	Let $Q>1$ and $h \colon [0,\infty) \to [0,\infty)$ given by $h(t) = \frac{t}{Q (e + t) \log(e + t) }$. Then $h$ attains its maximum value at $t_0$ and the value is $\frac{\kappa}{Q}$, where $t_0$ and $\kappa$ are the same as in Lemma \ref{Le:fepsilon}.
\end{lemma}

\begin{proposition}
	\label{Prop:CeramiCondition}
	Let \eqref{Assump:H2} be satisfied and $f$ fulfill \eqref{Asf:Caratheodory}, \eqref{Asf:WellDef}, \eqref{Asf:GrowthZero} and \eqref{Asf:CeramiAssumption}. Then the functionals $\ph_\pm$ satisfy the \textnormal{C}-condition.
\end{proposition}

\begin{proof}
	Here we give the argument only for $\ph_+$, the case of $\ph_-$ is almost the same. Let $C_1 > 0$ and $\{u_n\}_{n \in \N} \subseteq \WHlogzero$ be a sequence such that
	\begin{align}
		\abs{\ph_+(u_n)} & \leq C_1 \quad \text{for all } n \in \N, \label{Eq:C-bound} \\
		\label{Eq:C-conv}
		(1 + \normoneHlogzero{u_n}) \ph_+ ' (u_n) & \to 0 \quad \text{in } \left[ \WHlogzero \right] ^*.
	\end{align}

	From \eqref{Eq:C-conv} we know that there exists a sequence $\eps_n \to 0$ such that for all $v \in \WHlogzero$
	\begin{equation}
		\label{Eq:C-convEps}
		\begin{aligned}
			& \left| \into \abs{\nabla u_n}^{p(x)-2}\nabla
			u_n  \cdot \nabla v \dx \right. \\
			& \left. +\into \mu(x) \left[ \log (e + \abs{\nabla u_n} ) + \frac{\abs{\nabla u_n}}{q(x) (e + \abs{\nabla u_n})} \right]  \abs{\nabla u_n}^{q(x)-2}\nabla u_n  \cdot \nabla v \dx \right. \\
			& \left. - \into f(x,u_n^+) v \dx \right|
			\leq \frac{\eps_n \normoneHlogzero{v}}{1 + \normoneHlogzero{u_n}} \quad \text{for all } n \in \N.
		\end{aligned}
	\end{equation}

	For any $v \in \WHlogzero$ we know that $v^\pm \in \WHlogzero$ by Proposition \ref{Prop:Truncations}. In particular we can take $v = - u_n^- \in \WHlogzero$ in \eqref{Eq:C-convEps}. As the fraction in the brackets is nonnegative and $f(x,u_n^+)u_n^- = 0$ for a.a.\,$x \in \Omega$, we have
	\begin{align*}
		& \modoneHlogzero{u_n^-} \\
		& \leq \into \left( \abs{\nabla u_n^-}^{p(x)}
		+ \mu(x) \left[ \log (e + \abs{\nabla u_n^-} ) + \frac{\abs{\nabla u_n^-}}{q(x) (e + \abs{\nabla u_n^-})} \right]  \abs{\nabla u_n^-}^{q(x)}\right) \dx \\
		& \leq \eps_n \quad \text{for all } n \in \N,
	\end{align*}
	or equivalently (see Proposition \ref{Prop:HlogModularNorm} \textnormal{(v)})
	\begin{align}
		\label{Eq:NegativeToZero}
		- u_n^- \to 0 \quad \text{in } \WHlogzero.
	\end{align}

	{\bf Claim:} There exists $M>0$ such that $\normoneHlogzero{u_n^+} \leq M$ for all $n \in \N$.

	By Lemma \ref{Le:QuotientFracLog} and taking $v = u_n^+ \in \WHlogzero$ in \eqref{Eq:C-convEps} we have
	\begin{align*}
		& -  \left( 1 + \frac{\kappa}{q_-} \right)   \modoneHlogzero{u_n^+} + \into f(x,u_n^+) u_n^+ \dx\\
		& \leq - \into \left( \abs{\nabla u_n^+}^{p(x)}
		+ \mu(x) \left[ \log (e + \abs{\nabla u_n^+} ) + \frac{\abs{\nabla u_n^+}}{q(x) (e + \abs{\nabla u_n^+})} \right]  \abs{\nabla u_n^+}^{q(x)}\right) \dx \\
		& \quad + \into f(x,u_n^+) u_n^+ \dx
		\leq \eps_n \quad \text{for all } n \in \N.
	\end{align*}
	On the other hand, from \eqref{Eq:C-bound} we also know that
	\begin{align*}
		\left( 1 + \frac{\kappa}{q_-} \right) \modoneHlogzero{u_n^+}
		- \into q_+ \left( 1 + \frac{\kappa}{q_-} \right) F(x,u_n^+) \dx \leq C_2 \quad \text{for all } n \in \N.
	\end{align*}
	Adding both inequalities one gets
	\begin{align*}
		\into \left(f(x,u_n^+) u_n^+ - q_+ \left( 1 + \frac{\kappa}{q_-} \right) F(x,u_n^+)\right) \dx \leq C_3 \quad \text{for all } n \in \N.
	\end{align*}
	By \eqref{Asf:CeramiAssumption}, where we assume, without loss of generality, that $l_- \leq \widetilde{l}_-$, there exist $C_4,C_5 > 0$ such that
	\begin{align*}
		C_4 |s|^{l_-} - C_5
		\leq f(x,s)s - q_+ \left( 1 + \frac{\kappa}{q_-} \right) F(x,s)
	\end{align*}
	for all $s \in \R$ and for a.a.\,$x \in \Omega$. The combination of the last two inequalities imply
	\begin{align}
		\label{Eq:BoundedL-}
		\norm{u_n^+}_{l_-} \leq C_6 \quad \text{for all } n \in \N.
	\end{align}
	Note that $l_- < r_+ < p^*_-$ because of \eqref{Asf:WellDef} and \eqref{Asf:CeramiAssumption}, so one can find $t \in (0,1)$ such that
	\begin{align*}
		\frac{1}{r_+} = \frac{t}{p^*_-} + \frac{1-t}{l_-}.
	\end{align*}
	This situates us in the appropriate setting to apply an interpolation inequality, like the one found in Proposition 2.3.17 in the book by Papageorgiou--Winkert \cite{Papageorgiou-Winkert-2018}. Together with \eqref{Eq:BoundedL-}, this gives us
	\begin{align*}
		\norm{u^+_n}_{r_+}^{r_+}
		\leq \left( \norm{u^+_n}^{t}_{p^*_-} \norm{u^+_n}^{1-t}_{l_-} \right)^{r_+}
		\leq C_6^{(1-t)r_+} \norm{u^+_n}^{t r_+}_{p^*_-} \quad \text{for all } n \in \N.
	\end{align*}
	For simplicity, we consider the case that $\normoneHlogzero{u_n^+} \geq 1$ for all $n \in \N$. From Proposition \ref{Prop:HlogModularNorm} \textnormal{(iii)}, and then by \eqref{Eq:C-convEps} with $v = u_n^+ \in \WHlogzero$ and \eqref{Asf:WellDef} we have
	\begin{align*}
		\normoneHlogzero{u^+_n}^{p_-}
		\leq \modoneHlogzero{ u^+_n }
		\leq C_7 (1 + \norm{u_n^+}_1 + \norm{u^+_n}^{r_+}_{r_+}) \quad \text{for all } n \in \N.
	\end{align*}
	These last two inequalities together with the embeddings (see Proposition \ref{Prop:EmbeddingHlogSobolev} (i))
	\begin{align*}
		\WHlogzero \hookrightarrow \Wp{p_-} \hookrightarrow \Lp{p_-^*}, \qquad \Lp{r_+} \hookrightarrow \Lp{1}
	\end{align*}
	yield
	\begin{align*}
		\normoneHlogzero{u^+_n}^{p_-}
		\leq C_9 \left( 1 + \normoneHlogzero{u^+_n}^{tr_+} \right)  \quad \text{for all } n \in \N.
	\end{align*}
	From the interval assumption in \eqref{Asf:CeramiAssumption}, we know that
	\begin{align*}
		tr_+
		& = \frac{p^*_-(r_+ - l_-)}{p^*_- - l_-}
		= \frac{N p_-(r_+ - l_-)}{Np_- - N l_- + p_- l_-}  \\
		& < \frac{N p_-(r_+ - l_-)}{Np_- - N l_- + p_- (r_+ - p_-) \frac{N}{p_-}} = p_-,
	\end{align*}
	thus there exists $M>0$ such that $\normoneHlogzero{u^+_n} \leq M$ for all $n \in \N$ and this completes the proof of the Claim.

	The boundedness of $\{u_n\}_{n \in \N}$ in $\WHlogzero$ is established by \eqref{Eq:NegativeToZero} and the previous Claim. Hence there is a subsequence $\{u_{n_k}\}_{k \in \N}$ such that
	\begin{align*}
		u_{n_k} \weak u \quad \text{in } \WHlogzero.
	\end{align*}
	Now taking $v = u_{n_k} - u \in \WHlogzero$ in \eqref{Eq:C-convEps}, we know that
	\begin{align*}
		\lim\limits_{k \to \infty} \langle \ph_+ ' (u_{n_k}) , u_{n_k} - u \rangle = 0,
	\end{align*}
	and from the weak convergence of $\{u_{n_k}\}_{k \in \N}$ and Lemma \ref{Le:Propf} \ref{Propf:ConvergenceFterm} (check remark \ref{Re:StrongCont}) we obtain
	\begin{align*}
		\lim\limits_{k \to \infty} \into f(x,u_{n_k}^+) ( u_{n_k} - u ) \dx = 0.
	\end{align*}
	The last two limits together yield
	\begin{align*}
		\lim\limits_{k \to \infty} \langle A(u_{n_k}) , u_{n_k} - u \rangle = 0
	\end{align*}
	and the \textnormal{(S$_+$)}-property of the operator $A$ (see Theorem \ref{Th:PropertiesOperator}) implies that
	\begin{align*}
		u_{n_k} \to u \quad \text{in } \WHlogzero.
	\end{align*}
\end{proof}

Now we have to check the mountain pass geometry.

\begin{proposition}
	\label{Prop:PhLowerBound}
	Let \eqref{Assump:H2} be satisfied and $f$ fulfill \eqref{Asf:Caratheodory}, \eqref{Asf:WellDef} and \eqref{Asf:GrowthZero}. Then there exist constants $C_1,C_2,C_3 >0$ such that for all $\eps > 0$
	\begin{align*}
		\ph(u), \ph_\pm (u) \geq
		\begin{cases}
			C_1 a_\eps^{-1} \normoneHlogzero{u}^{q_+ + \eps} - C_2 \normoneHlogzero{u}^{r_-}, & \text{if } \normoneHlogzero{u} \leq \min\{1,C_3\}, \\
			C_1 \normoneHlogzero{u}^{p_-} - C_2 \normoneHlogzero{u}^{r_+}, & \text{if } \normoneHlogzero{u} \geq \max\{1,C_3\},
		\end{cases}
	\end{align*}
	where $a_\eps$ is the same one as in Lemma \ref{Le:fepsilon}.
\end{proposition}

\begin{proof}
	We do the argument only for $\ph$, as for $\ph_\pm$ we can use $\varrho_{r(\cdot)} ( \pm u^\pm ) \leq \varrho_{r(\cdot)} ( u )$. Applying Lemma \ref{Le:Propf} \ref{Propf:EpsilonUpperBoundF}, the embedding of $\WHlogzero \hookrightarrow \Lp{r(\cdot)}$ with constant $C_{\hlog}$ from Proposition \ref{Prop:EmbeddingHlogSobolev} \textnormal{(iii)} and Proposition \ref{Prop:ModularNormVarExp} \textnormal{(iii)}, \textnormal{(iv)}, we have for any $u \in \WHlogzero$
	\begin{align*}
		\ph (u)
		 & \geq \frac{1}{q_+} \modoneHlogzero{u}
		- C \varrho_{r(\cdot)} ( u )             \\
		 & \geq \frac{1}{q_+} \modoneHlogzero{u}
		- C_\eps \max_{k \in \{ r_+, r_- \} } \{ C_{\hlog} ^k \normoneHlogzero{u}^k \}.
	\end{align*}
	Choosing $C_3 = 1/C_{\hlog}$,
	the result follows from Proposition \ref{Prop:HlogModularNorm} \textnormal{(iii)} and \textnormal{(iv)} with
	\begin{align*}
		C_1 = \frac{1}{q_+}
		\quad \text{and}\quad
		C_2 =
		\begin{cases}
			C_\eps C_{\hlog}^{r_-} & \text{for } \normoneHlogzero{u} \leq C_3, \\
			C_\eps C_{\hlog}^{r_+} & \text{for } \normoneHlogzero{u} > C_3.
		\end{cases}
	\end{align*}
\end{proof}

\begin{corollary}
	\label{Cor:RingOfMountains}
	Let \eqref{Assump:H2} be satisfied and $f$ fulfill \eqref{Asf:Caratheodory}, \eqref{Asf:WellDef} with $q_+<r_-$ and \eqref{Asf:GrowthZero}. Then there exist $\delta >0$ such that
	\begin{align*}
		\inf_{\normoneHlogzero{u} = \delta} \,\ph(u) >0
		\quad \text{and} \quad
		\inf_{\normoneHlogzero{u} = \delta} \,\ph_\pm(u) > 0.
	\end{align*}
	Alternatively, there exists $\delta' > 0 $ such that $\ph(u) > 0$ for $0 < \normoneHlogzero{u} < \delta'$.
\end{corollary}

\begin{proposition}
	\label{Prop:PointBeyondMountains}
	Let \eqref{Assump:H2} be satisfied and $f$ fulfill \eqref{Asf:Caratheodory}, \eqref{Asf:WellDef} and \eqref{Asf:GrowthInfty}. Let $0 \neq u \in \WHlogzero$, then $\ph(tu) \xrightarrow{t \to \pm \infty} - \infty$. Furthermore, if $u \geq 0$ a.e.\,in $\Omega$, $\ph_\pm (tu) \xrightarrow{t \to \pm \infty} - \infty$.
\end{proposition}

\begin{proof}
	As before, we only show the argument for $\ph$. We can do this because if $u \geq 0$ a.e.\,in $\Omega$, $\ph_\pm (tu) = \ph(tu)$ for $\pm t > 0$.

	Fix any $0 \neq u \in \WHlogzero$ and let $t, \eps \geq 1$. First note that $ \norm{u}_{q_+} < \infty$ due to Proposition \ref{Prop:EmbeddingHlogSobolev} \textnormal{(iii)} and Lemma \ref{Le:Propf} \ref{Propf:q<r}. From \eqref{Eq:LogGrowthProduct} and Lemma \ref{Le:Propf} \ref{Propf:EpsilonLowerBoundF} we obtain
	\begin{align*}
		\ph(tu)
		& \leq \frac{\abs{t}^{p_+}}{p_-} \varrho_{p(\cdot)} ( \nabla u )
		+ \frac{\abs{t}^{q_+}}{q_-} \log (e + \abs{t}) \into \mu(x) \abs{\nabla u}^{q(x)} \dx \\
		& \quad + \frac{\abs{t}^{q_+}}{q_-} \into \mu(x) \abs{\nabla u}^{q(x)} \log(e + \abs{\nabla u}) \dx \\
		& \quad - \frac{\eps \abs{t}^{q_+}}{q_+} \into|u|^{q_+} \log (e + t\abs{u}) \dx \; + \; C_\eps |\Omega|.
	\end{align*}
	Note that
	\begin{align*}
		\int_{\{ u \geq 1 \}} |u|^{q_+} \log (e + t\abs{u}) \dx
		\geq \log (e + t) \int_{\{ u \geq 1 \}} |u|^{q_+} \dx
	\end{align*}
	and again by \eqref{Eq:LogGrowthProduct}
	\begin{align*}
		\int_{\{ 0 < u < 1 \}} |u|^{q_+} \frac{1/\abs{u}}{1/\abs{u}} \log (e + t\abs{u}) \dx
		\geq \log (e + t) \int_{\{ 0 < u < 1 \}} |u|^{q_+ + 1} \dx.
	\end{align*}
	Applying these two inequalities in the first inequality yields
	\begin{align*}
		\ph(tu)
		& \leq \frac{\abs{t}^{p_+}}{p_-} \varrho_{p(\cdot)} ( \nabla u )
		+ \abs{t}^{q_+} \log (e + \abs{t})\\
		&\quad \times \left( \frac{1}{q_-} \into \mu(x) \abs{\nabla u}^{q(x)} \dx - \eps \frac{1}{q_+} \left[ \int_{\{ u \geq 1 \}} |u|^{q_+} \dx + \int_{\{ 0 < u < 1 \}} |u|^{q_+ + 1} \dx \right]  \right)\\
		& \quad + \frac{\abs{t}^{q_+}}{q_-} \into \mu(x) \abs{\nabla u}^{q(x)} \log(e + \abs{\nabla u}) \dx + C_\eps |\Omega|.
	\end{align*}
	The second term is negative for values of $\eps$ large enough. Setting such a value directly yields $\ph(tu) \xrightarrow{t \to \pm \infty} - \infty$.
\end{proof}

Now we have all the required properties to apply the mountain pass theorem.

\begin{theorem}
	\label{Th:ConstantSignSolutions}
	Let \eqref{Assump:H2} be satisfied and $f$ fulfill \eqref{Asf:Caratheodory}, \eqref{Asf:WellDef}, \eqref{Asf:GrowthInfty}, \eqref{Asf:GrowthZero} and \eqref{Asf:CeramiAssumption}. Then there exist nontrivial weak solutions $u_0,v_0 \in \WHlogzero $ of problem \eqref{Eq:Problem} such that $u_0 \geq 0$ and $v_0 \leq 0$ a.e.\,in $\Omega$.
\end{theorem}

\begin{proof}
	Due to the combination of Propositions \ref{Cor:RingOfMountains}, \ref{Prop:PointBeyondMountains} and \ref{Prop:CeramiCondition}, the assumptions of Theorem \ref{Th:MPT} are satisfied for the truncated energy functionals $\ph_\pm$. This yields the existence of $u_0,v_0 \in \WHlogzero$ such that $\ph_+'(u_0) = 0 = \ph_-'(v_0)$ and
	\begin{align*}
		\ph_+(u_0), \ph_-(v_0) \geq \inf_{\normoneHlogzero{u} = \delta} \,\ph_\pm(u) > 0 = \ph_+ (0),
	\end{align*}
	which proves that $u_0 \neq 0 \neq v_0$. Furthermore, testing $\ph_+'(u_0) = 0$ with $-u_0^-$ we obtain $\modoneHlogzero{u^-} = 0$, which by Proposition \ref{Prop:HlogModularNorm} \textnormal{(i)} means that $-u_0^- = 0$ a.e.\,in $\Omega$, so $u_0 = u_0^+ \geq 0$ a.e.\,in $\Omega$. Analogously, $v_0 \leq 0$ a.e.\,in $\Omega$.
\end{proof}

Alternatively, we could have used the following assumptions instead of \eqref{Assump:H2} and \eqref{Asf:GrowthZero}.

\begin{enumerate}[label=\textnormal{(H$_2$')},ref=\textnormal{H$_2$'}]
	\item\label{Assump:H2prime}
	$\Omega \subseteq \R^N$, with $N \geq 2$, is a bounded domain with Lipschitz boundary $\partial \Omega$, $p,q \in C_+ (\close)$ with $p(x) \leq q(x) \leq q_+ < p^*_-$ for all $x \in \close$, $p$ satisfies a monotonicity condition, that is, there exists a vector $l \in \R^N \setminus \{ 0 \}$ such that for all $x \in \Omega$ the function
	\begin{align*}
		h_x (t) = p(x + tl) \quad \text{with } t \in I_x = \{ t \in \R \colon x + tl \in \Omega\}
	\end{align*}
	is monotone, and $0\leq\mu(\cdot) \in \Lp{\infty}$.
\end{enumerate}

\begin{enumerate} [label=\textnormal{(f$_{\arabic*}$')},ref=\textnormal{f$_{\arabic*}$'}]
	\setcounter{enumi}{3}
	\item\label{Asf:GrowthZeroAlt}
	\begin{align*}
		\lim\limits_{s \to 0} \frac{F(x,s)}{\abs{s}^{p(x)}} = 0 \quad
		\text{uniformly for a.a.\,} x \in \Omega.
	\end{align*}
\end{enumerate}
This new assumption in \eqref{Asf:GrowthZeroAlt} has a slightly different consequence than its counterpart earlier in this section.

\begin{lemma}
	\label{Le:PropLimftozero}
	Let $f \colon \Omega \times \R \to \R$ .
	\begin{enumerate}[label=(\roman*),font=\normalfont]
		\item
			\label{Propf:Zeroatzero}
			If $f$ fulfills \eqref{Asf:Caratheodory} and \eqref{Asf:GrowthZeroAlt}, then $f(x,0)=0$ for a.a.\,$x \in \Omega$.
		\item
			\label{Propf:EpsilonUpperBoundAlt}
			If $f$ fulfills \eqref{Asf:WellDef} and \eqref{Asf:GrowthZeroAlt}, then for each $\eps > 0$ there exists $C_\eps > 0$ such that
			\begin{align*}
				\abs{F(x,t)} \leq \frac{\eps}{p(x)} \abs{t}^{p(x)} + C_\eps \abs{t}^{r(x)} \quad \text{for a.a.\,} x \in \Omega.
			\end{align*}
	\end{enumerate}
\end{lemma}

\begin{remark}
	Without the assumption \eqref{Asf:GrowthZero} one needs to prove a result like Lemma \ref{Le:PropLimftozero} \ref{Propf:Zeroatzero}, since this condition is necessary to ensure that $\ph_\pm$ are differentiable among other important properties, see Remark \ref{Re:StrongCont}.
\end{remark}

In this case, all the propositions are identical except for the behavior close to zero, for which we provide the following proof.

\begin{proposition}
	\label{Prop:PhLowerBoundAlt}
	Let \eqref{Assump:H2prime} be satisfied and $f$ fulfill \eqref{Asf:Caratheodory}, \eqref{Asf:WellDef} and \eqref{Asf:GrowthZeroAlt}. Then there exist constants $C_1,C_2,C_3 >0$ such that for all $\eps > 0$
	\begin{align*}
		\ph(u), \ph_\pm (u) \geq
		\begin{cases}
			C_1 a_\eps^{-1} \normoneHlogzero{u}^{q_+ + \eps} - C_2 \normoneHlogzero{u}^{r_-}, & \text{if } \normoneHlogzero{u} \leq \min\{1,C_3\}, \\
			C_1 \normoneHlogzero{u}^{p_-} - C_2 \normoneHlogzero{u}^{r_+}, & \text{if } \normoneHlogzero{u} \geq \max\{1,C_3\},
		\end{cases}
	\end{align*}
	where $a_\eps$ is the same as in Lemma \ref{Le:fepsilon}.
\end{proposition}

\begin{proof}
	As in Proposition \ref{Prop:PhLowerBound} we only do the argument for $\ph$ for the same reasons. Using Lemma \ref{Le:PropLimftozero} \ref{Propf:EpsilonUpperBoundAlt}, Poincar\'e inequality for the modular in $W_0^{1,p(\cdot)} (\Omega)$ with constant $C_{p(\cdot)}$ from Proposition \ref{Prop:PoincareModular}, the embedding of $\WHlogzero \hookrightarrow \Lp{r(\cdot)}$ with constant $C_{\hlog}$ from Proposition \ref{Prop:EmbeddingHlogSobolev} \textnormal{(iii)} and Proposition \ref{Prop:ModularNormVarExp} \textnormal{(iii)}, \textnormal{(iv)}, we obtain for any $u \in \WHlogzero$
	\begin{align*}
		& \ph (u) \\
		& \geq \frac{1}{p_+} \varrho_{p(\cdot)} ( \nabla u ) + \frac{1}{q_+} \into \mu(x) \abs{\nabla u}^{q(x)} \log(e + \abs{\nabla u}) \dx
		- \frac{\lambda}{p_-} \varrho_{p(\cdot)} ( u ) - C_{\lambda} \varrho_{r(\cdot)} ( u )  \\
		& \geq \left( \frac{1}{p_+} - \frac{C_{p(\cdot)} \lambda}{p_-} \right) \varrho_{p(\cdot)} ( \nabla u ) +\frac{1}{q_+} \into \mu(x) \abs{\nabla u}^{q(x)} \log(e + \abs{\nabla u}) \dx \\
		& \quad - C_{\lambda} \max_{k \in \{ r_+, r_- \} } \{ C_{\hlog} ^k \normoneHlogzero{u}^k \}  \\
		& \geq \min \left\{ \frac{1}{p_+} - \frac{C_{p(\cdot)} \lambda}{p_-} , \frac{1}{q_+} \right\} \modoneHlogzero{u} - C_{\lambda} \max_{k \in \{ r_+, r_- \} } \{ C_{\hlog} ^k \normoneHlogzero{u}^k \}.
	\end{align*}
	If $p_+ = q_+$, any $\lambda > 0$ works, otherwise choose $0 < \lambda < \frac{p_-(q_+-p_+)}{C_{p(\cdot)} q_+ p_+}$. By Proposition \ref{Prop:HlogModularNorm} \textnormal{(iii)} and \textnormal{(iv)} the result follows with $C_3 = 1/C_{\hlog}$,
	\begin{align*}
		C_1 = \frac{1}{q_+} \quad\text{and}\quad
		C_2 =
		\begin{cases}
			C_{\lambda} C_{\hlog}^{r_-} & \text{for } \normoneHlogzero{u} \leq C_3, \\
			C_{\lambda} C_{\hlog}^{r_+} & \text{for } \normoneHlogzero{u} > C_3.
		\end{cases}
	\end{align*}
\end{proof}

With a reasoning identical to Theorem \ref{Th:ConstantSignSolutions} but using Proposition \ref{Prop:PhLowerBoundAlt}, we obtain the next result.

\begin{theorem}
	\label{Th:ConstantSignSolutionsAlt}
	Let \eqref{Assump:H2prime} be satisfied and $f$ fulfill \eqref{Asf:Caratheodory}, \eqref{Asf:WellDef}, \eqref{Asf:GrowthInfty}, \eqref{Asf:GrowthZeroAlt} and \eqref{Asf:CeramiAssumption}. Then there exist nontrivial weak solutions $u_0,v_0 \in \WHlogzero $ of problem \eqref{Eq:Problem} such that $u_0 \geq 0$ and $v_0 \leq 0$ a.e.\,in $\Omega$.
\end{theorem}

\section{Sign-changing solution}\label{sign-changing-solution}

The aim of this section is to prove the existence of a sign-changing solution on top of the other two solutions from last section, which were one positive and one negative. We need to substitute \eqref{Asf:GrowthInfty} with a stronger assumption and to restrict a bit more the assumptions on the functional space with respect to \eqref{Assump:H2prime}:

\begin{enumerate}[label=\textnormal{(H$_3$)},ref=\textnormal{H$_3$}]
	\item\label{Assump:H3}
	$\Omega \subseteq \R^N$, with $N \geq 2$, is a bounded domain with Lipschitz boundary $\partial \Omega$, $p,q \in C_+ (\close)$ with $p(x) \leq q(x) \leq q_+ < q_+ + 1 < p_-^*$ for all $x \in \close$, $p$ satisfies a monotonicity condition, that is, there exists a vector $l \in \R^N \setminus \{ 0 \}$ such that for all $x \in \Omega$ the function
	\begin{align*}
		h_x (t) = p(x + tl) \quad \text{ with } t \in I_x = \{ t \in \R \colon x + tl \in \Omega\}
	\end{align*}
	is monotone, and $0\leq\mu(\cdot) \in \Lp{\infty}$.
\end{enumerate}

\begin{enumerate} [label=\textnormal{(f$_{\arabic*}$')},ref=\textnormal{f$_{\arabic*}$'}]
	\setcounter{enumi}{2}
	\item\label{Asf:QuotientMono}
	The function $t \mapsto f(x,t)/\abs{t}^{q_+}$ is increasing in $(- \infty, 0)$ and in $(0, + \infty)$ for a.a.\,$x \in \Omega$.
\end{enumerate}

\begin{remark}
	A necessary condition for \eqref{Assump:H3} to be satisfied is that $p_+ + 1 \leq p_-^*$. For the case of constant exponents, which is the least restrictive case, this is equivalent to $\frac{- 1 + \sqrt{1 + 4 N}}{2} \leq p$. This strong assumption is required for \eqref{Asf:QuotientMono}.
\end{remark}

Similarly to the previous sections, we have some important consequences of these assumptions on the right-hand side.

\begin{lemma}
	\label{Le:Propq+1}
	Let $f \colon \Omega \times \R \to \R$ satisfy the following assumptions.
	\begin{enumerate}[label=(\roman*),font=\normalfont]
		\item
			If $f$ fulfills \eqref{Asf:WellDef} and  \eqref{Asf:QuotientMono}, then $q_+ + 1 \leq r_-$.
		\item\label{Propf:growthq+1}
			If $f$ fulfills \eqref{Asf:QuotientMono} then for any $\eps > 0$
			\begin{align*}
				\lim\limits_{s \to \pm \infty} \frac{F(x,s)}{\abs{s}^{q_+ + 1 - \eps}}  = + \infty \quad \text{uniformly for a.a.\,} x \in \Omega.
			\end{align*}
			In particular, \eqref{Asf:GrowthInfty} is satisfied.
	\end{enumerate}
\end{lemma}

\begin{remark}
	As in \eqref{Assump:H3} we assumed $q_+ + 1 < p^*_-$, we can always find at least one $r \in C_+(\close)$ such that $q_+ + 1 \leq r_- \leq r_+ < p^*_-$.
\end{remark}

We achieve the existence of such sign-changing solution making use of the Nehari manifold technique. Our treatment is inspired by the recent work by Crespo-Blanco--Winkert \cite{Crespo-Blanco-Winkert-2022}, a more comprehensive explanation of this technique can be found in the book chapter by Szulkin--Weth \cite{Szulkin-Weth-2010}. The Nehari manifold associated to $\ph$ is the set
\begin{align*}
	\mathcal{N} = \left\{ u \in \WHlogzero \colon \langle \ph'(u) , u \rangle = 0, \; u \neq 0 \right\}.
\end{align*}
A significant property of this set is that all weak solutions of \eqref{Eq:Problem} (or critical points of $\ph$) belong to this set, except for $u=0$, which can be studied separately. An important remark is that, although the most used name in the literature is Nehari manifold, it does not have to be one in general. For our purposes this is not relevant, so we do not do any discussion in this direction. We are looking for sign-changing solutions, so in the center of our results we work with the variation
\begin{align*}
	\mathcal{N}_0 = \left\{ u \in \WHlogzero \colon \pm u^\pm \in \mathcal{N} \right\}.
\end{align*}
Note that in our case $\mathcal{N}_0$ is a subset of $\mathcal{N}$, but this might not be true in general. Our treatment starts by establishing some structure on $\mathcal{N}$ which is needed for the work on $\mathcal{N}_0$. The following lemma is used in the proof and the main reason why we need \eqref{Asf:QuotientMono} as it is instead of with exponent $q_+ - 1$ (which would be much less restrictive).

\begin{lemma}
	\label{Le:QuotientMonotone}
	Let $b > 0$ and $Q > 1$, the mapping $t \mapsto \frac{ t^{1-\eps} b }{Q (e + t b)}$ is decreasing on $(0,\infty)$ only for $\eps \geq 1$.
\end{lemma}

\begin{proposition}\label{Prop:NehariManifoldProps}
	Let \eqref{Assump:H3} be satisfied  and $f$ fulfill \eqref{Asf:Caratheodory}, \eqref{Asf:WellDef}, \eqref{Asf:QuotientMono} and \eqref{Asf:GrowthZeroAlt}. Then for any $u \in \WHlogzero\setminus\{0\}$ there exists a unique $t_u > 0$ such that $t_u u \in \mathcal{N}$. Furthermore, $\ph(t_u u) > 0$, $\frac{\mathrm{d}}{\mathrm{d}t} \ph(tu) > 0$ for $0 < t < t_u$, $\frac{\mathrm{d}}{\mathrm{d}t} \ph(tu) = 0$ for $t = t_u$, $\frac{\mathrm{d}}{\mathrm{d}t} \ph(tu) < 0$ for $t > t_u$, and therefore $\ph(tu) < \ph(t_u u)$ for all $0 < t \neq t_u$.
\end{proposition}

\begin{proof}
	For any $u \in \WHlogzero \setminus \{0\}$ we define its associated fibering function $\theta_u \colon [0,\infty) \to \R$ as $\theta_u(t) = \ph(tu)$. This function is $C^1$ in $(0,\infty)$ and continuous in $[0,\infty)$ because it is the composition of such functions. Clearly $\theta_u(0) = 0$ and from Propositions \ref{Prop:PhLowerBoundAlt} and \ref{Prop:PointBeyondMountains} we know that there exist $\delta,K > 0$ with the properties
	\begin{align}
		\label{Eq:FiberingZeroInfty}
		\theta_u (t) > 0 \quad \text{for } 0 < t < \delta \quad \text{and} \quad \theta_u(t) < 0 \quad \text{for } t > K.
	\end{align}
	Then the extreme value theorem implies that the global maximum of $\theta_u$ is attained at some $t_u \in (0,K]$. Furthermore, it is a critical point of $\theta_u$ and by the chain rule we know
	\begin{align*}
		0 = \theta_u ' (t_u)
		= \langle \ph'(t_u u) , u \rangle,
	\end{align*}
	which proves that $t_u u \in \mathcal{N}$.

	In order to see the uniqueness and the sign of the derivatives, observe that for $t > 0$ and due to \eqref{Asf:QuotientMono}, it holds
	\begin{align*}
		\frac{f(x,tu)}{t^{q_+}\abs{u}^{q_+}} \text{ increasing in $t$, so } \frac{f(x,tu)u}{t^{q_+}} \text{ increasing in $t$} & \text{, for } x \in \Omega \text{ with } u(x) > 0, \\
		\frac{f(x,tu)}{t^{q_+}\abs{u}^{q_+}} \text{ decreasing in $t$, so } \frac{f(x,tu)u}{t^{q_+}} \text{ increasing in $t$} & \text{, for } x \in \Omega \text{ with } u(x) < 0.
	\end{align*}
	Similarly to above, we know that $tu \in \mathcal{N}$ implies $u \neq 0$ and $\theta_u '(t) = 0$. By multiplying this equation with $1/t^{q_+}$ it follows
	\begin{align*}
		 & \into \left( \frac{ 1 } {t^{q_+ + 1 - p(x)}}  \abs{\nabla u} ^{p(x)}\right.                                                                                                              \\
		 & \left.\qquad+ \frac{1}{ t^{q_+ - q(x)} } \mu(x) \abs{\nabla u}^{q(x)} \left[ \frac{\log (e + t \abs{\nabla u} )}{t} + \frac{\abs{\nabla u}}{q(x) (e + t \abs{\nabla u})} \right] \right. \\
		 & \qquad \left. - \frac{f(x,tu)u}{t^{q_+ }} \right) \dx = 0.
	\end{align*}
	On the set $\{ \nabla u \neq 0 \}$, the first term is strictly decreasing since $p_+ < q_+ + 1$, the second term is decreasing by Lemmas \ref{Le:fepsilon} and \ref{Le:QuotientMonotone}, and the third one is decreasing by the observation just above it. We know that $u \neq 0$, so the whole left-hand side of the equation is strictly decreasing as a function of $t$, which means that there can be at most one value $t_u>0$ such that $\theta_u'(t_u) = 0$, i.e.\,$t_u u \in \mathcal{N}$. Furthermore, $\theta_u'(t)$ cannot take value $0$ anywhere else, so it has constant sign on $(0,t_u)$ and $(t_u,\infty)$, and they must be positive and negative respectively by \eqref{Eq:FiberingZeroInfty}.
\end{proof}

The advantage of working in the Nehari manifold instead of the whole space is that our functional, even when it was not coercive in the original space, it does have some coercivity-type property on the restricted space.

\begin{proposition}
	\label{Prop:Coercivity}
	Let \eqref{Assump:H3} be satisfied  and $f$ fulfill \eqref{Asf:Caratheodory}, \eqref{Asf:WellDef}, \eqref{Asf:QuotientMono} and \eqref{Asf:GrowthZeroAlt}. Then the functional $\ph_{|_\mathcal{N}}$ is sequentially coercive, in the sense that for any sequence $\{u_n\}_{n \in \N} \subseteq \mathcal{N}$ such that $\normoneHlogzero{u_n} \xrightarrow{n \to \infty} + \infty$ it follows that $\ph(u_n) \xrightarrow{n \to \infty} + \infty$.
\end{proposition}

\begin{proof}
	Consider a sequence $\{u_n\}_{n \in \N}$ with the property that $\normoneHlogzero{u_n} \xrightarrow{n \to \infty} \infty$. Let us define $y_n = u_n / \normoneHlogzero{u_n}$. Due to Proposition \ref{Prop:EmbeddingHlogSobolev} \textnormal{(iii)}, we know that there exists a subsequence $\{y_{n_k}\}_{k \in \N}$ and $y \in \WHlogzero$ such that
	\begin{align*}
		y_{n_k} \weak y \quad\text{in } \WHlogzero,
		\quad y_{n_k} \to y \quad\text{in } \Lp{r(\cdot)}  \text{ and pointwisely a.e.\,in }\Omega.
	\end{align*}

	{\bf Claim:} $y = 0$

	We proceed by contradiction, so assume that $y \neq 0$. Let $\eps > 0$ and without loss of generality let $\normoneHlogzero{u_{n_k}} \geq 1$ for all $k \in \N$. Proposition \ref{Prop:oneHlogModularNorm} \textnormal{(iii)} and \textnormal{(iv)} yield for all $k \in \N$
	\begin{align*}
		\ph (u_{n_k}) & \leq
		\frac{1}{p_-} \modoneHlogzero{ u_{n_k} } - \into F(x, u_{n_k}) \dx \\
		& \leq \frac{a_\eps}{p_-} \normoneHlogzero{ u_{n_k} }^{q_+ + \eps} - \into F(x, u_{n_k}) \dx
	\end{align*}
	and dividing by $\normoneHlogzero{u_{n_k}}^{ q_+ + \eps }$ we obtain for all $k \in \N$
	\begin{align}
		\label{Eq:CoercivityClaim}
		\frac{\ph (u_{n_k})}{\normoneHlogzero{ u_{n_k} }^{q_+ + \eps}}
		\leq \frac{a_\eps}{p_-} - \into \frac{F(x, u_{n_k})}{\normoneHlogzero{ u_{n_k} }^{q_+ + \eps}} \dx.
	\end{align}
	On the other hand, by Lemma \ref{Le:Propq+1} \ref{Propf:growthq+1} we know
	\begin{align*}
		\lim_{k \to \infty} \frac{F(x, u_{n_k})}{\normoneHlogzero{ u_{n_k} }^{q_+ + \eps}}
		= \lim_{k \to \infty} \frac{F(x, u_{n_k})}{\abs{ u_{n_k} }^{q_+ + \eps}} \abs{ y_{n_k} }^{q_+ + \eps}
		= \infty
	\end{align*}
	for $x \in \Omega$ with $y(x) \neq 0$. So by Lemma \ref{Le:Propf} \ref{Propf:Boundedbelow} and Fatou's Lemma , where in the following $\Omega_0 = \{ x \in \Omega \colon y(x) = 0\}$
	\begin{align*}
		\into \frac{F(x, u_{n_k})}{\normoneHlogzero{ u_{n_k} }^{q_+ + \eps}}  \dx
		& = \int_{\Omega \setminus \Omega_0} \frac{F(x, u_{n_k})}{\normoneHlogzero{ u_{n_k} }^{q_+ + \eps}} \dx + \int_{\Omega_0} \frac{F(x, u_{n_k})}{\normoneHlogzero{ u_{n_k} }^{q_+ + \eps}}  \dx \\
		& \geq \int_{\Omega \setminus \Omega_0} \frac{F(x, u_{n_k})}{\normoneHlogzero{ u_{n_k} }^{q_+ + \eps}} \dx - \frac{M |\Omega| }{\normoneHlogzero{ u_{n_k} }^{q_+ + \eps}}
		\xrightarrow{k \to \infty} \infty.
	\end{align*}
	Together with \eqref{Eq:CoercivityClaim}, this implies that $\ph(u_{n_k}) < 0$ for $k$ large enough, which is a contradiction to the fact that $\ph(u_n) > 0$ for all $n \in \N$ given by Proposition \ref{Prop:NehariManifoldProps} and the Claim is proved.

	Take any $C>1$. By Proposition \ref{Prop:NehariManifoldProps}, as $u_{n_k} \in \mathcal{N}$ for all $k \in \N$, we know that $\ph(u_{n_k}) \geq \ph (C y_{n_k})$ for all $k \in \N$. Also by using Proposition \ref{Prop:oneHlogModularNorm} \textnormal{(iii)} one gets for all $k \in \N$
	\begin{align*}
		\ph (u_{n_k}) & \geq \ph (C y_{n_k})
		\geq \frac{1}{q_+}  \modoneHlogzero{C y_{n_k}}  - \into F(x, Cy_{n_k}) \dx\\
		& \geq \frac{1}{q_+} \normoneHlogzero{C y_{n_k}}^{p_-} - \into F(x, Cy_{n_k})\dx
		= \frac{C^{p_-}}{q_+} - \into F(x, Cy_{n_k}) \dx.
	\end{align*}
	We also know that $C y_{n_k} \weak 0$, which together with the strong continuity of Lemma \ref{Le:Propf} \ref{Propf:ConvergenceFterm} gives us a $k_0 \in \N$ such that for all $k \geq k_0$
	\begin{align*}
		\ph (u_{n_k}) \geq \frac{C^{p_-}}{q_+} - 1.
	\end{align*}
	In the previous argument, $C$ can be as big as desired and $k_0$ depends on $C$, so we derived that $\ph (u_{n_k}) \xrightarrow{k \to \infty} + \infty$. The subsequence principle yields $\ph (u_n) \xrightarrow{n \to \infty} + \infty$.
\end{proof}

Minimizing over the Nehari manifold we always stay strictly positive. This is useful to ensure that we are not converging to zero along a minimizing sequence.

\begin{proposition}
	\label{Prop:Infimum>0}
	Let \eqref{Assump:H3} be satisfied and $f$ fulfill \eqref{Asf:Caratheodory}, \eqref{Asf:WellDef}, \eqref{Asf:QuotientMono} and \eqref{Asf:GrowthZeroAlt}. Then
	\begin{align*}
		\inf_{u \in \mathcal{N}} \ph(u) > 0
		\quad \text{ and } \quad
		\inf_{u \in \mathcal{N}_0} \ph(u) > 0.
	\end{align*}
\end{proposition}

\begin{proof}
	From Proposition \ref{Prop:NehariManifoldProps} and Corollary \ref{Cor:RingOfMountains}, for any $u \in \mathcal{N}$ we know that
	\begin{align*}
		\ph (u)
		\geq \ph \left( \frac{\delta}{\normoneHlogzero{u}} u \right)
		\geq \inf_{\normoneHlogzero{u} = \delta} \ph(u)
		> 0,
	\end{align*}
	which proves the first part. The second one follows from the first one since $\ph (u) = \ph(u^+) + \ph(-u^-)$ and $u^+, -u^- \in \mathcal{N}$ for $u \in \mathcal{N}_0$.
\end{proof}

We can now perform calculus of variations on $\mathcal{N}_0$ to obtain our candidate for a third solution.

\begin{proposition}
	\label{Prop:MinimizerExistence}
	Let \eqref{Assump:H3} be satisfied and $f$ fulfill \eqref{Asf:Caratheodory}, \eqref{Asf:WellDef}, \eqref{Asf:QuotientMono} and \eqref{Asf:GrowthZeroAlt}. Then there exists $w_0 \in \mathcal{N}_0$ such that $\ph(w_0) = \inf_{u \in \mathcal{N}_0} \ph(u)$.
\end{proposition}

\begin{proof}
	We prove the result via the direct method of calculus of variations. Let $\{u_n\}_{n \in \N} \subseteq \mathcal{N}_0$ be a minimizing sequence, i.e., a sequence such that $\ph(u_n) \searrow \inf_{u \in \mathcal{N}_0} \ph(u)$. Recall that, by Proposition \ref{Prop:Truncations}, $u_n^+, - u_n^- \in \WHlogzero$ for all $n \in \N$. Then the sequences $\{ \ph( u_n^+ ) \}_{n \in \N}$ and $\{ \ph( - u_n^- ) \}_{n \in \N}$ are bounded in $\R$, since $\ph(u_n) = \ph(u_n^+) \mathop{+} \ph(-u_n^-)$ for all $n \in \N$ and by Proposition \ref{Prop:NehariManifoldProps} we have that $\ph( u_n^+) > 0$ and $\ph(- u_n^-) > 0$ for all $n \in \N$. Hence the coercivity condition of Proposition \ref{Prop:Coercivity} yields the boundedness of $\{  u_n^+  \}_{n \in \N}$ and $\{  - u_n^-  \}_{n \in \N}$ in $\WHlogzero$. By Proposition \ref{Prop:EmbeddingHlogSobolev} \textnormal{(iii)}, we obtain subsequences $\{  u_{n_k} ^+\}_{k \in \N}$ and $\{- u_{n_k} ^-\}_{k \in \N}$, and $z_1, z_2 \in \WHlogzero$ such that
	\begin{equation}
		\label{Eq:MinimizingSubsequence}
		\begin{aligned}
			& u_{n_k} ^+ \weak z_1 \quad\text{in } \WHlogzero, \quad
			u_{n_k} ^+ \to z_1 \quad\text{in } \Lp{r(\cdot)} \quad
			\text{and a.e.\,in }\Omega, \\
			& - u_{n_k} ^- \weak z_2 \quad\text{in } \WHlogzero, \quad
			- u_{n_k} ^- \to z_2 \quad\text{in } \Lp{r(\cdot)} \quad
			\text{and a.e.\,in }\Omega, \\
			& \text{with } z_1 \geq 0, \quad
			z_2 \leq 0 \quad
			\text{and } \quad z_1 z_2 = 0.
		\end{aligned}
	\end{equation}

	{\bf Claim:} $z_1 \neq 0 \neq z_2$

	We only prove the assertion for $z_1$, the other one is exactly the same. Arguing by contradiction, let $z_1 = 0$. By assumption $u_{n_k} ^+ \in \mathcal{N}$, so we have
	\begin{align*}
		0
		& = \langle \ph'( u_{n_k} ^+) , u_{n_k} ^+  \rangle \\
		& =  \modoneHlogzero{ u_{n_k} ^+}
		+ \into \frac{\abs{ \nabla u_{n_k} ^+}}{q(x) (e + \abs{ \nabla u_{n_k} ^+})}  \abs{ \nabla u_{n_k} ^+}^{q(x)} \dx
		- \into f(x, u_{n_k}^+) u_{n_k} ^+  \dx\\
		& \geq  \modoneHlogzero{ u_{n_k} ^+}
		- \into f(x, u_{n_k}^+)  u_{n_k} ^+  \dx.
	\end{align*}
	The weak convergence of \eqref{Eq:MinimizingSubsequence} along with the strong continuity of Lemma \ref{Le:Propf} \ref{Propf:ConvergenceFterm} yield $\modoneHlogzero{ u_{n_k} ^+} \to 0$, i.e., $ u_{n_k} ^+ \to 0$ in $\WHlogzero$ by Proposition \ref{Prop:HlogModularNorm} \textnormal{(v)}. However, as $\ph$ is continuous and by Proposition \ref{Prop:Infimum>0}
	\begin{align*}
		0 < \inf_{u \in \mathcal{N}} \ph(u)
		\leq \ph ( u_{n_k} ^+)
		\longrightarrow \ph (0) = 0 \quad \text{as } k \to \infty.
	\end{align*}
	This is a contradiction and finishes the proof of the Claim.

	From the previous Claim and Proposition \ref{Prop:NehariManifoldProps}, we have $t_1, t_2 > 0$ such that $t_1 z_1, t_2 z_2 \in \mathcal{N}$. We define $w_0 = t_1 z_1 + t_2 z_2$, which by \eqref{Eq:MinimizingSubsequence} satisfies that $w_0^+ = t_1 z_1$ and $- w_0^- = t_2 z_2$, thus $w_0 \in \mathcal{N}_0$. In order to see that this is the minimizer, we note that $\ph$ is sequentially weakly lower semicontinuous. Indeed, the $F$ term is even strongly continuous by Lemma \ref{Le:Propf} \ref{Propf:ConvergenceFterm}, the part with exponent $p(\cdot)$ is weakly lower semicontinuous because it is continuous and convex, and the part with exponent $q(\cdot)$ and logarithm is also weakly lower semicontinuous for the same reasons. If we combine Proposition \ref{Prop:NehariManifoldProps} with the previous property, we get
	\begin{align*}
		\inf_{u \in \mathcal{N}_0} \ph(u)
		= \lim\limits_{k \to \infty} \ph(u_{n_k})
		& = \lim\limits_{k \to \infty} \ph(u_{n_k}^+) + \ph(- u_{n_k} ^-) \\
		& \geq \liminf\limits_{k \to \infty} \ph(t_1 u_{n_k}^+) + \ph(- t_2 u_{n_k} ^-) \\
		& \geq \ph(t_1 z_1) + \ph(t_2 z_2) \\
		& = \ph(w_0^+) + \ph(- w_0^-)
		= \ph(w_0)
		\geq \inf_{u \in \mathcal{N}_0} \ph(u).
	\end{align*}
\end{proof}

Our last step towards finding a sign-changing solution is to see that any minimizer like the one we obtained in the previous proposition is actually a critical point of $\ph$.

\begin{proposition}
	\label{Prop:Sing-changing-CritPoint}
	Let \eqref{Assump:H3} be satisfied and $f$ fulfill \eqref{Asf:Caratheodory}, \eqref{Asf:WellDef}, \eqref{Asf:QuotientMono} and  \eqref{Asf:GrowthZeroAlt}. Let $w_0 \in \mathcal{N}_0$ such that $\ph(w_0) = \inf_{u \in \mathcal{N}_0} \ph(u)$. Then $w_0$ is a critical point of $\ph$.
\end{proposition}

\begin{proof}
	This is a proof by contradiction. Assume that $\ph'(w_0) \neq 0$, one can find $\lambda, \delta_0 > 0$ such that
	\begin{align*}
		\norm{ \ph'(u) }_{*} \geq \lambda, \quad
		\text{for all } u \in \WHlogzero \text{ with }
		\normoneHlogzero{u - w_0} < 3 \delta_0.
	\end{align*}
	On the other hand, let $C_{\hlog}$ be the constant of the embedding $\WHlogzero \hookrightarrow \Lp{p_-}$ given by Proposition \ref{Prop:EmbeddingHlogSobolev} \textnormal{(i)} and the usual embedding $\Wpzero{p_-} \hookrightarrow \Lp{p_-}$. As $w_0^+ \neq 0 \neq w_0^-$, for any $v \in \WHlogzero$ we know that
	\begin{align*}
		\normoneHlogzero{ w_0 - v }
		\geq  C_{\hlog}^{-1} \norm{ w_0 - v}_{p_-} \geq
		\begin{cases}
			C_{\hlog}^{-1} \norm{ w_0^- }_{p_-}, & \text{if } v^- = 0, \\
			C_{\hlog}^{-1} \norm{ w_0^+ }_{p_-}, & \text{if } v^+ = 0.
		\end{cases}
	\end{align*}
	We can choose some value $0 < \delta_1 < \min \{ C_{\hlog}^{-1} \norm{ w_0^- }_{p_-} , C_{\hlog}^{-1} \norm{ w_0^+ }_{p_-} \}$. As a consequence, for any $v \in \WHlogzero$ with $\normoneHlogzero{ w_0 - v } < \delta_1$ we know that $v^+ \neq 0 \neq v^-$. For the rest of the proof we work with $\delta = \min \{ \delta_0, \delta_1 / 2\}$.

	Observe that the mapping $(s,t) \mapsto s w_0^+ - t w_0^-$ is a continuous mapping $[0,\infty)^2 \to \WHlogzero$. Hence, we can find some $0 < \alpha < 1$ such that for all $s,t \geq 0$ with $\max \{ \abs{s - 1}, \abs{t - 1} \} < \alpha$ it holds
	\begin{align*}
		\normoneHlogzero{ s w_0^+ - t w_0^- - w_0 } < \delta.
	\end{align*}
	Let $D = ( 1 - \alpha, 1 + \alpha)^2$. Note also that for any $s,t \geq 0$ with $s \neq 1 \neq t$, by Proposition \ref{Prop:NehariManifoldProps} we know
	\begin{equation}
		\label{Eq:ParametersEstimatedByInfimum}
		\begin{aligned}
			\ph(s w_0^+ - t w_0^-)
			 & = \ph(s w_0^+) + \ph(- t w_0^-) \\
			 & < \ph(w_0^+) + \ph(- w_0^-)
			= \ph(w_0)
			= \inf_{u \in \mathcal{N}_0} \ph(u).
		\end{aligned}
	\end{equation}
	In particular, this implies that
	\begin{align*}
		\zeta = \max_{(s,t) \in \partial D} \ph(s w_0^+ - t w_0^-)
		< \ph(w_0)
		= \inf_{u \in \mathcal{N}_0} \ph (u).
	\end{align*}

	At this point we are in the position to apply the quantitative deformation lemma given in Lemma \ref{Le:DeformationLemma}. With its notation, take
	\begin{align*}
		S = B (w_0, \delta), \quad c = \inf_{u \in \mathcal{N}_0} \ph(u), \quad \eps = \min \left\lbrace \frac{c - \zeta}{4}, \frac{\lambda \delta}{8} \right\rbrace , \quad \delta \text{ be as defined above.}
	\end{align*}
	The assumptions are satisfied because $S_{2 \delta} = B (w_0, 3 \delta)$ and the choice of $\eps$, so we know that there is a mapping $\eta$ with the properties stated in the lemma. Because of the choice of $\eps$ we also know that
	\begin{align}
		\label{Eq:2epsAtBoundary}
		\ph(s w_0^+ - t w_0^-)
		\leq \zeta + c - c
		< c - \left( \frac{c - \zeta}{2} \right)
		\leq c - 2 \eps
	\end{align}
	for all $(s,t) \in \partial D$. Let us define $h \colon [0,\infty)^2 \to \WHlogzero$ and $H \colon (0,\infty)^2 \to \R^2$ by
	\begin{align*}
		h(s,t)  & = \eta(1 , s w_0^+ - t w_0^-)                                                                                                        \\
		H (s,t) & = \left( \; \frac{1}{s} \langle \ph'(h^+(s,t)) , h^+(s,t) \rangle \; , \; \frac{1}{t} \langle \ph'(- h^- (s,t)) , -h^- (s,t) \rangle \; \right).
	\end{align*}
	As $\eta$ is continuous, so is $h$, and as $\ph$ is $C^1$, $H$ is also continuous. Because of Lemma \ref{Le:DeformationLemma} \textnormal{(i)} and \eqref{Eq:2epsAtBoundary}, for all $(s,t) \in \partial D$ we know that $h(s,t) = s w_0^+ - t w_0^-$ and
	\begin{align*}
		H (s,t) = \left( \; \langle \ph'(s w_0^+) , w_0^+ \rangle \; , \; \langle \ph'(- t w_0^-) , - w_0^- \rangle \; \right).
	\end{align*}
	Furthermore, if we also take into consideration the information on the derivatives from Proposition \ref{Prop:NehariManifoldProps} we have the componentwise inequalities
	\begin{align*}
		& H_1(1 - \alpha,t) > 0 > H_1 (1 + \alpha,t),\\
		& H_2 (t, 1 - \alpha) > 0 > H_2 (t,1 + \alpha) \quad
		\text{for all } t \in [1 - \alpha, 1 + \alpha],
	\end{align*}
	where $H=(H_1,H_2)$. With the information above, by the Poincar\'e-Miranda existence theorem given in Theorem \ref{Th:PoincareMiranda} applied to $d(s,t) = - H(1 + s, 1 + t)$, we find $(s_0,t_0) \in D$ such that $H(s_0,t_0) = 0$, or equivalently
	\begin{align*}
		\langle \ph'(h^+(s_0 , t_0)) , h^+(s_0 , t_0) \rangle = 0 = \langle \ph'(- h^- (s_0 , t_0)) , -h^- (s_0 , t_0) \rangle.
	\end{align*}
	From Lemma \ref{Le:DeformationLemma} \textnormal{(iv)} and the choice of $\alpha$, we also know that
	\begin{align*}
		\normoneHlogzero{h(s_0 , t_0) - w_0} \leq 2 \delta \leq \delta_1,
	\end{align*}
	which by the choice of $\delta_1$ gives us
	\begin{align*}
		h^+(s_0 , t_0) \neq 0 \neq - h^-(s_0 , t_0).
	\end{align*}
	Altogether, this means that $h(s_0 , t_0) \in \mathcal{N}_0$. However, by Lemma \ref{Le:DeformationLemma} \textnormal{(ii)}, the choice of $\alpha$ and \eqref{Eq:ParametersEstimatedByInfimum}, we also know that $\ph( h(s_0 , t_0) ) \leq c - \eps$, which is a contradiction and this finishes the proof.
\end{proof}

The combination of Propositions \ref{Prop:MinimizerExistence} and \ref{Prop:Sing-changing-CritPoint}, Lemma \ref{Le:Propq+1} \ref{Propf:growthq+1} and Theorem \ref{Th:ConstantSignSolutionsAlt} yield the following results.

\begin{theorem}
	\label{Th:SignChangingSolution}
	Let \eqref{Assump:H3} be satisfied and $f$ fulfill  \eqref{Asf:Caratheodory}, \eqref{Asf:WellDef}, \eqref{Asf:QuotientMono} and  \eqref{Asf:GrowthZeroAlt}. Then there exists a nontrivial weak solution $w_0 \in \WHlogzero$ of problem \eqref{Eq:Problem} with changing sign.
\end{theorem}

\begin{theorem}
	\label{Th:ThreeSolutions}
	Let \eqref{Assump:H3} be satisfied and $f$ fulfill  \eqref{Asf:Caratheodory}, \eqref{Asf:WellDef}, \eqref{Asf:QuotientMono},  \eqref{Asf:GrowthZeroAlt} and \eqref{Asf:CeramiAssumption}. Then there exist nontrivial weak solutions $u_0,v_0,w_0 \in \WHlogzero$ of problem \eqref{Eq:Problem} such that $u_0 \geq 0$, $v_0 \leq 0$ and $w_0$ has changing sign.
\end{theorem}

\section{Nodal domains} \label{nodal-domains}

In this last section, we provide more information on the sign-changing solution found in Section \ref{sign-changing-solution}. In particular we will determine the number of nodal domains, that is, the number of maximal regions where it does not change sign. The usual definition of nodal domains for $u \in C(\Omega,\R)$ is the connected components of $\Omega \setminus Z$, where the set $Z = \{ x \in \Omega \colon u(x) = 0 \}$ is called the nodal set of $u$. In our case, we have no information on the continuity of the solutions, so this definition is not meaningful for us. A similar situation was already noted in a previous work by Crespo-Blanco--Winkert \cite{Crespo-Blanco-Winkert-2022}, where the following definition was proposed.

Let $u \in \WHlogzero$ and $A$ be a Borelian subset of $\Omega$ with $|A| > 0$. We say that $A$ is a nodal domain of $u$ if
\begin{enumerate}
	\item[\textnormal{(i)}]
		$u \geq 0$ a.e.\,on $A$ or $u \leq 0$ a.e.\,on $A$;
	\item[\textnormal{(ii)}]
		$0 \neq u 1_A \in \WHlogzero$;
	\item[\textnormal{(iii)}]
		$A$ is minimal w.r.t.\,\textnormal{(i)} and \textnormal{(ii)}, i.e., if $B \subseteq A$ with $B$ being a Borelian subset of $\Omega$, $|B| > 0$ and $B$ satisfies \textnormal{(i)} and \textnormal{(ii)}, then $|A \setminus B| = 0$.
\end{enumerate}

For the purpose of determining the number of nodal domains, we need one last assumption on our right-hand side.

\begin{enumerate} [label=\textnormal{(f$_{\arabic*}$)},ref=\textnormal{f$_{\arabic*}$}]
	\setcounter{enumi}{5}
	\item\label{Asf:Nonnegative}
	$f(x,t)t - q_+ \left( 1 + \frac{\kappa}{q_-} \right) F(x,t) \geq 0$ for all $t \in \R$ and for a.a.\,$x \in \Omega$.
\end{enumerate}

\begin{proposition}
	\label{Prop:NodalDomains}
	Let \eqref{Assump:H3} be satisfied and $f$ fulfill \eqref{Asf:Caratheodory}, \eqref{Asf:WellDef}, \eqref{Asf:QuotientMono}, \eqref{Asf:GrowthZeroAlt} and \eqref{Asf:Nonnegative}. Then, any minimizer of $\ph_{|_{\mathcal{N}_0}}$ has exactly two nodal domains.
\end{proposition}

\begin{proof}
	Let $w_0$ be such that $\ph(w_0) = \inf_{u \in \mathcal{N}_0} \ph(u)$. The sets $\Omega_{w_0 > 0} = \{ x \in \Omega \colon  w_0 (x) > 0 \}$ and $\Omega_{w_0 < 0} = \{ x \in \Omega \colon  w_0 (x) < 0 \}$ are determined up to a zero measure set (for any choice of a representative of the class $w_0$ they may differ, but at most in a zero measure set), which is no problem for the definition above. Furthermore, by Proposition \ref{Prop:Truncations} and $w_0^+ = w_0 1_{\Omega_{w_0 > 0}}$, $-w_0^- = w_0 1_{\Omega_{w_0 < 0}}$, conditions \textnormal{(i)} and \textnormal{(ii)} hold for these sets. So it is only left to see their minimality.

	We proceed by contradiction. Without loss of generality, let $B_1, B_2$ be Borelian subsets of $\Omega$ with the following properties: disjoint sets with $\Omega_{w_0 < 0} = B_1 \dot{\cup} B_2$, both have positive measure and $B_1$ fulfills \textnormal{(i)} and \textnormal{(ii)} in the definition above. This implies that for $B_2$ conditions \textnormal{(i)} and \textnormal{(ii)} hold too, since $w_0 1_{B_2} = w_0 1_{\Omega_{w_0 < 0}} - w_0 1_{B_1} \in \WHlogzero$.

	Let $u = 1_{\Omega_{w_0 > 0}} w_0 + 1_{B_1} w_0$ and $v = 1_{B_2} w_0$, hence $u ^+= 1_{\Omega_{w_0 > 0}} w_0$ and $- u ^- = 1_{B_1} w_0$. By Proposition \ref{Prop:Sing-changing-CritPoint} we know that $\ph'(w_0) = 0$ and as the supports of $u^+ , - u^-$ and $v$ are disjoint, we get
	\begin{align*}
		0 = \langle \ph'(w_0) , u^+ \rangle = \langle \ph'(u^+) , u^+ \rangle,
	\end{align*}
	that is, $u^+ \in \mathcal{N}$. With a similar argument $- u^- \in \mathcal{N}$, so $u \in \mathcal{N}_0$; and also $\langle \ph'(v) , v \rangle = 0$.

	Altogether, by these properties, condition \eqref{Asf:Nonnegative} and Lemma \ref{Le:QuotientFracLog}, we get
	\begin{align*}
		\inf_{w \in \mathcal{N}_0} \ph(w)
		& = \ph (w_0)
		= \ph(u) + \ph(v) - \frac{1}{ q_+ \left( 1 + \frac{\kappa}{q_-} \right) } \langle \ph'(v) , v \rangle \\
		 & \geq \ph(u) + \left[ \frac{1}{p_+} - \frac{1}{ q_+ \left( 1 + \frac{\kappa}{q_-} \right) } \right] \varrho_{p(\cdot)} ( \nabla v ) \\
		& \quad + \into \left[ \frac{1}{ q_+ \left( 1 + \frac{\kappa}{q_-} \right) } f(x,v) v - F(x,v) \right] \dx  \\
		& \quad + \frac{1}{q_+} \into \mu(x) \abs{\nabla v}^{q(x)} \log(e + \abs{\nabla v}) \dx \\
		& \quad - \frac{1}{ q_+ \left( 1 + \frac{\kappa}{q_-} \right) } \into \mu(x) \abs{\nabla v}^{q(x)} \left[ \log (e + \abs{\nabla v}) + \frac{ \abs{\nabla v} }{ q(x) (e + \abs{\nabla v}) } \right] \dx \\
		& \geq \ph(u) + \left[ \frac{1}{p_+} - \frac{1}{ q_+ \left( 1 + \frac{\kappa}{q_-} \right) } \right] \varrho_{p(\cdot)} ( \nabla v )  \\
		& \geq \inf_{w \in \mathcal{N}_0} \ph(w) + \left[ \frac{1}{p_+} - \frac{1}{ q_+ \left( 1 + \frac{\kappa}{q_-} \right) } \right] \varrho_{p(\cdot)} ( \nabla v ).
	\end{align*}
	Since $p_+ \leq q_+$ and $v \neq 0$, we get the desired contradiction.
\end{proof}

Combining Theorems \ref{Th:SignChangingSolution} and \ref{Th:ThreeSolutions} with Proposition \ref{Prop:NodalDomains}, we obtain the final existence results of this work.

\begin{theorem}
	Let \eqref{Assump:H3} be satisfied and $f$ fulfill  \eqref{Asf:Caratheodory}, \eqref{Asf:WellDef}, \eqref{Asf:QuotientMono}, \eqref{Asf:GrowthZeroAlt} and \eqref{Asf:Nonnegative}. Then there exists a nontrivial weak solution $w_0 \in \WHlogzero$ of problem \eqref{Eq:Problem} with changing sign and two nodal domains.
\end{theorem}

\begin{theorem}
	Let \eqref{Assump:H3} be satisfied and $f$ fulfill  \eqref{Asf:Caratheodory}, \eqref{Asf:WellDef}, \eqref{Asf:QuotientMono},  \eqref{Asf:GrowthZeroAlt}, \eqref{Asf:CeramiAssumption} and \eqref{Asf:Nonnegative}. Then there exist nontrivial weak solutions $u_0,v_0,w_0 \in \WHlogzero$  of problem \eqref{Eq:Problem} such that $u_0 \geq 0$, $v_0 \leq 0$ and $w_0$ has changing sign with two nodal domains.
\end{theorem}

Let us finish this work with some examples of right-hand side functions that would fit in our assumptions.

\begin{example}
	For simplicity, assume not only \eqref{Assump:H3}, but also $q_+ \kappa / q_- < 1$, which implies $q_+ (1 + \kappa/q_-) < q_+ + 1 < p^*_-$.
	\begin{enumerate}[label=(\roman*),font=\normalfont]
		\item
			Let $0 < \eps < 1 - q_+ \kappa / q_-$ and take
			\begin{align*}
				f(x,t) = \abs{t}^{q_+ \left( 1 + \frac{\kappa}{q_-} \right) + \eps - 2 } t.
			\end{align*}
			This function satisfies \eqref{Asf:Caratheodory}, \eqref{Asf:WellDef}, \eqref{Asf:QuotientMono},  \eqref{Asf:GrowthZeroAlt}, \eqref{Asf:CeramiAssumption} and \eqref{Asf:Nonnegative}. For \eqref{Asf:WellDef} take $r = q_+ ( 1 + \kappa / q_- ) + \eps$ and for \eqref{Asf:CeramiAssumption} take $l = \widetilde{l} = q_+ ( 1 + \kappa / q_- )$.
		\item
			Let $l, \widetilde{l}, m \in C_+ (\Omega)$ such that $q_+ + 1 \leq \min\{ m_-, l_-, \widetilde{l}_- \}$, $\max\{ l_+, \widetilde{l}_+\} < p^*_-$ and
			\begin{align*}
				\frac{ \max\{ l_+, \widetilde{l}_+ \} }{p_-} - \frac{ \min\{ l_-, \widetilde{l}_- \} }{N} < 1.
			\end{align*}
			Then we can take
			\begin{align*}
				f(x,t) =
				\begin{cases}
					|t|^{\widetilde{l}(x)-2}t[1 + \log(-t)], & \text{if } \phantom{-1}t \leq -1, \\
					|t|^{m(x)-2}t,                           & \text{if } -1 < t < 1,            \\
					|t|^{l(x)-2}t[1 + \log(t)],              & \text{if } \phantom{-1} 1 \leq t,
				\end{cases}
			\end{align*}
			which also satisfies \eqref{Asf:Caratheodory}, \eqref{Asf:WellDef}, \eqref{Asf:QuotientMono},  \eqref{Asf:GrowthZeroAlt}, \eqref{Asf:CeramiAssumption} and \eqref{Asf:Nonnegative}. For \eqref{Asf:WellDef} take $r(x) = \max\{ l(x), \widetilde{l}(x) \} + \eps$, where $0 < \eps < p^*_- - \max\{ l_+, \widetilde{l}_+\}$ and is also small enough so that
			\begin{align*}
				\frac{r_+}{p_-} - \frac{ \min\{ l_-, \widetilde{l}_- \} }{N} < 1.
			\end{align*}
			For \eqref{Asf:CeramiAssumption} take $l$ and $\widetilde{l}$ to be the same ones as here. This inequality is why we need the compatibility assumption on $\max \{ l_+, \widetilde{l}_+ \}$ and $\min \{ l_-, \widetilde{l}_- \}$. In particular, when $l=\widetilde{l}$ and they are constant functions, this condition is exactly $l < p^*_-$, hence redundant in this case.
	\end{enumerate}
\end{example}

\section*{Acknowledgments}
The first author acknowledges the support of the Start-up Research Grant (SRG) SRG/2023/000308, Science and Engineering Research Board (SERB), India, and Seed grant IIT(BHU)/DMS/2023-24/493. The second author was funded by the Deutsche Forschungsgemeinschaft (DFG, German Research Foundation) under Germany's Excellence Strategy -
The Berlin Mathematics Research Center MATH+ and the Berlin Mathematical School (BMS) (EXC-2046/1, project ID: 390685689).

%
%
%
%
%



\begin{thebibliography}{99}	

\bibitem{Aberqi-Bennouna-Benslimane-Ragusa-2023}
	A. Aberqi, J. Bennouna, O. Benslimane, M.A. Ragusa,
	{\it Weak solvability of nonlinear elliptic equations involving variable exponents},
	Discrete Contin. Dyn. Syst. Ser. S {\bf 16} (2023), no. 6, 1142--1157.

\bibitem{Alves-daSilva-2023}
C.O. Alves, I.S. da Silva,
{\it Existence of a positive solution for a class of Schr\"odinger logarithmic equations on exterior domains},
Z. Angew. Math. Phys. 75 (2024), no. 3, Paper No. 77, 33 pp.

\bibitem{Alves-deMoraisFilho-2018}
C.O. Alves, D.C. de Morais Filho,
{\it Existence and concentration of positive solutions for a Schr\"{o}dinger logarithmic equation},
Z. Angew. Math. Phys. {\bf 69} (2018), no. 6, Paper No. 144, 22 pp.

\bibitem{Alves-Ji-2020}
C.O. Alves, C. Ji,
{\it Existence and concentration of positive solutions for a
logarithmic Schr\"{o}dinger equation via penalization method},
Calc. Var. Partial Differential Equations {\bf 59} (2020), no. 1, Paper No. 21, 27 pp.

\bibitem{Alves-Moussaoui-Tavares-2019}
C.O. Alves, A. Moussaoui, L.  Tavares,
{\it An elliptic system with logarithmic nonlinearity},
Adv. Nonlinear Anal. {\bf 8} (2019), no. 1, 928--945.

\bibitem{Arora-Shmarev-2023}
R. Arora, S. Shmarev,
{\it Double-phase parabolic equations with variable growth and nonlinear sources},
Adv. Nonlinear Anal. {\bf 12} (2023), no. 1, 304--335.

\bibitem{Arora-Shmarev-2023-b}
R. Arora, S. Shmarev,
{\it Existence and regularity results for a class of parabolic problems with double phase flux of variable growth},
Rev. R. Acad. Cienc. Exactas F\'{\i}s. Nat. Ser. A Mat. RACSAM {\bf 117} (2023), no. 1, Paper No. 34, 48 pp.

\bibitem{Bahrouni-Radulescu-Repovs-2019}
A. Bahrouni, V.D. R\u{a}dulescu, D.D. Repov\v{s},
{\it Double phase transonic flow problems with variable growth: nonlinear patterns and stationary waves},
Nonlinearity {\bf 32} (2019), no. 7, 2481--2495.

\bibitem{Bai-Papageorgiou-Zeng-2023}
Y. Bai, N.S. Papageorgiou, S. Zeng,
{\it Parametric singular double phase Dirichlet problems},
Adv. Nonlinear Anal. {\bf 12} (2023), no. 1, Paper No. 20230122, 20 pp.

\bibitem{Baroni-Colombo-Mingione-2015}
P. Baroni, M. Colombo, G. Mingione,
{\it Harnack inequalities for double phase functionals},
Nonlinear Anal. {\bf 121} (2015), 206--222.

\bibitem{Baroni-Colombo-Mingione-2016}
P. Baroni, M. Colombo, G. Mingione,
{\it Non-autonomous functionals, borderline cases and related function classes},
St. Petersburg Math. J. {\bf 27} (2016), 347--379.

\bibitem{Baroni-Kuusi-Mingione-2015}
P. Baroni, T. Kuusi, G. Mingione,
{\it Borderline gradient continuity of minima},
J. Fixed Point Theory Appl. {\bf 15} (2014), no. 2, 537--575.

\bibitem{Baroni-Colombo-Mingione-2018}
P. Baroni, M. Colombo, G. Mingione,
{\it Regularity for general functionals with double phase},
Calc. Var. Partial Differential Equations {\bf 57} (2018), no. 2, Art. 62, 48 pp.

\bibitem{Beck-Mingione-2020}
L. Beck, G. Mingione,
{\it Lipschitz bounds and nonuniform ellipticity},
Comm. Pure Appl. Math. {\bf 73} (2020), no. 5,  944--1034.

\bibitem{Benci-DAvenia-Fortunato-Pisani-2000}
V. Benci, P. D'Avenia, D. Fortunato, L. Pisani,
{\it Solitons in several space dimensions: {D}errick's problem and  infinitely many solutions},
Arch. Ration. Mech. Anal. {\bf 154} (2000), no. 4, 297--324.

\bibitem{Byun-Oh-2020}
S.-S. Byun, J. Oh,
{\it Regularity results for generalized double phase functionals},
Anal. PDE {\bf 13} (2020), no. 5,  1269--1300.

\bibitem{Byun-Ok-Song-2022}
S.-S. Byun, J. Ok, K. Song,
{\it H\"{o}lder regularity for weak solutions to nonlocal double phase problems},
J. Math. Pures Appl. (9) {\bf 168} (2022), 110--142.

\bibitem{Cherfils-Ilyasov-2005}
L. Cherfils, Y. Il'yasov,
{\it On the stationary solutions of generalized reaction diffusion  equations with {$p\&q$}-{L}aplacian},
Commun. Pure Appl. Anal. {\bf 4} (2005), no. 1, 9--22.

\bibitem{Clop-Gentile-Passarelli-2023}
A. Clop, A. Gentile, A. Passarelli di Napoli,
{\it Higher differentiability results for solutions to a class of non-homogeneous elliptic problems under sub-quadratic growth conditions},
Bull. Math. Sci. {\bf 13} (2023), no. 2, Paper No. 2350008, 55 pp.

\bibitem{Colasuonno-Squassina-2016}
F. Colasuonno, M. Squassina,
{\it Eigenvalues for double phase variational integrals},
Ann. Mat. Pura Appl. (4) {\bf 195} (2016), no. 6, 1917--1959.

\bibitem{Crespo-Blanco-Gasinski-Harjulehto-Winkert-2022}
\'{A}. Crespo-Blanco, L. Gasi\'{n}ski, P. Harjulehto, P. Winkert,
{\it A new class of double phase variable exponent problems: existence and uniqueness},
J. Differential Equations {\bf 323} (2022), 182--228.

\bibitem{Crespo-Blanco-Winkert-2022}
\'{A}. Crespo-Blanco, P. Winkert,
{\it Nehari manifold approach for superlinear double phase problems with variable exponents},
Ann. Mat. Pura Appl. (4) {\bf 203} (2024), no. 2, 605--634.

\bibitem{Colombo-Mingione-2015a}
M. Colombo, G. Mingione,
{\it Bounded minimisers of double phase variational integrals},
Arch. Ration. Mech. Anal. {\bf 218} (2015), no. 1, 219--273.

\bibitem{Colombo-Mingione-2015b}
M. Colombo, G. Mingione,
{\it Regularity for double phase variational problems},
Arch. Ration. Mech. Anal. {\bf 215} (2015), no. 2, 443--496.

\bibitem{Cupini-Marcellini-Mascolo-2023}
G. Cupini, P. Marcellini, E. Mascolo,
{\it Local boundedness of weak solutions to elliptic equations with $p,q$-growth},
Math. Eng. {\bf 5} (2023), no. 3, Paper No. 065, 28 pp.

\bibitem{DeFilippis-Mingione-2021-2}
C. De Filippis, G. Mingione,
{\it Lipschitz bounds and nonautonomous integrals},
Arch. Ration. Mech. Anal. {\bf 242} (2021), 973--1057.

\bibitem{DeFilippis-Mingione-2020-2}
C. De Filippis, G. Mingione,
{\it On the regularity of minima of non-autonomous functionals},
J. Geom. Anal. {\bf 30} (2020), no. 2, 1584--1626.

\bibitem{DeFilippis-Mingione-2023}
C. De Filippis, G. Mingione,
{\it Regularity for double phase problems at nearly linear growth},
Arch. Ration. Mech. Anal. {\bf 247} (2023), no. 5, 85.

\bibitem{De-Filippis-Palatucci-2019}
C. De Filippis, G. Palatucci,
{\it H\"{o}lder regularity for nonlocal double phase equations},
J. Differential Equations {\bf 267} (2019), no. 1,  547--586.

\bibitem{Dinca-Mawhin-2021}
G. Dinca, J. Mawhin,
``Brouwer Degree -- The Core of Nonlinear Analysis'',
Birkh\"{a}user/ Springer, Cham, 2021.

\bibitem{Diening-Harjulehto-Hasto-Ruzicka-2011}
L. Diening, P. Harjulehto, P. H\"{a}st\"{o}, M. R$\mathring{\text{u}}$\v{z}i\v{c}ka,
``Lebesgue and Sobolev Spaces with Variable Exponents'',
Springer, Heidelberg, 2011.

\bibitem{Fan-Shen-Zhao-2001}
X. Fan, J. Shen, D. Zhao,
{\it Sobolev embedding theorems for spaces $W^{k,p(x)}(\Omega)$},
J. Math. Anal. Appl. {\bf 262} (2001), no. 2, 749--760.

\bibitem{Fan-Zhang-Zhao-2005}
X. Fan, Q. Zhang, D. Zhao,
{\it Eigenvalues of $p(x)$-Laplacian Dirichlet problem},
J. Math. Anal. Appl. {\bf 302} (2005), no. 2, 306--317.

\bibitem{Fan-Zhao-2001}
X. Fan, D. Zhao,
{\it On the spaces $L^{p(x)}(\Omega)$ and $W^{m,p(x)}(\Omega)$},
J. Math. Anal. Appl. {\bf 263} (2001), no. 2, 424--446.

\bibitem{Fang-Radulescu-Zhang-2024}
Y. Fang, V.D. R\u{a}dulescu, C. Zhang,
{\it Equivalence of weak and viscosity solutions for the nonhomogeneous double phase equation},
Math. Ann. {\bf 388} (2024), no. 3, 2519--2559.

\bibitem{Farkas-Winkert-2021}
C. Farkas, P. Winkert,
{\it An existence result for singular Finsler double phase problems},
J. Differential Equations {\bf 286} (2021), 455--473.

\bibitem{Fuchs-Mingione-2000}
M. Fuchs, G. Mingione,
{\it Full {$C^{1,\alpha}$}-regularity for free and constrained local minimizers of elliptic variational integrals with nearly linear growth},
Manuscripta Math. {\bf 102} (2000), no. 2, 227--250.

\bibitem{Fuchs-Seregin-2000}
M. Fuchs, G. Seregin,
``Variational Methods for Problems from Plasticity Theory and for Generalized Newtonian Fluids'',
Springer-Verlag, Berlin, 2000.

\bibitem{Gao-Jiang-Liu-Wei-2023}
Y. Gao, Y. Jiang, L. Liu, N. Wei,
{\it Multiple positive solutions for a logarithmic {K}irchhoff type
problem in $\mathbb{R}^3$},
Appl. Math. Lett. {\bf 139} (2023), Paper No. 108539, 6 pp.

\bibitem{Gasinski-Papageorgiou-2019}
L. Gasi\'nski, N.S. Papageorgiou,
{\it Constant sign and nodal solutions for superlinear double phase problems},
Adv. Calc. Var. {\bf 14} (2021), no. 4, 613–626.

\bibitem{Gasinski-Winkert-2020}
L. Gasi\'nski, P. Winkert,
{\it Existence and uniqueness results for double phase problems with convection term},
J. Differential Equations {\bf 268} (2020), no. 8, 4183--4193.

\bibitem{Hardy-Littlewood-Polya-1934}
G.H. Hardy, J.E. Littlewood, G. P\'{o}lya,
``Inequalities'',
Cambridge University Press, 1934.

\bibitem{Harjulehto-Hasto-2019}
P. Harjulehto, P. H\"{a}st\"{o},
``Orlicz Spaces and Generalized Orlicz Spaces'',
Springer, Cham, 2019.

\bibitem{Harjulehto-Hasto-Toivanen-2017}
P. Harjulehto, P. H\"{a}st\"{o}, O. Toivanen,
{\it H\"{o}lder regularity of quasiminimizers under generalized growth conditions},
Calc. Var. Partial Differential Equations {\bf 56} (2017), no. 2, Paper No. 22, 26pp.

\bibitem{Hasto-Ok-2019}
P. H\"{a}st\"{o}, J. Ok,
{\it Maximal regularity for local minimizers of non-autonomous functionals},
J. Eur. Math. Soc. {\bf 24} (2022), no. 4, 1285--1334.

\bibitem{Heinonen-Kilpelainen-Martio-2006}
J. Heinonen, T. Kilpel\"{a}inen, O. Martio,
``Nonlinear Potential Theory of Degenerate Elliptic Equations'',
Dover Publications, Inc., Mineola, NY, 2006.

\bibitem{Ho-Kim-Winkert-Zhang-2022}
K. Ho, Y.-H. Kim, P. Winkert, C. Zhang,
{\it The boundedness and H\"{o}lder continuity of weak solutions to elliptic equations involving variable exponents and critical growth},
J. Differential Equations {\bf 313} (2022), 503--532.

\bibitem{Ho-Winkert-2023}
K. Ho, P. Winkert,
{\it New embedding results for double phase problems with variable exponents and a priori bounds for corresponding generalized double phase problems},
Calc. Var. Partial Differential Equations {\bf 62} (2023), no. 8, Paper No. 227, 38 pp.

\bibitem{Lindqvist-2019}
P. Lindqvist,
``Notes on the Stationary $p$-Laplace Equation'',
Springer, Cham, 2019.

\bibitem{Liu-Dai-2018}
W. Liu, G. Dai,
{\it Existence and multiplicity results for double phase problem},
J. Differential Equations {\bf 265} (2018), no. 9, 4311--4334.

\bibitem{Liu-Pucci-2023}
J. Liu, P. Pucci,
{\it Existence of solutions for a double-phase variable exponent equation without the Ambrosetti-Rabinowitz condition},
Adv. Nonlinear Anal. {\bf 12} (2023), no. 1, Paper No. 20220292, 18 pp.

\bibitem{Marcellini-2023}
P. Marcellini,
{\it  Local Lipschitz continuity for $p,q$-PDEs with explicit $u$-dependence},
Nonlinear Anal. {\bf 226} (2023), Paper No. 113066, 26 pp.

\bibitem{Marcellini-1991}
P. Marcellini,
{\it Regularity and existence of solutions of elliptic equations with $p,q$-growth conditions},
J. Differential Equations {\bf 90} (1991), no. 1, 1--30.

\bibitem{Marcellini-1989}
P. Marcellini,
{\it Regularity of minimizers of integrals of the calculus of variations with nonstandard growth conditions},
Arch. Rational Mech. Anal. {\bf 105} (1989), no. 3,  267--284.

\bibitem{Marcellini-Papi-2006}
P. Marcellini, G. Papi,
{\it Nonlinear elliptic systems with general growth},
J. Differential Equations {\bf 221} (2006), no. 2, 412--443.

\bibitem{Montenegro-deQueiroz}
M. Montenegro, O.S. de Queiroz,
{\it Existence and regularity to an elliptic equation with logarithmic nonlinearity},
J. Differential Equations {\bf 246} (2009), no. 2, 482--511.

\bibitem{Nehari-1960}
Z. Nehari,
{\it On a class of nonlinear second-order differential equations},
Trans. Am. Math. Soc. {\bf 95} (1960) 101-123.

\bibitem{Nehari-1961}
Z. Nehari,
{\it Characteristic values associated with a class of non-linear second-order differential equations},
Acta Math. {\bf 105} (1961) 141-175.

\bibitem{Nekvinda-2004}
A. Nekvinda,
{\it Hardy-Littlewood maximal operator on $L^{p(x)} (\mathbb{R})$},
Math. Inequal. Appl. {\bf 7} (2004), no. 2, 255--265.

\bibitem{Ok-2018}
J. Ok,
{\it Partial regularity for general systems of double phase type with continuous coefficients},
Nonlinear Anal. {\bf 177} (2018), 673--698.

\bibitem{Ok-2020}
J. Ok,
{\it Regularity for double phase problems under additional integrability assumptions},
Nonlinear Anal. {\bf 194} (2020), 111408.

\bibitem{Papageorgiou-Radulescu-Repovs-2019-a}
N.S. Papageorgiou, V.D. R\u{a}dulescu, D.D. Repov\v{s},
{\it Double-phase problems and a discontinuity property of the spectrum},
Proc. Amer. Math. Soc. {\bf 147} (2019), no. 7, 2899--2910.

\bibitem{Papageorgiou-Radulescu-Repovs-2020}
N.S. Papageorgiou, V.D. R\u{a}dulescu, D.D. Repov\v{s},
{\it Ground state and nodal solutions for a class of double phase problems},
Z. Angew. Math. Phys. {\bf 71} (2020), no. 1, Paper No. 15, 15 pp.

\bibitem{Papageorgiou-Radulescu-Repovs-2019a}
N.S. Papageorgiou, V.D. R\u{a}dulescu, D.D. Repov\v{s},
``Nonlinear Analysis -- Theory and Methods'',
Springer, Cham, 2019.

\bibitem{Papageorgiou-Winkert-2018}
N.S.  Papageorgiou, P. Winkert,
``Applied Nonlinear Functional Analysis'',
De Gruyter, Berlin, 2018.

\bibitem{Perera-Squassina-2018}
K. Perera, M. Squassina,
{\it Existence results for double-phase problems via Morse theory},
Commun. Contemp. Math. {\bf 20} (2018), no. 2, 1750023, 14 pp.

\bibitem{Ragusa-Tachikawa-2020}
M.A. Ragusa, A. Tachikawa,
{\it Regularity for minimizers for functionals of double phase with variable exponents},
Adv. Nonlinear Anal. {\bf 9} (2020), no. 1, 710--728.

\bibitem{Ragusa-Tachikawa-2024}
M.A. Ragusa, A. Tachikawa,
{\it Regularity of minimizers for double phase functionals of borderline case with variable exponents},
Adv. Nonlinear Anal. {\bf 13} (2024), no. 1, Paper No. 20240017, 27 pp.

\bibitem{Seregin-Frehse-1999}
G.A. Seregin, J. Frehse,
{\it Regularity of solutions to variational problems of the deformation theory of plasticity with logarithmic hardening},
in: ``Proceedings of the St. Petersburg Mathematical Society, Vol. {V}'',
Amer. Math. Soc., Providence, RI {\bf 193}, 1999, 127--152.

\bibitem{Shi-Radulescu-Repovs-Zhang-2020}
X. Shi, V.D. R\u{a}dulescu, D.D. Repov\v{s}, Q. Zhang,
{\it Multiple solutions of double phase variational problems with variable exponent},
Adv. Calc. Var. {\bf 13} (2020), no. 4, 385--401.

\bibitem{Shuai-2019}
W. Shuai,
{\it Multiple solutions for logarithmic {S}chr\"{o}dinger equations},
Nonlinearity {\bf 32} (2019), no. 6, 2201--2225.

\bibitem{Shuai-2023}
W. Shuai,
{\it Two sequences of solutions for the semilinear elliptic equations with logarithmic nonlinearities},
J. Differential Equations {\bf 343} (2023), 263--284.

\bibitem{Simon-1978}
J. Simon,
{\it R\'{e}gularit\'{e} de la solution d'une \'{e}quation non lin\'{e}aire dans $\mathbb{R}^{N}$},
Journ\'{e}es d'Analyse Non Lin\'{e}aire (Proc. Conf. Besan\c{c}on, 1977), Springer, Berlin {\bf 665} (1978), 205--227.

\bibitem{Squassina-Szulkin-2015}
M. Squassina, A. Szulkin,
{\it Multiple solutions to logarithmic {S}chr\"{o}dinger equations with
periodic potential},
Calc. Var. Partial Differential Equations {\bf 54} (2015), no. 1, 585--597.

\bibitem{Szulkin-Weth-2010}
A. Szulkin, T. Weth,
{\it The method of Nehari manifold},
in: ``Handbook of Nonconvex Analysis and Applications'',
Int. Press, Somerville, MA, 2010, 597--632.

\bibitem{Tachikawa-2024}
	A. Tachikawa,
	{\it Partial regularity of minimizers for double phase functionals with variable exponents},
	NoDEA Nonlinear Differential Equations Appl. {\bf 31} (2024), no. 2, Paper No. 24, 26 pp.

\bibitem{Willem-1996}
M. Willem,
``Minimax Theorems'',
Birkh\"{a}user Boston, Inc., Boston, MA, 1996.

\bibitem{Zeidler-1990}
E. Zeidler,
``Nonlinear Functional Analysis and its Applications. II/B'',
Springer-Verlag, New York, 1990.

\bibitem{Zeng-Bai-Gasinski-Winkert-2020}
S. Zeng, Y. Bai, L. Gasi\'{n}ski, P. Winkert,
{\it Existence results for double phase implicit obstacle problems involving multivalued operators},
Calc. Var. Partial Differential Equations {\bf 59} (2020), no. 5, 176.

\bibitem{Zeng-Radulescu-Winkert-2022}
S. Zeng, V.D. R\u{a}dulescu, P. Winkert,
{\it Double phase implicit obstacle problems with convection and multivalued mixed boundary value conditions},
SIAM J. Math. Anal. {\bf 54} (2022), 1898--1926.

\bibitem{Zhang-Radulescu-2018}
Q. Zhang, V.D. R\u{a}dulescu,
{\it Double phase anisotropic variational problems and combined effects of 	reaction and absorption terms},
J. Math. Pures Appl. (9) {\bf 118} (2018),  159--203.

\bibitem{Zhikov-1986}
V.V. Zhikov,
{\it Averaging of functionals of the calculus of variations and elasticity theory},
Izv. Akad. Nauk SSSR Ser. Mat. {\bf 50} (1986), no. 4, 675--710.

\bibitem{Zhikov-1995}
V.V. Zhikov,
{\it On Lavrentiev's phenomenon},
Russian J. Math. Phys. {\bf 3} (1995), no. 2, 249--269.

\bibitem{Zhikov-2011}
V.V. Zhikov,
{\it On variational problems and nonlinear elliptic equations with nonstandard growth conditions},
J. Math. Sci. {\bf 173} (2011), no. 5, 463--570.


\end{thebibliography}
\end{document}